\documentclass[12pt]{amsart}
\usepackage{geometry}    
\usepackage{amsmath, amscd} 
\usepackage{amssymb} 
\usepackage{xcolor} 

\usepackage{pst-node}
\usepackage{tikz-cd} 
\usepackage{graphicx} 
\usepackage{comment}
\usepackage[parfill]{parskip}    
\usepackage{hyperref} 
\usepackage{fancyhdr}

\usepackage{tikz-cd}
\usepackage{subcaption}
\usepackage{amsthm}
\usepackage{mathtools}
\usepackage{xfrac}
\usepackage{stmaryrd} 
\usepackage{wrapfig} 
\usepackage{epigraph} 
\usepackage{ wasysym } 
\usepackage{mathrsfs} 

\newcommand{\F}{\mathbb{F} }
\newcommand{\R}{\mathbb{R} }
\newcommand{\Q}{\mathbb{Q} }
\newcommand{\N}{\mathbb{N} }
\newcommand{\Z}{\mathbb{Z} }

\newcommand{\C}{\mathbb{C} }

\newcommand{\D}{\mathbb{D} }

\newcommand{\Qrc}{\overline{\Q}^{\operatorname{rc}}}
\newcommand{\K}{\mathbb{K}}

\newcommand{\frakg}{\mathfrak{g}}
\newcommand{\frakp}{\mathfrak{p}}
\newcommand{\frakk}{\mathfrak{k}}
\newcommand{\fraka}{\mathfrak{a}}
\newcommand{\frakap}{\overline{\mathfrak{a}}^+}
\newcommand{\Ap}{A^+}
\newcommand{\ApF}{A_\F^+}
\newcommand{\ApR}{A_\R^+}
\newcommand{\frakn}{\mathfrak{n}}
\newcommand{\frakh}{\mathfrak{h}}
\newcommand{\frakt}{\mathfrak{t}}

\newcommand{\frakl}{\mathfrak{l}}
\newcommand{\eL}{\mathcal{L}}

\newcommand{\SLnR}{\operatorname{SL}_n(\R) }

\newcommand{\GLnB}{\operatorname{GL}_n(\mathbb{B}) }
\newcommand{\GLnD}{\operatorname{GL}_n(\D) }

\newcommand{\SL}{\operatorname{SL} }

\newcommand{\Id}{\operatorname{Id} }

\newcommand{\bG}{\mathbf{G}}
\newcommand{\bH}{\mathbf{H}}
\newcommand{\bT}{\mathbf{T}}
\newcommand{\bS}{\mathbf{S}}
\newcommand{\bA}{\mathbf{A}}
\newcommand{\bB}{\mathbf{B}}
\newcommand{\bP}{\mathbf{P}}
\newcommand{\bK}{\mathbf{K}}
\newcommand{\bU}{\mathbf{U}}
\newcommand{\bL}{\mathbf{L}}

\newcommand{\tran}{{^{\mathsf{T}}\!}}
\newcommand{\KPhi}{\vphantom{\Phi}_{\K}\Phi}
\newcommand{\RPhi}{\vphantom{\Phi}_{\R}\Phi}
\newcommand{\FPhi}{\vphantom{\Phi}_{\F}\Phi}
\newcommand{\KW}{\vphantom{W}_{\K}W}
\newcommand{\RW}{\vphantom{W}_{\R}W}
\newcommand{\FW}{\vphantom{W}_{\F}W}
\newcommand{\KDelta}{\vphantom{\Delta}_{\K}\Delta}

\newcommand{\FDelta}{\vphantom{\Delta}_{\F}\Delta}

\numberwithin{equation}{section}
\newtheorem{theorem}{Theorem}[section]
\newtheorem{lemma}[theorem]{Lemma}
\newtheorem{proposition}[theorem]{Proposition}
\newtheorem{corollary}[theorem]{Corollary}

\newtheorem{question}[theorem]{Question}

\makeatletter
\newtheorem*{rep@theorem}{\rep@title}
\newcommand{\newreptheorem}[2]{%
	\newenvironment{rep#1}[1]{%
		\def\rep@title{#2 \ref{##1}}%
		\begin{rep@theorem}}%
		{\end{rep@theorem}}}
\makeatother

\newreptheorem{theorem}{Theorem}
\newreptheorem{proposition}{Proposition}
\newreptheorem{lemma}{Lemma}
\newreptheorem{corollary}{Corollary}

\theoremstyle{remark} 
\newtheorem{remark}{Remark}[section]

\title[Semisimple algebraic groups over real closed fields]{Semisimple algebraic groups \\ over real closed fields}

\author{Raphael Appenzeller}
\address{Institute for Mathematics, Heidelberg University, Germany}
\email{rappenzeller@mathi.uni-heidelberg.de}

\date{\today}      

\begin{document}

\def\subjclassname{\textup{2020} Mathematics Subject Classification}
\expandafter\let\csname subjclassname@1991\endcsname=\subjclassname
\date{\today}

\begin{abstract}
	We give a self-contained introduction to linear algebraic and semialgebraic groups over real closed fields, and we generalize several key results about semisimple Lie groups to algebraic and semialgebraic groups over real closed fields. We prove that a torus in a semisimple algebraic group is maximal $\R$-split if and only if it is maximal $\F$-split for real closed fields $\F$. For the $\F$-points we formulate and prove the Iwasawa-decomposition $KAU$, the Cartan-decomposition $KAK$ and the Bruhat-decomposition $BWB$. For unipotent subgroups we prove the Baker-Campbell-Hausdorff formula, facilitating the analysis of root groups. We give a proof of the Jacobson-Morozov Lemma about subgroups whose Lie algebra is isomorphic to $\mathfrak{sl}_2$ for algebraic groups and a version for the $\F$-points, when the root system is reduced. We describe the rank 1 subgroups which are the semisimple parts of Levi-subgroups. We prove a semialgebraic version of Kostant's convexity theorem. The main tool used is a model theoretic transfer principle that follows from the Tarski-Seidenberg theorem.
\end{abstract}


\maketitle

\renewcommand{\baselinestretch}{0.6}\normalsize

	\setcounter{tocdepth}{2}
	\tableofcontents
	
	\renewcommand{\baselinestretch}{1.0}\normalsize
	
\section{Introduction and main results}
	
	Algebraic groups were first studied for algebraically closed fields, but significant progress has since been made in understanding them over arbitrary fields \cite{Bor66, Bor, Hum1, Zim}. The $\R$-points of algebraic groups are real Lie groups, and there are many more tools available from that point of view. In this paper, we extend several key results from the theory of Lie groups to algebraic and semialgebraic groups over real closed fields. In real algebraic geometry, real closed fields serve a role analogous to that of algebraically closed fields in classical algebraic geometry.
	
	Let $\F$ and $\K$ be real closed fields such that $\K \subseteq \F \cap \R$. Let $\bG$ be a semisimple, self-adjoint (if $g\in \bG$ then $g\tran \in \bG$) linear algebraic $\K$-group. An algebraic subgroup $\bS < \bG$ is a \emph{torus} if all its elements are simultaneously diagonalizable. If the elements are simultaneously diagonalizable over $\F$, then $\bS$ is called $\F$-split. Many decompositions rely on the choice of a torus. The following first result shows that all real closed fields define the same split tori.
	\begin{reptheorem}{thm:split_tori}
		A torus $\bS < \bG$ is maximal $\K$-split if and only if it is maximal $\F$-split. Moreover, there is a self-adjoint maximal $\F$-split torus.
	\end{reptheorem} 
	
	A subgroup $G < \operatorname{GL}_n(\K)$ that is a semialgebraic subset is called a \emph{linear semialgebraic group}. The $\K$-points $\bG(\K)$ of $\bG$ form such a linear semialgebraic group. 
	Let $G \subseteq \bG(\K)$ be a semialgebraic subgroup of $\bG(\K)$ that contains the semialgebraic connected component of the identity. Its semialgebraic $\R$-extension $G_\R \subseteq \R^{n \times n}$ is then a semisimple Lie group \cite{Pil88}. The main goal of this paper is to study the $\F$-extension $G_\F$ of $G$. Let $\bS(\K)$ be the $\K$-points of a self-adjoint maximal $\K$-split torus of $\bG$. Let $A$ be the semialgebraic connected component of $\bS(\K)$ and $A_\F$ its semialgebraic extension. Let $K := G \cap \operatorname{SO}_n$ and $K_\F$ the semialgebraic extension. Over the reals, $K_\R$ is a maximal compact subgroup of $G_\R$. We also extend the semialgebraic groups $N = \operatorname{Nor}_{K}(A)$ and $M = \operatorname{Cen}_{K}(A)$ to $N_\F$ and $M_\F$. An order on the root system $\Sigma$ associated to $A_\R$ allows us to define $U$, $U_\R$ and $U_\F$ as the exponentials of the sum of root spaces $(\frakg_\alpha)_\K$, $(\frakg_\alpha)_\R$ and $(\frakg_\alpha)_\F$ corresponding to positive roots. We prove the following versions of the Iwasawa ($KAU$), Cartan ($KAK$) and Bruhat ($BWB$) decompositions for $G_\F$. 
	 \begin{reptheorem}{thm:KAU}[$G=KAU$]
	 	For every $g \in G_\F$, there are $k \in K_\F$, $a \in A_\F$, $u \in U_\F$ such that $g = kau$. This decomposition is unique.
	 \end{reptheorem} 
	 \begin{reptheorem}{thm:KAK}[$G=KAK$]
	 	For every $g \in G_\F$, there are $k_1,k_2 \in K_\F$, $a \in A_\F$ such that $g = k_1 a k_2$.
	 	In this decomposition $a$ is uniquely determined up to conjugation by an element of the spherical Weyl group $N_{\F}/M_{\F}$.
	 \end{reptheorem}
	 \begin{reptheorem}{thm:BWB}[$G=BWB$]
	 	For every $g \in G_\F$, there are $b_1,b_2 \in B_\F := M_\F A_\F U_\F$ and $n\in N_\F$ such that $g = b_1 n b_2$. In this decomposition $n$ is unique up to multiplying by an element in $M_{\F}$. For the spherical Weyl group $W_s := N_\F/M_\F$, we have a disjoint union of double cosets
	 	$$
	 	G_{\F} =  \bigsqcup_{[n] \in W_s} B_{\F}nB_{\F}.
	 	$$ 
	 \end{reptheorem}
	 A version of the Bruhat-decomposition is known for algebraic groups \cite[Theorem 14.11]{Bor} over arbitrary fields, the other two decompositions come from the theory of Lie groups.
	 For the Iwasawa-decomposition, a related result has been obtained by Conversano \cite[Theorem 2.1]{Con14} in a more general (not neccessarily linear) setting of definable groups.

	 There are various definitions of root systems and Weyl groups in the literature. We use Theorem \ref{thm:split_tori} to verify how these objects defined via the theories of algebraic groups, real Lie groups, Lie algebras and in the semialgebraic setting all coincide. 
	 \begin{repproposition}{prop:weylgroups}
	 	The algebraic root system $\KPhi$ is isomorphic to the root system $\Sigma$ from the real setting. The spherical Weyl groups $\KW, \FW, N_\R/M_\R, N_\F/M_\F$ and the group generated by reflections in roots of the root system $\Sigma$ are all isomorphic. 
	 \end{repproposition}

	In contrast to the setting of Lie groups, the exponential map may not be defined for $G_\F$, it is however still defined for $U_\F$ as the elements of $U_\F$ are unipotent. We observe that the Baker-Campbell-Hausdorff formula holds for elements in $U_\F$, which is useful in the study of the structure of $U_\F$, see Section \ref{sec:U}. 
	
	\begin{repproposition}{prop:BCH}
		Let $u,v \in U_\F, X := \log(u), Y:= \log(v)$ and $Z = \log(uv)$. Then $\exp(X)\exp(Y) = \exp(Z)$ and the element $Z$ is given by a finite sum of iterated commutators in $X$ and $Y$, the first terms of which are given by
		$$
		Z = X + Y + \frac{1}{2}[X,Y] + \frac{1}{12}\left(\left[X,\left[X,Y\right]\right]- \left[Y, \left[Y,X\right]\right]\right) - \frac{1}{24} \left[Y,\left[X,\left[X,Y\right]\right]\right] + \ldots
		$$ 
	\end{repproposition}

	We give a proof of the folklore result that the Jacobson-Morozov Lemma holds for algebraic groups. A semialgebraic version is given in Proposition \ref{prop:Jacobson_Morozov_real_closed}.
	\begin{repproposition}{prop:Jacobson_Morozov_group}
		Let $g\in \bG$ be any unipotent element in a semisimple linear algebraic group $G$ over an algebraically closed field $\D$ of characteristic $0$. Then there is an algebraic subgroup $\operatorname{SL}_g < \bG$ with Lie algebra $\operatorname{Lie}(\operatorname{SL}_g) \cong \mathfrak{sl}_2$ and $g \in \operatorname{SL}_g$. The element $\log(g) \in \operatorname{Lie}(\bG)$ corresponds to 
		$$
		\begin{pmatrix}
			0 & 1 \\ 0 & 0
		\end{pmatrix}\in \mathfrak{sl}_2.
		$$
		Moreover, if $g\in \bG(\F)$ for a field $\F \subseteq \D$, then $\operatorname{SL}_g$ is defined over $\F$. 
	\end{repproposition}

	The Jacobson-Morozov Lemma produces subgroups with Lie algebras isomorphic to $\mathfrak{sl}_2$. The following theorem produces potentially larger rank one subgroups $\bL_{\pm \alpha}$ associated to a root $\alpha \in \Sigma$ in the algebraic setting. These subgroups are the semisimple parts of Levi subgroups. 
	\begin{reptheorem}{thm:levi_group}
		Let $\alpha \in \Sigma$. Then there is a semisimple self-adjoint linear algebraic group $\bL_{\pm \alpha}$ defined over $\K$ such that
		\begin{enumerate}
			\item [(i)] $\operatorname{Lie}(\bL_{\pm \alpha}) = ( \frakg_\alpha \oplus \frakg_{2\alpha} ) \oplus (\frakg_{-\alpha} \oplus \frakg_{-2\alpha} ) \oplus ( [\frakg_\alpha , \frakg_{-\alpha}] + [\frakg_{2\alpha}, \frakg_{-2\alpha}] )$, and
			\item [(ii)] $\operatorname{Rank}_\R(\bL_{\pm \alpha}) = \operatorname{Rank}_\F(\bL_{\pm \alpha}) = 1$.
		\end{enumerate}
	\end{reptheorem}

	Using the Iwasawa decomposition $G_\F=U_\F A_\F K_\F$ (Theorem \ref{thm:KAU_R}) we associate to every $g=uak\in G_\F$ its $A$-component $a_\F(g) = a \in A_\F$. Over the reals, Kostant's convexity theorem describes the $A$-components of the $K_\R$-orbit (under left multiplication) of an element in $A_\F$ as a convex set. We prove the following semialgebraic version of Kostant's convexity theorem over $\F$, restricted to the multiplicative closed Weyl chamber
	$$
	\Ap_\F : = \left\{ a \in A_{\F} \colon \chi_\delta(a) \geq 1 \text{ for all } \delta \in \Delta \right\} ,
	$$
	where the $\chi_\delta$ are the multiplicative characters associated to the elements of some basis $\Delta$ of the root system $\Sigma$. 
	
	\begin{reptheorem}{thm:kostant_F}
		For all $b \in \ApF$, we have
		$$
		\left\{ a \in \ApF \colon \exists k \in K_\F , a_\F(kb) = a \right\}
		= \left\{ a\in \ApF \colon \chi_i (a) \leq \chi_i (b) \text{ for all } i \right\}
		$$
		for certain algebraic characters $\chi_i \colon A_\F \to \F_{\geq 0}$.
	\end{reptheorem}
	
	\subsection{Applications}
	
	This work is part of the author's doctoral dissertation \cite{App24arxiv}, where the results of this paper are applied to construct a quotient $B$ of the nonstandard symmetric space $G_\F / K_\F$, following ideas from \cite{KrTe} and to show that $B$ is an affine $\Lambda$-building, see also \cite{appenzeller2026buildings}. In \cite{BIPP21,BIPP23arxiv}, Burger, Iozzi, Parreau and Pozzetti use Theorem \ref{thm:split_tori}, Corollary \ref{cor:split_tori} and Theorem \ref{thm:KAK}, as well as the building $B$ to interpret boundary points of the real spectrum compactification of character varieties. The real spectrum compactification is a promising new approach to studying character varieties and has some advantages over other compactifications, for instance it preserves connected components. We believe that the results in this paper can widely be applied in any context where semisimple algebraic groups over real closed fields appear, just as the corresponding results about real Lie groups were in the past. 
	
	\subsection{Semialgebraic groups and o-minimality.}
	
	In model theory, a structure $(M,<,\ldots)$ is an \emph{o-minimal structure} if $<$ is a dense linear order and every definable subset of $M$ is a finite union of points and open intervals.  Starting with \cite{Pil88}, groups definable in o-minimal structures have been studied extensively \cite{OPP96, PPS00a, PPS00b, PPS02, CoPi13, Con14, BJO14, BBO19, COP22}, also see the surveys \cite{Ote08, Con21}. Real closed fields are mayor examples of o-minimal structures. A group $G$ is called \emph{semialgebraic over a real closed field $\F$} if the set $G$ and the graph of the multiplication are definable in $\F$ with parameters in $\F$. A semialgebraic group $G$ is called \emph{linear semialgebraic} if $G < \operatorname{GL}_n(\F) \subseteq \F^{n \times n}$. The groups $G$ considered in this article are examples of linear semialgebraic grops.
	
	If $G$ is a semisimple group definable over a o-minimal expansion of a real closed field $\F$, then $G$ is actually semialgebraic over $\F$ \cite{PPS02}. If $G$ is a centerless semisimple semialgebraic group, then $G$ is definably isomorphic to a semisimple linear semialgebraic group \cite[Cor. 3.3]{OPP96}. Let $\bG$ be the Zariski-closure of a linear semialgebraic group $G$, then $G$ is a semialgebraic subgroup of $\bG(\F)$ and contains the semialgebraic connected component of $\bG(\F)$. In the setting of this paper we additionally assume that $G$ is self-adjoint, see Remark \ref{rem:self-adjoint} for a suggestion of how to get rid of that assumption. It would be interesting to find out exactly how far the present results can be generalized to groups definable in o-minimal structures.

	\subsection{Outline}
	We start by giving a short introduction to the theory of real closed fields, the transfer priciple and its application to extensions of semialgebraic sets in Section \ref{sec:real_closed}. In Section \ref{sec:algebraic_groups}, we give a self-contained account of the theory of linear algebraic groups, including examples. In Section \ref{sec:Lie_algebra_results}, we recall the theory of real Lie algebras before extending some of the results to real closed fields and proving Theorem \ref{thm:split_tori} about maximal split tori. In Section \ref{sec:decompositions} we finally introduce the slightly more general notion of semialgebraic groups and give proofs for group decompositions and the other results mentioned above.

	\subsection{Acknowledgements}
	
	I would like to thank Marc Burger for continued support and specifically for help with Theorems \ref{thm:split_tori} and \ref{thm:kostant_F}. I would like to thank Luca de Rosa and Xenia Flamm for discussing early versions of Proposition \ref{prop:Jacobson_Morozov_group}, and Victor Jaeck for feedback on the sections on algebraic groups. I am very thankful for the comments of two anonymous referees, one of them pointed out the connection to groups definable in o-minimal structures.

	\section{Real closed fields}\label{sec:real_closed}
	A great reference for real algebraic geometry is \cite{BCR}. An \emph{ordered field} is a field together with a total order such that the sum and the product of positive elements are positive. Note that every ordered field is of characteristic $0$.
	A field $\F$ is called \emph{real closed} if it satisfies one of the following equivalent conditions.
	\begin{enumerate}
		\item [(1)] There is a total order on $\F$ turning $\F$ into an ordered field such that every positive element has a square root and every polynomial of odd degree has a solution.
		\item [(2)] There is an order on $\F$ that does not extend to any proper algebraic field extension of $\F$.
		\item [(3)] $\F$ is not algebraically closed but every finite extension is algebraically closed.
		\item [(4)] $\F$ is not algebraically closed but $\F[\sqrt{-1}]$ is algebraically closed.
	\end{enumerate}
	An ordered field is called \emph{Archimedean} if every element is bounded by a natural number. The real numbers $\R$ and the subset of real algebraic numbers are examples of Archimedean real closed fields.
	A major tool when working with real closed fields is the following transfer principle from model theory.

	\subsection{The transfer principle}
	
	Let $\K$ be an ordered field. Recall that a \emph{first-order formula of ordered fields with parameters in $\K$} is a formula that contains a finite number of conjunctions $\wedge$, disjunctions $\lor$, negations $\lnot$, and universal $\forall$ or existential $\exists$ quantifiers on variables, starting from atomic formulas which are formulas of the kind $f(x_1, \ldots, x_n) = 0$ or $g(x_1, \ldots , x_n) \leq 0$, where $f$ and $g$ are polynomials with coefficients in $\K$. A first-order formula without free variables is called a \emph{sentence}. By the Tarski-Seidenberg theorem, any sentence is equivalent to a sentence without quantifiers, from which can be deduced that the theory of real closed fields in the language of ordered fields is complete. In practice this means the following. 
	\begin{theorem}\label{thm:logic} (Transfer principle, \cite{BCR})
		Let $\F$ and $\F'$ be real closed fields. Let $\varphi$ be a sentence with parameters in $\F \cap \F'$. Then $\varphi$ is true for $\F$ if and only if $\varphi$ is true for $\F'$, formally $\F \models \varphi \iff \F' \models \varphi$.
	\end{theorem}
	
	Let $\varphi$ be a first-order formula with parameters in some field $\K \subseteq \F \cap \R$ with $n$ free variables. Let $X$ be a subset of $\K^n$ which can be described as $X = \{x \in \K^n \colon \F \models \varphi(x)\}$. It follows from the transfer principle, that the \emph{semialgebraic extension $X_\F = \{x \in \F^n \colon \F \models \varphi(x)\}$ of $X$} depends only on $X$ and not on $\varphi$ and is thus well defined. Sets of the form $\{x \in k \colon k \models \varphi(x)\}$ for any ordered field $k \supseteq \K$ are called \emph{semialgebraic} sets.

	\subsection{Examples of real closed fields}\label{sec:real_closed_fields_examples}
	The field of \emph{Puiseux series} over the real algebraic numbers $\Qrc$
	$$
	\F  := \left\{  \sum_{k=-\infty}^{k_0} c_k X^{\frac{k}{m}} \, \colon \, k_0, m \in \mathbb{Z}, \, m > 0 , \, c_k \in \Qrc, \, c_{k_0}\neq 0\right\},
	$$ 
	is a non-Archimedean real closed field, where the usual order on $\Qrc$ is extended by $X>r$ for all $r \in \Qrc$ \cite{BCR}. 
	
	A \emph{non-principal ultrafilter on $\Z$} is a function $\omega \colon \mathcal{P}(\Z) \to \{0,1\}$ that satisfies
	\begin{enumerate}
		\item [(1)] $\omega(\emptyset) = 0$, $\omega(\Z) = 1$
		\item [(2)] If $A,B \subseteq \Z$ satisfy $A\cap B = \emptyset$, then $\omega(A \cup B) = \omega(A) + \omega(B)$.
		\item [(3)] All finite subsets $A \subseteq \Z$ satisfy $\omega(A) = 0$.
	\end{enumerate}
	Ultrafilters can be thought of as finitely-additive probability measures that only take values in $0$ and $1$. The existence of non-principal ultrafilters is equivalent to the axiom of choice, \cite{Hal}. For a given ultrafilter $\omega$, we define the \emph{hyperreal numbers} $\R_{\omega}$ to be the equivalence classes of infinite sequences $\R_{\omega} = \R^{\N}/ \!\! \sim$, where $x = (x_i)_{i \in \N} \sim y = (y_i)_{i \in \N}$ if $\omega(\{ i \in \N \colon x_i \neq y_i\}) = 0$ or $\omega(\{ i \in \N \colon x_i = y_i\}) = 1$. We define addition and multiplication componentwise, the multiplicative inverse is obtained by taking the inverses of all non-zero entries, turning $\F_\omega$ into a field. Considering constant sequences, the real numbers are a subfield of $\R_\omega$. The hyperreals are an ordered field with respect to the order defined by $[(x_i)_{i\in \N}] \leq [(y_i)_{i \in \N} ]$ if and only if $\omega(\{i \in \N \colon x_i \leq y_i\}) = 1$. The hyperreals are real closed, since $\R$ is. The hyperreals are non-Archimedean, since the equivalence class containing $(1, 2, 3, \ldots)$ is an \emph{infinite} element, meaning it is larger than any natural number. 
	
	Let $b \in \R_{\omega}$ be an infinite element. Then
	$$
	O_b := \{ x \in \R_{\omega} \colon |x| < b^m \text{ for some } m \in \Z \}
	$$
	is an order convex subring of $\R_\omega$ with maximal ideal
	$$
	J_b := \{ x \in \R_\omega \colon |x| <  b^m \text{ for all } m \in \Z \}.
	$$
	The \emph{Robinson field} associated to the non-principal ultrafilter $\omega$ and the infinite element $b$ is the quotient $\R_{\omega,b} := O_b/J_b$ \cite{Rob96}. The Robinson field is a non-Archimedean real closed field. Note that $[b] \in \R_{\omega,b} $ is a \emph{big} element, meaning that for all $a \in \R_{\omega,b}$ there is an $n \in \N$ such that $a < b^n$. 
	
	By definition (1) of real closed fields, the sentence
	$$
	\varphi \colon \quad \forall a \colon a > 0 \to  \exists b \colon b\cdot b = a 
	$$
	in the language of ordered fields with parameters in $\Q$ holds over any real closed field $\mathbb{F}$, formally $\mathbb{F} \models \varphi$. In fact, being real closed can be described by sentences in the language of ordered fields. On the other hand, being Archimedean can not be described as a sentence in the language of ordered fields. The attempt of writing the Archimedean condition as a sentence
	$$
	\psi \colon \quad \forall a \colon  \exists n \in \mathbb{N} \colon a < n
	$$
	fails, since the use of the symbols $\in$ and $\mathbb{N}$ is not allowed in the language of ordered fields. This is compatible with Theorem \ref{thm:logic}, as there are some real closed fields that are Archimedean and others that are not. We note that by Gödel's incompleteness theorem, the natural numbers $\mathbb{N}$ can not be described in any theory that satisfies the transfer principle.

\section{Linear algebraic groups}\label{sec:algebraic_groups}
	 \subsection{Definitions}
	 
	 \begin{wrapfigure}{r}{0.1\textwidth}
	 	\centering
	 	\includegraphics[width=0.11\textwidth]{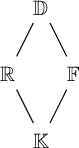}
	 \end{wrapfigure}
	 
	 Let $\K \subseteq \R$ be a subfield (usually $\K = \Q$, $\K= \overline{\Q}^{\operatorname{rc}}$ or $\K = \R$), $\F$ a real closed field containing $\K$ (usually a non-Archimedean field such as the Puiseux series) and $\D$ an algebraically closed field that contains both $\R$ and $\F$. We follow a naive approach to algebraic groups, viewing them as matrix groups. Since we are working with fields of characteristic 0, the algebraic geometry can be kept at a minimum. For an extensive introduction to algebraic groups, we refer to \cite{Bor66}, \cite{Bor}, \cite{Hum1} or \cite{Zim}.
	 
	 The general linear group $\GLnD$ can be realized as an affine algebraic variety
	 $$
	 \GLnD \cong \left\{ \begin{pmatrix}
	 	A & \\ & t
	 \end{pmatrix} \in \D^{(n+1)\times(n+1)} \colon \det (A) \cdot t = 1 \right\}.
	 $$
	 A subgroup $\bG < \GLnD$ is called a \emph{linear algebraic group defined over $\K$}, if it is a set of common zeros for a set of polynomials in the coordinate ring $\K[\operatorname{GL}_n(\D)]:= \K[(x_{ij}), \det(x_{ij})^{-1}]$ of $\operatorname{GL}_n(\D)$. 
	 We will not consider more general algebraic groups and hence also call $\bG$ an \emph{algebraic group} or a $\K$-\emph{group}.
	 For any commutative $\K$-subalgebra $\mathbb{B} \subseteq \D$ containing $\K$, we let 
	 $$
	 \GLnB = \left\{ (a_{ij}) \in \GLnD \colon a_{ij} \in \mathbb{B} \text{ and } \det(a_{ij})^{-1} \in \mathbb{B} \right\}.
	 $$ 
	 The \emph{group of $\mathbb{B}$-points of a linear algebraic group $G$}
	 is $\bG(\mathbb{B}) := \bG \cap \GLnB$. The $\R$-points of a linear algebraic group $\bG$ form a real Lie group \cite[Theorem 2.1]{Mil13}. 
	 
	 Viewing $\bG \subseteq \GLnD $ as an algebraic subset of $\D^{n+1\times n+1}$, we endow $\bG$ with the Zariski-topology, whose closed sets are given by common zeros of sets of polynomials in $\K[\operatorname{GL}_n(\D)]$. A linear algebraic group $\bG$ is \emph{semisimple} if it is connected and every closed connected normal abelian subgroup is trivial. Similarly, a Lie group is \emph{semisimple} if every closed connected normal abelian subgroup is trivial, but now in the Lie group topology. Using Zariski-closures of subgroups, one sees that $\bG$ is semisimple if and only if $\bG(\R)$ is semisimple.

 \subsubsection{Examples}
 
 The multiplicative group $\bG_m = \operatorname{GL}_1(\D)$ is a linear algebraic group defined over $\Q$. Note that $\bG_m$ is connected in the Zariski-topology. Its $\C$-points $(\bG_m)_\C = \C \setminus \{0\}$ are connected in the Euclidean topology but its $\R$-points $(\bG_m)_\R = \R \setminus \{0\}$ are not.
 
   If $V$ is a $\D$-vector space, then the group $\operatorname{GL}(V)$ of automorphisms of $V$ is an algebraic group. As $\operatorname{GL}(V)$ contains the closed connected normal abelian subgroup of scalar multiplication, $\operatorname{GL}(V)$ is not semisimple.  
  
  The real Lie groups $\SLnR$, $\operatorname{SO}_n(\R)$ and $\operatorname{Sp}_{2n}(\R)$ are groups of $\R$-points of the linear algebraic groups $\SL_n(\D)$, $\operatorname{SO}_n(\D)$ and $\operatorname{Sp}_{2n}(\D)$ defined over $\Q$.

 \subsection{Morphisms and tori }

For a linear algebraic group $\bG$ we consider the ideal 
$$
I(\bG) = \left\{p \in \D\left[\left(x_{ij}\right),z\right] \colon  p\left(g,\det\!\left(g\right)^{-1}\right) = 0 \text{ for all } g \in \bG   \right\},
$$
which is finitely generated as a consequence of Hilbert's Basis theorem. The \emph{coordinate ring} of $\bG$ is 
$$
\D [\bG] = \D\left[\left(x_{ij}\right),d\right] / I(\bG)
$$
and its elements are called \emph{regular functions} on $\bG$. Given a map $\varphi\colon \bG \to \bH$ between algebraic groups $\bG$ and $\bH$, we can define the \emph{transposed map} $\varphi^\circ \colon \D[\bH] \to \D[\bG]$ by $\varphi^\circ (f) = f \circ \varphi$. A \emph{morphism} of linear algebraic groups $\bG$ and $\bH$ is a group homomorphism $\varphi \colon \bG \to \bH$ whose transposed map $\varphi^\circ$ is a ring homomorphism. If $\bG$ and $\bH$ are defined over $\K$ and $\varphi^\circ$ maps $\K[\bH]$ into $\K[\bG]$, then $\varphi$ is \emph{defined over} $\K$. An important example of a morphism $\bG \to \bG$ is the conjugation $\operatorname{Int}(g) \colon h \mapsto ghg^{-1}$ by an element $g \in \bG$, and if $g \in \bG(\K)$, then $\operatorname{Int}(g)$ is defined over $\K$.

\begin{lemma}\label{lem:basic_algebraic_geometry}
	Let $\varphi \colon \bG \to \bH$ be a morphism of algebraic $\K$-groups $\bG < \operatorname{GL}(n,\D)$, $\bH < \operatorname{GL}(m,\D)$ defined over $\K$. If we write componentwise
	\begin{align*}
		\varphi_\F \colon \bG(\F) & \to \bH(\F) \\
		 \begin{pmatrix}
			x_{11} & \ldots & \\ \vdots & \ddots & \\ & & x_{nn}
		\end{pmatrix} & \mapsto \begin{pmatrix}
		\varphi_{11}(x_{11}, \ldots , x_{nn}) & \ldots & \\ \vdots & \ddots & \\ & & \varphi_{mm}(x_{11}, \ldots, x_{nn})
		\end{pmatrix} 
	\end{align*}
	then all $\varphi_{ij} \colon \bG(\F) \to \F $ are polynomials with coefficients in $\K$.
\end{lemma}
\begin{proof}
	Consider the polynomial representing $X_{ij} \in \K[\bH]$. Then 
	\begin{align*}
		\varphi^\circ (X_{ij})(x_{11}, \ldots , x_{nn}) &= X_{ij} (\varphi(x_{11}, \ldots, x_{mm})) \\
		&= X_{ij}(\varphi_{11}(x_{11}, \ldots , x_{nn}), \ldots , \varphi_{mm}(x_{11}, \ldots , x_{nn}))\\
		& = \varphi_{ij}(x_{11}, \ldots , x_{nn})
	\end{align*} 
	and thus $\varphi_{ij} \in \K[\bG]$ is a polynomial.
\end{proof}

 A linear algebraic group $\bT$ that is isomorphic to $(\bG_m)^d$ is called a \emph{$d$-dimensional torus}. An element $g \in \operatorname{GL}_n(\D)$ is \emph{semisimple} if $\D^n$ is spanned by eigenvectors of $g$, or equivalently, if $g$ is diagonalizable. 
 
 \begin{theorem}[8.5 in \cite{Bor}] 
 	\label{thm:bor_torus}
 	For a Zariski-connected algebraic group $\bT$, the following conditions are equivalent.
 	\begin{enumerate}
 		\item $\bT$ is a torus.
 		\item $\bT$ consists only of semisimple elements.
 		\item The whole group $\bT$ is simultaneously diagonalizable.
 	\end{enumerate}
 \end{theorem}

Since Zariski-connectedness of algebraic groups and semisimplicity of elements is preserved under morphisms \cite[4.4(4)]{Bor}, the image of a torus under a morphism is a torus. 
 A morphism $\chi \colon \bG \mapsto \bG_m$ from an algebraic group $\bG$ to $\bG_m$ is called a \emph{character}. The set of characters of $\bG$ is an abelian group, denoted by $\hat{\bG}$. If $\bT \cong (\bG_m)^d$ is a $d$-dimensional torus and an element $x \in \bT$ corresponds to $(x_1, \ldots , x_d)  \in (\bG_m)^d$, then every character $\chi$ of $\bT$ is of the form 
 $$
 \chi(x) = x_1^{n_1} \cdots x_d^{n_d}
 $$
 for some $n_i \in \Z$ and hence $\hat{\bT} = \Z^d$. If there is an isomorphism $\bT \to (\bG_m)^d$ defined over $\K$, $\bT$ is said to be \emph{$\K$-split}.
  
   \begin{theorem}[8.4 in \cite{Bor}]
  	For a torus $\bT$ defined over $\K$, the following conditions are equivalent.
  	\begin{enumerate}
 	\item $\bT$ is $\K$-split.
 	\item All characters of $\bT$ are defined over $\K$, 
  	\end{enumerate}
  \end{theorem}

A \emph{maximal torus} $\bT \subseteq \bG$ is a subgroup that is a torus and is not properly contained in any other torus in $\bG$. A \emph{maximal $\K$-split torus} is a $\K$-split torus that is maximal among $\K$-split tori.

\begin{theorem}[11.3 and 20.9 in \cite{Bor}]
	In a connected algebraic group $\bG$, all maximal tori of $\bG$ are conjugate. In a semisimple 
	$\K$-group $\bG$, the maximal $\K$-split tori of $\bG$ are conjugate under $\bG(\K)$.
\end{theorem}
The \emph{rank} of $\bG$ is the common dimension of the maximal tori and if $\bG$ is semisimple 
and defined over $\K$, the $\K$-\emph{rank} of $\bG$ is the common dimension of the maximal $\K$-split tori of $\bG$. A morphism $\bG_m \to \bG$ is a \emph{multiplicative one-parameter subgroup of $\bG$}. For a torus $\bT$ the one-parameter subgroups of $\bT$ are denoted by $X_{\star}(\bT)$. For $\chi \in \hat{\bT}$ and $\lambda \in X_{\star}(\bT) $, the composition $(\chi\circ\lambda) \colon \bG_m \to \bG_m$ is a character of the torus $\bG_m$ and hence sends $x \in \bG_m$ to $x^m$ for some $m \in \Z$. This defines a map $b \colon \hat{\bT} \times X_{\star}(\bT) \to \Z, (\chi,\lambda) \mapsto m$. 
\begin{proposition}{\cite[Proposition 8.7]{Bor}}\label{prop:char_cochar}
	For any torus $\bT$, the map 
	$$
	b\colon \hat{\bT} \times X_{\star}(\bT) \to \Z
	$$
	is a nondegenerate bilinear form.
\end{proposition}

 \subsubsection{Examples}\label{sec:ex_split}
 The $\R$-group $\operatorname{SO}_2(\mathbb{C}) = \{ g \in \operatorname{SL}_2(\C) \colon gg\tran = \Id \}$ is a torus. An isomorphism with $\bG_m$ is given by
 \begin{align*}
 	\bG_m & \to \operatorname{SO}_2(\mathbb{C}) \\
   a & \mapsto \frac{1}{2}\begin{pmatrix}
 	a+a^{-1} & i (a^{-1}-a) \\
 	i(a - a^{-1}) & a+a^{-1}
 \end{pmatrix}  \\
	 a+ib & \mapsfrom \phantom{\frac{1}{2}}  \begin{pmatrix}
	a & b \\
	c & d
\end{pmatrix}
 \end{align*}
 but $\operatorname{SO}_2(\mathbb{C})$ is not $\R$-split. 
 On the other hand, the $\R$-group $\bA(\mathbb{C}) = \{ g \in \operatorname{SL}_2(\C) \colon g\text{ is diagonal}\}$ is an $\R$-split torus. Since $\C$ is algebraically closed, both tori are $\C$-split. Both $\operatorname{SO}_2(\mathbb{C})$ and $\bA(\mathbb{C})$ are maximal tori of $\operatorname{SL}_2(\mathbb{C})$ and they are conjugate under $\operatorname{SL}_2(\mathbb{C})$, but not under $\operatorname{SL}_2(\R)$. The rank and the $\R$-rank of $\operatorname{SL}_2(\mathbb{C})$ is $1$.

\subsection{The Lie Algebra and the adjoint representation}

Let $e\in \bG$ be the neutral element of an algebraic $\K$-group $\bG$. There are multiple ways to define the Lie Algebra of $\bG$, a few of which can be found for instance in Chapters 5 and 9 of \cite{Hum1}. A first explicit definition uses the description of $\bG$ as a matrix group
$$
\bG = \left\{ (a_{ij}) \in \operatorname{GL}_n(\D) \colon f(a_{ij}) = 0 \text{ for all } f \in I \right\}
$$
where $I \subset \K[(x_{ij}),\det(x_{ij})^{-1}]$ is a finite subset. For the following definition we restrict ourselves to subgroups of $\operatorname{SL}_n(\D)$, so that we can ignore the dependence on $\operatorname{det}(x_{ij})^{-1}$. For $f \in I$ we define the \emph{differential of $f$ at $e:=\operatorname{Id}$} to be the linear polynomial
$$
\operatorname{d}_{e}\!f = \sum_{ij}  \frac{\partial f}{\partial x_{ij}} (e) \cdot x_{ij} \in \K[(x_{ij})]. 
$$
and the \emph{Zariski tangent space}  
$$
T_e\bG := \left\{ (a_{ij}) \in \D^{n\times n} \colon  \operatorname{d}_e\!f(a_{ij}) = 0 \text{ for all } f \in I\right\}.
$$
We now give another useful description of the tangent space in terms of derivations. A \emph{point derivation} is a $\D$-linear map $\delta \colon \D[\bG] \to \D$ that satisfies the Leibnitz rule $\delta(f f') = \delta(f) f' + f \delta(f')$ for $f,f' \in \D[\bG]$. For $a = (a_{ij}) \in T_e\bG$ the function $\delta_a = f \mapsto \operatorname{d}_e\!f(a_{ij})$ is a point derivation and every point derivation is of this form \cite[Section 5.1]{Hum1}. 

From a more algebraic point of view, a \emph{derivation} of the coordinate ring $\D[\bG]$ (viewed as a $\D$-algebra) of a linear algebraic group $\bG$ is a linear map
$$
X \colon \D[\bG] \to \D[\bG]
$$
that satisfies the Leibnitz rule $X(ff') = X(f)f' + f X(f')$ for $f,f' \in \D[\bG]$. Algebraic groups act by left (and right) translation on their coordinate rings $\D[\bG]$. For $g \in \bG$, the left translation is given by $\lambda_g (f) (h) = f(g^{-1}h)$ for $f \in \D[\bG]$, $h \in \bG$ (and the right translation by $\rho_g (f) (h) = f(hg)$). The $\D$-vector space of derivations of $\D[\bG]$ which commute with right translations
\begin{align*} 
	\frakg &:= \left\{ X \colon \D[\bG] \to \D[\bG] \colon \begin{matrix}
		X \text{ is a derivation of } \D[\bG] \text{ with }\\
		X\circ \rho_g =\rho_g \circ X \text{ for all }  g \in \bG
	\end{matrix} \right\},
\end{align*}
is isomorphic to the space of point derivations (a derivation $X$ corresponds to the point derivation $\delta \coloneq  \operatorname{ev}_e \circ X \colon \D[\bG] \to \D$, so $\delta(f) = (Xf)(e)$ for $f \in \D[\bG]$) \cite[Theorem 9.1]{Hum1}, which in turn is identified with the Zariski tangent space $T_e\bG$ of $\bG$ as before.
 Endowed with the Lie algebra structure defined by the bracket operation on derivations, $\frakg$ is called the \emph{Lie algebra of} $\bG$. For a commutative $\K$-subalgebra $\mathbb{B} \subseteq \D$ containing $\K$, we can use the identification $\frakg = T_e\bG$, to define the \emph{$\mathbb{B}$-points of the Lie algebra of $\bG$}
 $$
 \frakg(\mathbb{B}) = T_e\bG \cap \mathbb{B}^{n\times n}.
 $$
 We note that the real Lie algebra $\frakg(\R)$ is the Lie algebra of the real Lie group $\bG(\R)$ \cite[Theorem 2.1]{Mil13}.
 
Let $\varphi  \colon \bG \to \bH$ be a morphism of algebraic groups $\bG,\bH$. The transposed map $\varphi^\circ \colon \D[\bH] \to \D[\bG]$ gives rise to a linear map $\operatorname{d}_e\!\varphi$ which is called the \emph{differential of $\varphi$ at $e\in \bG$}. In terms of point derivations, the differential is defined as
\begin{align*}
	\operatorname{d}_e\!\varphi \colon  \frakg  &\to \frakh \\
	\delta &\mapsto \delta \circ \varphi^{\circ} .
\end{align*}

The group $\bG$ acts on itself by conjugation $\operatorname{Int}(g) \colon \bG \to \bG, h \mapsto ghg^{-1}$. The differential of $\operatorname{Int}(g)$ at $e$ is denoted $\operatorname{Ad} (g) $. The morphism $\operatorname{Ad} \colon \bG \to \operatorname{GL}(\frakg)$ is called the \emph{adjoint representation of} $\bG$ and is given by
$$
\operatorname{Ad}(g)(X) = gXg^{-1} \in T_e\bG
$$
when $X \in \frakg$ is viewed as an element of $T_e\bG$. The differential of $\operatorname{Ad}$ at $e$ is called the \emph{adjoint representation} $\operatorname{ad} \colon \frakg \to \mathfrak{gl}(\frakg) $. Identifying $\frakg \cong T_e\bG$, $\operatorname{ad}$ is given by
$
 \operatorname{ad}(X)(Y) = \left[ X , Y \right]
$
for $X,Y \in \frakg$, where $[\cdot \,  , \cdot]$ is the matrix bracket.

\subsection{Root systems and the spherical Weyl group}\label{sec:alg_root_system}

Let $\bG$ now be a semisimple
algebraic group and $\bS$ a torus of $\bG$. Since $\operatorname{Ad} (\bS)$ is also a torus, its elements are simultaneously diagonalizable and the Lie algebra $\frakg$ decomposes into eigenspaces
$$
\frakg = \frakg_0^{(\bS)} \oplus\bigoplus_{\alpha \neq 0} \frakg_\alpha^{(\bS)}
$$
where 
$$
\frakg_\alpha^{(\bS)} = \{ X \in \frakg \colon \operatorname{Ad} (s) (X) = \alpha(s) \cdot X \ \text{ for all } s\in \bS \}
$$
for some $\alpha \in \hat{\bS}$. Elements $\alpha \neq 0$ with $\frakg_\alpha^{(\bS)} \neq 0$, are called the \emph{roots (relative to $\bS$)} and $\frakg_\alpha^{(\bS)}$ the \emph{root spaces}. The set of all roots is denoted by $\Phi(\bG,\bS)$. If $\bG$ is defined over $\K$ and $\bS$ is a maximal $\K$-split torus of $\bG$, then $\KPhi := \Phi(\bG,\bS)$ is called the set of $\K$-\emph{roots} of $\bG$. Since all maximal $\K$-split tori of $\bG$ are conjugate over $\K$ \cite[Theorem 20.9(ii)]{Bor}, $\KPhi$ only depends on $\K$ and not on the choice of maximal $\K$-split torus $\bS$.

\begin{theorem}[21.6 \cite{Bor}]\label{thm:alg_rootsystem}
	Let $\bG$ be a semisimple $\K$-group and $\bS$ a maximal $\K$-split torus of $\bG$. Recall that $\hat{\bS}\cong \Z^d$. Let $\Phi = \Phi(\bG,\bS)$. Then there is an admissible scalar product on the $\R$-vector space $V = \hat{\bS} \otimes_\Z \R$ such that ($V$,$\Phi$) is a crystallographic root system, that is:
	\begin{enumerate}
		\item $\Phi$ is a finite, symmetric ($\Phi = -\Phi$) subset of $V$, which spans $V$ and does not contain $0$.
		\item For every $\alpha \in \Phi$ there is a reflection $r_\alpha \colon V \to V$ with respect to $\alpha$ which leaves $\Phi$ stable.
		\item If $\alpha, \beta \in \Phi$, then $2\langle \alpha, \beta \rangle / \langle \alpha, \alpha \rangle \in \Z$.
	\end{enumerate} 
\end{theorem}

If $\bS$ is a maximal $\K$-split torus, the \emph{spherical Weyl group relative to $\K$} is
$$
\KW =\KW (\bS,\bG) = \operatorname{Nor}_{\bG}(\bS)/\operatorname{Cen}_{\bG}(\bS)
$$
and acts faithfully on $\bS$, $\hat{\bS}$ and $\KPhi$.

\subsection{Borel subgroups, parabolic subgroups}
Let $\bG$ be a connected algebraic group. 
A subgroup is \emph{solvable} if its derived series consisting of iterated commutator groups terminates. A subgroup $\bB < \bG$ is a \emph{Borel subgroup} of $\bG$ if it is maximal among the connected solvable subgroups. 
A closed subgroup $\bP < \bG$ is \emph{parabolic} if and only if it contains a Borel subgroup. 
The minimal parabolic 
subgroups are exactly the Borel subgroups. In general, a Borel subgroup may not be defined over $\K$ and in this case the minimal parabolic $\K$-subgroups may not be Borel subgroups.

\begin{theorem}(Bruhat decomposition, \cite[Theorem 21.15]{Bor})
Let $\bS$ be a maximal $\K$-split torus and $\bP$ a minimal parabolic $\K$-subgroup containing $\bS$.  Denote by $\pi \colon \operatorname{Nor}_{\bG}(\bS) \to \KW$ the Weyl group projection. Then $\bG(\K) = \bP(\K) \operatorname{Nor}_{\bG(\K)}(\bS(\K))\bP(\K)$, in fact there is a disjoint union of double classes
$$
\bG(\K) = \bigsqcup_{w \in  \KW } \bP(\K) \pi^{-1}(w) \bP(\K).
$$
\end{theorem}

After choosing a minimal parabolic $\K$-subgroup $\bP$ containing a maximal $\K$-split torus, any parabolic $\K$-subgroup containing $\bP$ is called \emph{standard parabolic}.

\subsubsection{Examples}
The group of diagonal matrices in $\operatorname{SL}_3=\operatorname{SL}_3(\D)$ is a maximal $\Q$-split torus $\bT$. There are six minimal parabolic $\Q$-subgroups (which in this case are Borel subgroups) containing $\bT$. They are given by
\begin{align*}
	\left\{ \begin{pmatrix}
		\star & \star & \star  \\
		0 & \star & \star  \\
		0 & 0 & \star  \\
	\end{pmatrix} \in \operatorname{SL}_3 \right\} , \quad \left\{ \begin{pmatrix}
	\star & 0 & \star  \\
	\star & \star & \star  \\
	0 & 0 & \star  \\
\end{pmatrix} \in \operatorname{SL}_3 \right\} , \quad\left\{ \begin{pmatrix}
\star & \star & \star  \\
0 & \star & 0  \\
0 & \star & \star  \\
\end{pmatrix} \in \operatorname{SL}_3 \right\} , \\
\left\{ \begin{pmatrix}
	\star & 0 & 0  \\
	\star & \star & 0  \\
	\star & \star & \star  \\
\end{pmatrix} \in \operatorname{SL}_3 \right\} , \quad \left\{ \begin{pmatrix}
	\star & 0 & 0  \\
	\star & \star & \star  \\
	\star & 0 & \star  \\
\end{pmatrix} \in \operatorname{SL}_3 \right\} , \quad\left\{ \begin{pmatrix}
	\star & \star & 0  \\
	0 & \star & 0  \\
	\star & \star & \star  \\
\end{pmatrix} \in \operatorname{SL}_3 \right\} 
\end{align*}
and correspond to the Weyl chambers on which the spherical Weyl group acts transitively by conjugation. A corresponding set of representatives of $\KW$ is given by 
	\begin{align*}
		&\begin{pmatrix}
			1 & 0 & 0 \\
			0 & 1 & 0 \\
			0 & 0 & 1 \\
		\end{pmatrix}, \quad
		\begin{pmatrix}
			0 & -1 & 0 \\
			1 & \phantom{-}0 & 0 \\
			0 & \phantom{-}0 & 1 \\
		\end{pmatrix}, \quad
		\begin{pmatrix}
			1 & 0 & \phantom{-}0 \\
			0 & 0 & -1 \\
			0 & 1 & \phantom{-}0 \\
		\end{pmatrix}, \\
		&\begin{pmatrix}
			0 & 0 & -1 \\
			0 & 1 & \phantom{-}0 \\
			1 & 0 & \phantom{-}0 \\
		\end{pmatrix}, \quad
		\begin{pmatrix}
			0 & 0 & 1 \\
			1 & 0 & 0 \\
			0 & 1 & 0 \\
		\end{pmatrix}, \quad
		\begin{pmatrix}
			0 & 1 & 0 \\
			0 & 0 & 1 \\
			1 & 0 & 0 \\
		\end{pmatrix}.
	\end{align*}

\section{Results about Lie algebras and split tori}\label{sec:Lie_algebra_results}
Let $\bG$ be a semisimple algebraic group defined over $\K \subseteq \R$, which is invariant under transposition. The Lie algebra $\frakg = T_e\bG$ is defined by finitely many linear equations with coefficients in $\K$ and we can therefore consider its $\K$-points $\frakg(\K) \subseteq \K^{n\times n}$, which is a $\K$-vector space and an algebraic (hence semialgebraic) set. Let $\frakg_\F$ be the semialgebraic extension of $\frakg(\K)$ for real closed fields $\F \supseteq \K$. We note that $\frakg_{\F} = \frakg(\F)$, but in what follows, we will use the semialgebraic point of view. For consistency we put $\frakg_{\K}:= \frakg(\K)$.

In this section, we recall facts about the real Lie group $\bG(\R)$ and its Lie algebra $\frakg_\R$. We sometimes reference Chapter 3 of Helgason's book \cite{Hel2} and chapter 6 of Knapp's book \cite{Kna}, a compact account of which can be found in \cite{Wis01}.

\subsection{Cartan involutions and Cartan decompositions}\label{sec:killing_involutions_decompositions}

Recall that the \emph{adjoint representation} of the Lie algebra $\frakg_\R$ 
$$
\operatorname{ad} := \operatorname{ad}_{\frakg_\R} := \operatorname{d}_e\!\operatorname{Ad}_{\bG(\R)} \colon \frakg_\R \to \mathfrak{gl}(\frakg_\R)
$$
is given by $X \mapsto [X,\cdot\,]$. The \emph{Killing form} $B$ is the bilinear form on $\frakg_\R$ defined by
$$
B(X,Y) := \operatorname{tr}\left(\operatorname{ad}(X)\circ\operatorname{ad}(Y)\right)
$$
for $X,Y \in \frakg_\R$.
A Lie algebra is called \emph{simple} if it is non-abelian and does not contain any non-zero proper ideals. 
A Lie algebra is called \emph{semisimple} if it is a direct product of simple Lie algebras. By Cartan's criterion, being semisimple is equivalent to having a non-degenerate Killing form.
 Since $\bG$ is semisimple, so is $\bG(\R)$ and $\frakg_\R$.
A Lie algebra automorphism $\theta \colon \frakg_\R \to \frakg_\R$ is called an \emph{involution} if $\theta^2 = \operatorname{Id}$. For any involution $\theta \colon \frakg_\R \to \frakg_\R$ we can define a bilinear form $B_\theta$ by
$$
B_\theta(X,Y) := -B(X,\theta Y)
$$ 
for $X,Y \in \frakg_\R$. If $B_\theta$ is positive-definite, then $\theta$ is called a \emph{Cartan involution}. The following is a technical result on Cartan involutions which is stated as an exercise in \cite{Hel2} and proven in \cite[Theorem 6.16]{Kna}.
\begin{lemma} \label{lem:aligntheta}
	Let $\frakg_\R$ be a real semisimple Lie algebra, $\theta$ a Cartan involution and $\sigma$ any involution on $\frakg_\R$. Then there exists an automorphism
	$
	\varphi \in \operatorname{Ad}(\bG(\R)) 
	$
	such that $\varphi \theta \varphi^{-1}$ commutes with $\sigma$.
\end{lemma}
An application of the preceding Lemma is that all Cartan involutions are conjugated, see \cite[Corollary 6.19]{Kna}.
\begin{theorem}\label{thm:cartan_inv_unique}
	Let $\frakg_\R$ be a real semisimple Lie algebra. Any two Cartan involutions of $\frakg_\R$ are conjugate via an element of $\operatorname{Ad}(\bG(\R))$.
\end{theorem}
\begin{proof}
	Let $\theta$ and $\theta'$ be two Cartan involutions of $\frakg_\R$. If $\theta$ and $\theta'$ commute, they have the same eigenspaces. We claim $E_{1}(\theta) = E_1(\theta')$ and $E_{-1}(\theta)=E_{-1}(\theta')$, since otherwise if for instance $\theta(X)=X$ and $\theta'(X)=-X$, then 
	$$
	B_{\theta'}(X,X) = -B(X,\theta'(X)) = - B(X,-X) = B(X,\theta(X)) = - B_{\theta}(X,X),
	$$
	but both $B_{\theta}(X,X)$ and $B_{\theta'}(X,X)$ should be positive. We conclude that if $\theta$ and $\theta'$ commute, then $\theta = \theta'$, since $\theta$ and $\theta'$ take the same values on their eigenspaces. If $\theta$ and $\theta'$ do not commute, then we can apply Lemma \ref{lem:aligntheta} to find $\varphi \in \operatorname{Ad}(\bG(\R))$ such that $\varphi \theta \varphi^{-1}$ commutes with $\theta'$, and therefore $\varphi \theta \varphi^{-1} = \theta'$ by the previous argument.
\end{proof}

A decomposition of $\frakg_\R$ as a direct sum $
\frakg_\R =  \frakk \oplus \frakp.
$
is called a \emph{Cartan decomposition} if
$$
[\frakk, \frakk] \subseteq \frakk, \quad [\frakk, \frakp] \subseteq \frakp , \quad [\frakp, \frakp] \subseteq \frakk
$$
and the Killing form $B$ is negative definite on $\frakk$ and positive definite on $\frakp$. There is a correspondence between Cartan involutions and Cartan decompositions. A Cartan involution $\theta$ defines a decomposition into eigenspaces $\frakk = E_{1}(\theta)$, $\frakp = E_{-1}(\theta)$ of $\frakg_\R$. The bracket relations can be checked using that $\theta$ commutes with the bracket operation. Since then $\operatorname{ad}(\frakk)\operatorname{ad}(\frakp)(\frakk) \subseteq \frakp$ and $\operatorname{ad}(\frakk)\operatorname{ad}(\frakp)(\frakp) \subseteq \frakk$, we have 
$$
B(\frakk, \frakp) = \operatorname{Tr}(\operatorname{ad}(\frakk)\operatorname{ad}(\frakp)) = 0
$$ 
and the decomposition is orthogonal. Since $\theta$ is a Cartan involution, $B_\theta$ is positive definite, and thus $B$ is negative definite on $\frakk$ and positive definite on $\frakp$. Starting from a Cartan decomposition $\frakg_\R = \frakk \oplus \frakp$, we can define the involution
$$
\theta \colon X \mapsto \left\{\begin{matrix}
	X & \text{ if }X \in \frakk \\
	-X & \text{ if }X \in \frakp
\end{matrix} \right.
$$  
which is compatible with the bracket operation and for which $B_\theta$ is positive definite. To refine the decomposition further, we need the following Lemma.
\begin{lemma}
	Let $\frakg_\R = \frakk \oplus \frakp$ be a Cartan decomposition with Cartan involution $\theta$ and let $X \in \frakp$. The map $\operatorname{ad}(X) \colon \frakg_\R \to \frakg_\R$ is symmetric with respect to the scalar product $B_\theta$.
\end{lemma}
\begin{proof}
	We have
	\begin{align*}
		B_\theta(\operatorname{ad}(X)(Y),Z) & =-B([X,Y],\theta(Z)) = B(Y,[X,Z]) = - B(Y,[\theta(X),\theta(Z)])\\
		& = -B(Y,\theta([X,Z])) = B_\theta(Y,\operatorname{ad}(X)Z)
	\end{align*}  
for all $X \in \frakp$ and $Y,Z \in \frakg_\R$. 
\end{proof}
Let now $\frakg_\R = \frakk \oplus \frakp$ be a fixed Cartan decomposition with associated Cartan involution $\theta$. Let $\fraka \subseteq \frakp$ be a maximal abelian subspace. We can define the \emph{real rank of $\bG(\R)$}, $\operatorname{rank}_\R(\bG(\R)) := \operatorname{dim}(\fraka)$, which is independent of $\frakp$ as a consequence of Theorem \ref{thm:cartan_inv_unique} and independent of $\fraka$ since any two maximal abelian subspaces are conjugate to each other \cite[Theorem 6.51]{Kna}. The set $\{\operatorname{ad}(H) \colon H \in\fraka\}$ consists of symmetric, hence diagonalizable linear maps. Since they all commute, they are simultaneously diagonalizable. This results in the decomposition
$$
\frakg_\R = \frakg_0 \oplus\bigoplus_{\alpha \in \Sigma}\frakg_\alpha 
$$  
where for every $\alpha \in \fraka^\star$ in the dual space $\fraka^\star$ of $\fraka$
$$
\frakg_\alpha = \left\{ X \in \frakg_\R \colon [H,X] = \alpha(H) \cdot X \text{ for all } H \in \fraka  \right\}
$$
and where $\Sigma = \{\alpha \in \fraka^\star \colon \alpha \neq 0 \text{ and } \frakg_\alpha \neq 0\}$. The elements of $\Sigma$ are called \emph{restricted roots} and $\frakg_\alpha$ their associated \emph{restricted root spaces}. Note that in contrast to the decomposition of complex Lie algebras, the root spaces may not be one-dimensional\footnote{In the real semisimple Lie algebra $\mathfrak{sl}_2(\C)$ we have $\operatorname{dim}(\frakg_\alpha) = 2$.}. In Section \ref{sec:alg_root_system} we defined the root system of an algebraic group with a maximal split torus. In Section \ref{sec:compatibility} we will show that these two approaches give the same root spaces and root systems.
\begin{proposition}(\cite[Proposition 6.40]{Kna})\label{prop:root_decomp} 
	The \emph{restricted root space decomposition} is an orthogonal direct sum with respect to $B_\theta$ and satisfies for all $\alpha,\beta \in \Sigma$
	\begin{itemize}
		\item [(i)] $
		[\frakg_{\alpha}, \frakg_{\beta}] \subseteq \frakg_{\alpha + \beta}$,
		\item [(ii)]$\theta (\frakg_\alpha) = \frakg_{-\alpha}$ and
		\item [(iii)]$\frakg_0 = \fraka \oplus \mathfrak{z}_\frakk(\fraka)$, where $ \mathfrak{z}_\frakk(\fraka) = \left\{X\in \frakk \colon [H,X] = 0  \text{ for all } H \in \fraka \right\}$.
	\end{itemize}
\end{proposition}

The inner product $B_{\theta}$ may be restricted to $\fraka$ and used to set up an isomorphism $\fraka^{\star} \cong \fraka$ which turns $\fraka^{\star}$ into a Euclidean vector space.

\begin{theorem}(\cite[Corollary 6.53]{Kna})\label{thm:sigma_root}
	The set of roots $(\Sigma,\fraka^\star)$ is a crystallographic root system\footnote{$\Sigma$ may not be reduced, for example when $G= \operatorname{SU}(2,1)$.}.
\end{theorem}
By choosing an ordered basis of the root system $\Sigma$, we can define the set of positive roots $\Sigma_{>0}$. 
Then
$$
\frakn = \bigoplus_{\alpha \in \Sigma_{>0}} \frakg_\alpha
$$
is a nilpotent subalgebra of $\frakg_\R$. We have the following properties. 
\begin{lemma}\label{lem:kan_form}(\cite[Lemma 6.45]{Kna})
	Let $\frakg_\R$ be a real semisimple Lie algebra. There exists a basis of $\frakg_\R$ such that the matrices representing $\operatorname{ad}(\frakg_\R)$ have the following properties
	\begin{enumerate}
		\item The matrices of $\operatorname{ad}(\frakk)$ are skew-symmetric.
		\item The matrices of $\operatorname{ad}(\fraka)$ are diagonal.
		\item The matrices of $\operatorname{ad}(\frakn)$ are upper triangular with $0$ on the diagonal.
	\end{enumerate}
\end{lemma}
We also have the Iwasawa decomposition on the level of Lie algebras. 
\begin{theorem}(\cite[Proposition 6.43]{Kna})
	Let $\frakg_\R$ be a semisimple Lie algebra. Then $\frakg_\R = \frakk \oplus \fraka \oplus \frakn$ is a direct sum.
\end{theorem}

The following lemma about the Lie subalgebra  
$$
\frakl := ( \frakg_\alpha \oplus \frakg_{2\alpha} ) \oplus (\frakg_{-\alpha} \oplus \frakg_{-2\alpha} ) \oplus ( [\frakg_\alpha , \frakg_{-\alpha}] + [\frakg_{2\alpha}, \frakg_{-2\alpha}] )
$$
for some $\alpha \in \Sigma$ will be useful when we consider certain rank 1 subgroups in Section \ref{sec:rank1}. Note that by Cartan's criterion, $\frakl$ is semisimple, as the Killing form is the restriction of the definite Killing form of $\frakg$. Then $\frakk \cap \frakl \oplus \frakp \cap \frakl$ is a Cartan decomposition of $\frakl$.
\begin{lemma}\label{lem:levi_algebra} 
	Let $\alpha \in \Sigma$ and $X\in \frakg_\alpha  \setminus \{0\}$. Then the real rank $\operatorname{rank}_\R$ of $\frakl$ is one. A maximal abelian subspace of the symmetric part of $\frakl$ is given by $\fraka \cap \frakl = \langle [X,\theta(X)]\rangle$.
\end{lemma}
\begin{proof}
	We use the fact that $B_\theta$ gives us a scalar product on $\fraka$ allowing us to identify $\fraka \cong \fraka^\star$, sending $\alpha \in \Sigma$ to $H_\alpha$ defined by $\alpha(H)=B_\theta(H_\alpha,H)$ for all $H \in \fraka$. Let $X\in \frakg_\alpha, Y \in \frakg_{-\alpha}$ and $H \in \fraka$. Then
	\begin{align*}
		B_\theta([X,Y],H) & = -B([X,Y],\theta(H)) = B([X,Y], H) = - B(Y,[X,H])  \\
		&= B(Y,[H,X]) = B(Y,\alpha(H)X) = \alpha(H)B(Y,X) \\
		&=  B_\theta(H_\alpha, H)B(X,Y) = B_\theta(B(X,Y) \cdot H_\alpha ,H),
	\end{align*}
	where we used that $B$ is $\operatorname{ad}(X)$-invariant.
	The element 
	$$
	W = [X,Y]- B(X,Y) \cdot H_\alpha  \in \frakg_0
	$$
	satisfies $B_\theta(W,H)=0$ for all $H \in \fraka$, hence lies in $\frakg_0$ perpendicular to $\fraka$, hence $W \in \mathfrak{z}_\frakk(\fraka) \subseteq \frakk$ by Proposition \ref{prop:root_decomp}. Similarly, one can show that for any $X' \in \frakg_{2\alpha}$ and $Y' \in \frakg_{-2\alpha}$, $[X',Y'] =W' + 2B(X',Y') \cdot H_\alpha$ for some $W' \in \mathfrak{z}_\frakk(\fraka)$. Any element $Z\in \fraka \cap \frakl$  lies in $[\frakg_\alpha , \frakg_{-\alpha} ] \oplus [ \frakg_{2\alpha} , \frakg_{-2\alpha}]$ since $\fraka \subseteq \frakg_0$ is orthogonal to $\frakg_\alpha \oplus \frakg_{2\alpha} \oplus \frakg_{-\alpha} \oplus \frakg_{-2\alpha}$. Therefore, $Z$ can be written as $Z=W+W'+ (B(X,Y)+2B(X',Y')) \cdot H_\alpha  \in \mathfrak{z}_\frakk(\fraka) \oplus \fraka$, hence $W+W'=0$ and $Z \in \langle H_\alpha \rangle$.
	
	If $Y=\theta(X)$, then 
	$$
	\theta(W)=\theta([X,\theta(X)]-B(X,Y)\cdot H_\alpha ) = [\theta(X),X]-\theta(H_\alpha) B(X,Y) = - W
	$$ 
	which implies $W \in \frakp \cap \frakk$, hence $W=0$. We have shown that $[X,\theta(X)] \in \fraka \cap \frakl \subseteq \langle H_\alpha \rangle$. Since $B$ is definite, $[X,\theta(X)] \neq 0$ and $\fraka \cap \frakl = \langle [ X, \theta(X)] \rangle$.
	
	The map $\theta|_{\frakl}$ is a Cartan-involution with Cartan decomposition $\frakl = \frakk\cap \frakl \oplus \frakp \cap \frakl$. There is a maximal abelian subspace of $\frakp \cap \frakl$ contained in $\fraka$, which is equal to $\fraka \cap \frakl = \langle [X,\theta(X)] \rangle$ by the above. The dimension of a maximal abelian subspace of $\frakp \cap \frakl$ is exactly the rank of $\frakl$ and it is one. 
\end{proof}

The following is a special case of the Jacobson-Morozov Lemma in the literature. Note that this Lemma does not work for general $X\in \frakg_\alpha \oplus \frakg_{2\alpha}$.

\begin{lemma}\label{lem:JM_basic}
	Let $\alpha \in \Sigma$ and $X \in \frakg_\alpha \setminus \{0\}$. Then there is a $Y \in \frakg_{-\alpha} \setminus \{0\}$ and $H \in \fraka$ such that $(X,Y,H)$ forms a $\mathfrak{sl}(2)$-triplet, i.e.
	$$
	[X,Y] = H, \quad [H,X] = 2X,  \quad [H,Y]=2Y.
	$$
\end{lemma}
\begin{proof}
	Let $H_\alpha \in \fraka$ be as in the proof of Lemma \ref{lem:levi_algebra} defined by $\alpha(H) = B_{\theta}(H_\alpha,H)$ for all $H \in \fraka$. Given $X \in \frakg_\alpha$, let
	$$
	Y := \frac{-2}{B_{\theta}(X,X) \cdot \alpha(H_\alpha)} \theta(X) \in \frakg_{-\alpha}
	$$
	and $H:= [X,Y]$. In the proof of Lemma \ref{lem:levi_algebra} we saw that
	$$
	[X,\theta(X)] = B(X,\theta(X))\cdot H_\alpha = -B_{\theta}(X,X) \cdot H_\alpha,
	$$ 
	whence $H \in \fraka$ and moreover
	\begin{align*}
		[H,X] &= \left[\left[X,\frac{-2}{B_{\theta}(X,X)\alpha(H_\alpha)} \theta(X)\right],X\right] \\
		&= \frac{2}{\alpha(H_\alpha)} \left[H_\alpha,X\right] = \frac{2}{\alpha(H_\alpha)} \alpha(H_\alpha) X = 2X
	\end{align*}
	and similarly $[H,Y] = -2Y$.
\end{proof}

\subsubsection{Examples}
Since $\bG(\R)$ is invariant under transposition, $\sigma \colon g \mapsto (g^{-1})\tran$ is an involution of the Lie group whose differential $\theta := \operatorname{d}_e\! \sigma \colon X \mapsto -X\tran$ is an involution on the Lie algebra $\frakg_\R$ which decomposes $\frakg_\R= \frakp \oplus \frakk$ into a symmetric part $\frakp$ and an skew-symmetric part $\frakk$. One can check that
$$
[\frakk, \frakk] \subseteq \frakk, \quad [\frakk, \frakp] \subseteq \frakp  \quad \text{ and } \quad [\frakp, \frakp] \subseteq \frakk
$$
and that the Killing form is negative definite on $\frakk$ and positive definite on $\frakp$. Thus $\frakg_\R = \frakk \oplus \frakp$ is a Cartan decomposition and $\theta$ is a Cartan involution.

\subsection{Real Cartan subalgebras and $\R$-split subalgebras}\label{sec:real_cartan_split}

In this section we discuss Cartan subalgebras of a finite dimensional semisimple real Lie algebra $\frakg$. These results apply in particular to the Lie algebra $\frakg_\R$ of the real Lie group $\bG(\R)$. 

An abelian subalgebra $\frakh \subseteq \frakg$ is called a \emph{Cartan subalgebra}\footnote{More generally, a Cartan subalgebra is defined to be a nilpotent self-normalizing subalgebra. In our context, we restrict to semisimple Lie algebras over fields of characteristic $0$, in which case Cartan subalgebras are abelian. } 
 if 
$$
\frakh = \operatorname{Nor}_{\frakg}(\frakh) := \left\{ X \in \frakg \colon [X,Y] \in \frakh \text{ for all } Y \in \frakh \right\}.
$$
An abelian subalgebra $\fraka \subseteq \frakg$ is called \emph{$\R$-split} if $\operatorname{ad}(X)\colon \frakg \to\frakg$ is diagonalizable over $\R$ for every $X \in \fraka$. Let $r_\R(\frakg)$ be the maximal dimension of such an $\R$-split abelian subalgebra and denote
$$
V(\frakg) := \{ \fraka \subseteq \frakg \colon \text{$\fraka$ is an $\R$-split abelian subalgebra with $\dim(\fraka) = r_\R(\frakg)$}\}.
$$ 
A \emph{maximally $\R$-split Cartan subalgebra}
is a Cartan subalgebra containing an element of $V(\frakg_\R)$ as a subset. Let
$$
\mathcal{C} (\frakg) = \left\{  \frakh \subseteq \frakg \colon \frakh \text{ is a maximally $\R$-split Cartan subalgebra } \right\}.
$$

Knapp \cite{Kna} uses slightly different definitions and names for these objects. We will show here that they coincide. We start by showing that our notion of Cartan subalgebra coincides with the notion in \cite{Kna}.

\begin{lemma}\label{lem:defs_Cartan_subalgebra}
	Let $\mathfrak{h}$ be a subalgebra of a finite dimensional real semisimple Lie algebra $\frakg$. Then $\frakh$ is a Cartan subalgebra if and only if it satisfies one of the following conditions
	\begin{enumerate}
		\item [(i)] $\frakh$ is abelian and $\operatorname{Nor}_\frakg (\frakh) = \frakh$.
		\item [(ii)] The complexification $\frakh_\C := \frakh \oplus i \frakh$ is abelian and $\operatorname{Nor}_{\frakg_\C} (\frakh_\C) = \frakh_\C$.
		\item [(iii)] The complexification $\frakh_\C $ is nilpotent and $\operatorname{Nor}_{\frakg_\C} (\frakh_\C) = \frakh_\C$.
		\item [(iv)] The complexification $\frakh_\C $ is nilpotent and 
		$$
		\frakh_\C = \left\{ X \in \frakg_\C \colon \exists n \in \N ,\ \forall H \in \frakh_\C ,\ \operatorname{ad}(H)^n X = 0  \right\}.
		$$
		\item [(v)] The complexification $\frakh_\C$ is maximal abelian and the subset $\operatorname{ad}(\frakh_\C) \subseteq \mathfrak{gl}(\frakg_\C)$ is simultaneously diagonalizable.
	\end{enumerate}
\end{lemma}
\begin{proof}
	Notion (i) is our definition of Cartan subalgebra. For any $X = X_1 + i X_2 \in \frakg_\C$ and $Y= Y_1+iY_2 \in \frakg_\C$, we have
	$$
	\operatorname{ad}(Y)X  = [Y_1+iY_2 , X_1 + iX_2] = [Y_1,X_1] - [Y_2,X_2] +i ([Y_1, X_2] + [Y_2,X_1]).
	$$  
	From this formula we can deduce that a subalgebra of a real Lie algebra is abelian if and only if its complexification is abelian. It also follows that for any real subalgebra $\frakh \subseteq \frakg$ we have
	$$
	\operatorname{Nor}_{\frakg}(\frakh) = \frakh \iff \operatorname{Nor}_{\frakg_\C}(\frakh_\C) = \frakh_\C. 
	$$
	It follows that notions (i) and (ii) coincide. For (ii) implies (iii), it suffices to notice that every abelian subalgebra is nilpotent, the converse uses semisimplicity and is given in \cite[Proposition 2.10]{Kna}. Condition (iv) is the definition used in \cite{Kna} and the equivalence of (iii) and (iv) is given by \cite[Proposition 2.7]{Kna}. The equivalence of (iv) and (v) is given by \cite[Corollary 2.13]{Kna}.
\end{proof}

A subset $\frakh\subseteq \frakg$ is called \emph{$\theta$-stable} if there is a Cartan involution $\theta$ such that $\theta(\frakh) = \frakh$. If $\theta$ is given by the context, $\theta$-stable refers to that particular Cartan involution. 
\begin{proposition}\label{prop:Kna_real_Cartan_conj}(\cite[Prop 6.59]{Kna})
	Let $\frakg$ be the Lie algebra of a real semisimple Lie group $G$ and $\operatorname{Ad}(G)^\circ$ the connected component of $\operatorname{Ad}(G) \subseteq \operatorname{GL}(\frakg)$ in the Lie group topology. Any Cartan subalgebra $\frakh$ is conjugate via $\operatorname{Ad}(G)^\circ$ to a $\theta$-stable Cartan subalgebra.
\end{proposition}
If $\frakg = \frakp \oplus \frakk$ is the Cartan decomposition associated to a Cartan involution $\theta$, and $\frakh$ is a $\theta$-stable Cartan subalgebra, then $\fraka := \frakh \cap \frakp$ and $\frakt := \frakh \cap \frakk$ are $\theta$-stable and $\frakh = \fraka \oplus \frakt$. A Cartan subalgebra is called \emph{maximally noncompact} if $\dim(\fraka)$ is maximal. 
While all complex Cartan subalgebras are conjugated \cite[Theorem 2.15]{Kna}, for real Cartan subalgebras, the dimensions of the spaces $\fraka$ and $\frakt$ are preserved.
\begin{proposition}\label{prop:Kna_real_theta_Cartan_conj}(\cite[Prop 6.61]{Kna}) Let $\frakg$ be the Lie algebra of a semisimple real Lie group $G$ and let $K$ be a subgroup with Lie algebra $\frakk$ where $\frakg = \frakp \oplus \frakk$ is the Cartan decomposition associated to a Cartan involution $\theta$. 
	
	All maximally noncompact $\theta$-stable Cartan subalgebras are conjugate under $\operatorname{Ad}(K)^\circ$.
\end{proposition}
Together, Propositions \ref{prop:Kna_real_Cartan_conj} and \ref{prop:Kna_real_theta_Cartan_conj} allow us to extend our definition of maximally noncompact $\theta$-stable Cartan subalgebras to all Cartan subalgebras, by stipulating that a Cartan subalgebra is called \emph{maximally noncompact} if it is conjugated to a $\theta$-stable maximally noncompact Cartan subalgebra.
We will now show that the set of maximally $\R$-split Cartan subalgebras $\mathcal{C}(\frakg)$ coincides with the set of maximally noncompact Cartan subalgebras.

\begin{proposition}\label{prop:defs_max_noncompact}
	For any maximally $\R$-split subalgebra $\fraka \subseteq \frakg$ of a finite dimensional real semisimple Lie algebra $\frakg$, there is a maximally noncompact Cartan subalgebra $\frakh \supseteq \fraka$. Every maximally noncompact Cartan subalgebra contains a maximally $\R$-split subalgebra $\fraka$ and if $\frakh$ is $\theta$-stable, then $\fraka = \frakh \cap \frakp$ and $\frakh = \fraka \oplus \frakt$ in the decomposition above and $r_\R(\frakg) = \dim (\fraka)$.
\end{proposition}
\begin{proof}
	We introduce a few concepts. If $k$ is a field and $\frakg$ is a semisimple $k$-Lie algebra, then an abelian subalgebra $\frakh$ is called \emph{toral} if $\operatorname{ad}(\frakh) \subseteq \mathfrak{gl}(\frakg)$ consists of linear maps that are diagonalizable over the algebraic closure of $k$. If all the elements of $\operatorname{ad}(\frakh)$ are diagonalizable over $k$ itself, then $\frakh$ is called \emph{$k$-split toral}. If moreover $\frakh$ is maximal among all $k$-split toral subalgebras, $\frakh$ is called \emph{maximal $k$-split toral}.
	
	Let now $\frakg$ be a finite dimensional real semisimple Lie algebra. Let $\fraka \in V(\frakg)$, in our terminology $\fraka$ is maximal $\R$-split toral. Let $\frakh$ be maximal toral containing $\fraka$, which exists since $\frakg$ is finite dimensional. We want to show that $\frakh$ is a Cartan subalgebra and use characterization (v) of Lemma \ref{lem:defs_Cartan_subalgebra}. Since $\frakh$ is abelian, so is its complexification $\frakh_\C$. Since $\frakh$ is $\C$-split, all elements of $\operatorname{ad}(\frakh_\C)=\operatorname{ad}(\frakh) \oplus i \operatorname{ad}(\frakh)$ are diagonalizable over $\C$. This means that $\frakh_\C$ is a toral subalgebra of $\frakg_\C$ and by a general linear algebra fact, $\operatorname{ad}(\frakh_\C)$ is simultaneously diagonalizable. The complexification $\frakh_\C$ is also maximal abelian, since for all $X=X_1+iX_2 \in \frakg_\C$, if $[X,\frakh_\C] =0$, then we have in particular for $H \in \frakh$
	$$
	0 = [X,H] = [X_1+iX_2, H] = [X_1, H] + i [X_2, H] ,
	$$
	hence $X_1,X_2 \in \frakh$, so $X \in \frakh_\C$. This shows that $\frakh$ is a Cartan subalgebra by characterization (v).
	
	By Proposition \ref{prop:Kna_real_Cartan_conj}, there is a Cartan involution $\theta$ and a $g\in G_\R$ such that $\operatorname{Ad}(g)(\frakh)$ is a $\theta$-invariant Cartan subalgebra with decomposition $\operatorname{Ad}(g)(\frakh) = \fraka' \oplus \frakt'$ as described before. Since $\fraka' \subseteq \frakp $, it is $\R$-split. In fact, for $X=X_{\fraka'} + X_{\frakt'} \in \fraka' \oplus \frakt'$, $\operatorname{ad}(X)$ is only diagonalizable over $\R$ if $X_{\frakt'}=0$. Since diagonalizability is preserved under conjugation, $\operatorname{Ad}(\fraka) \subseteq \fraka'$ and by maximality of $\fraka$, we have $\operatorname{Ad}(\fraka) = \fraka'$. Thus $\frakh$ is maximally noncompact. 
	
	If we instead start with a maximally noncompact Cartan subalgebra $\frakh \subseteq \frakg$, we similarly obtain (possibly using a conjugation to a $\theta$-stable one) an $\R$-split subalgebra $\fraka$ with $\frakh = \fraka \oplus \frakt$. A priori $\fraka$ is maximal $\R$-split as a subalgebra of $\frakh$, but if $\fraka$ was contained in a larger subalgebra $\fraka'$ maximal $\R$-split in $\frakg$, then the above construction would result in a Cartan subalgebra $\frakh' = \fraka' \oplus \frakt'$ with $\dim(\fraka')>\dim(\fraka)$, which is impossible, since we assumed $\frakh$ to be maximally noncompact, which means $\dim(\fraka)$ is maximal. By definition, $r_\R(\frakg)$ is the maximal dimension of an $\R$-split abelian subalgebra, so $r_\R(\frakg) = \dim(\fraka)$.
\end{proof}

We now turn back to the Lie algebra $\frakg_\R$ of the $\R$-points $\bG(\R)$ of a semisimple algebraic group $\bG$ and collect a few properties of $V(\frakg_\R)$ and $\mathcal{C}(\frakg_\R)$. 
\begin{lemma}\label{lem:basic_facts} Recall that for the algebraic group $\bG$, $ \operatorname{rank}_\R(\bG)$ is the maximal dimension of any abelian subspace of $\frakp \subseteq \frakg_\R$, where $\frakg_\R = \frakp \oplus \frakk$ is the Cartan decomposition associated to some Cartan involution $\theta$. We have
	\begin{itemize}
		\item [(1)] $r_\R(\frakg_\R) = \operatorname{rank}_\R(\bG)$.
		\item [(2)] $\bG(\R)$ acts transitively on $\mathcal{C}(\frakg_\R)$ and $ V(\frakg_\R)$.
		\item [(3)] $\mathcal{C}(\frakg_\R)$ contains a $\theta$-stable element. 
	\end{itemize}
\end{lemma}
\begin{proof}
	A maximally $\R$-split Cartan subalgebra, or by Proposition \ref{prop:defs_max_noncompact} a maximally noncompact Cartan subalgebra can be obtained by starting with a Cartan involution $\theta$, taking $\fraka \subseteq \frakp$ a maximal abelian subspace of $\frakp$ and taking $\frakt \subseteq \operatorname{Cen}_{\frakk}(\fraka)$ a maximal abelian subspace of $\operatorname{Cen}_{\frakk}(\fraka)$. Then
	$\frakh = \fraka \oplus \frakt$ is a $\theta$-stable maximally noncompact Cartan subalgebra, see \cite[Proposition 6.47]{Kna}. This shows (3).
	
	By Proposition \ref{prop:defs_max_noncompact}, $r_\R(\frakg_\R) = \dim(\fraka) $ and since $\fraka$ is a maximal abelian subspace of $\frakp$, $r_\R(\frakg_\R) = \operatorname{rank}_\R(G)$, showing (1).
	
	By Propositions \ref{prop:Kna_real_Cartan_conj} and \ref{prop:Kna_real_theta_Cartan_conj}, $\operatorname{Ad}(\bG(\R))$ and thus $\bG(\R)$ act transitively on $\mathcal{C}(\frakg_\R)$. Next, let $\fraka, \fraka'\in V(\frakg_\R)$. By Proposition \ref{prop:defs_max_noncompact}, there are maximally noncompact Cartan subalgebras $\frakh, \frakh'$ such that $\fraka \subseteq \frakh$ and $\fraka' \subseteq \frakh'$. By Proposition \ref{prop:Kna_real_theta_Cartan_conj}, both $\frakh$ and $\frakh'$ are conjugated to a $\theta$-stable maximally noncompact Cartan subalgebra $\frakh''$. By Proposition \ref{prop:defs_max_noncompact}, the corresponding conjugates of $\fraka$ and $\fraka'$ coincide with the intersection $\frakh'' \cap \frakp$, hence with each other. This shows that $\operatorname{Ad}(\bG(\R))$ acts transitively on $V(\frakg_\R)$ and concludes the proof of (2).
\end{proof}

\subsection{Lie Algebras over real closed fields}

Let $\K$ be a subfield of $\R$ and $\F$ be a real closed fields with $\K \subseteq \F$. In this section we additionally assume that $\K$ is real closed. Let $\frakg_\K  \subseteq \K^{n\times n}$ be the $\K$-points of the Lie algebra of a self-adjoint ($g \in \bG$ implies $g\tran \in \bG$) semisimple algebraic group $\bG$ defined over $\K$. Let $\frakg_\R$ and $\frakg_\F$ be the semialgebraic extensions. The definitions of Section \ref{sec:real_cartan_split} apply also to $\frakg_\F$:

An abelian subalgebra $\frakh \subseteq \frakg_\F$ is called \emph{Cartan subalgebra} if $\frakh = \operatorname{Nor}_{\frakg_\F}(\frakh)$. An abelian subalgebra $\fraka \subseteq \frakg_\F$ is called \emph{$\F$-split} if $\operatorname{ad}(X) \colon \frakg_\F \to \frakg_\F$ is diagonalizable over $\F$ for every $X \in \fraka$. Let $r_{\F}(\frakg_\F)$ be the maximal dimension of such an $\F$-split abelian subalgebra and denote by $V(\frakg_\F)$ the set of all $\F$-split abelian subalgebras with $\dim(\fraka) = r_\F (\frakg_\F)$. A \emph{maximally $\F$-split Cartan subalgebra} is a Cartan subalgebra containing an element of $V(\frakg_\F)$ as a subset. We denote by $\mathcal{C}(\frakg_\F)$ the set of all maximally $\F$-split Cartan subalgebras. We will now use the transfer principle to relate $\F$-split algebras to the real subalgebras studied in the previous section. In the following Lemma we fix the Cartan-involution $\theta \colon X \mapsto -X\tran$ that exists since $\bG$ is self-adjoint.

\begin{lemma} \label{lem:3}
	Whenever $\F$ and $\K$ are two real closed fields with $\K \subseteq \F$ and $\K \subseteq \R$, then
	\begin{itemize}
		\item [(1)] $r_\K(\frakg_\K) = r_\R (\frakg_\R) =  r_\F (\frakg_\F)$.
		\item [(2)] $\bG(\K)$ acts transitively on $\mathcal{C}(\frakg_\K)$ and $V(\frakg_\K)$ and those two sets are non-empty.
		\item [(3)] \label{lem:3.3}$\mathcal{C}(\frakg_\K)$ contains a $\theta$-stable element, i.e. there is $\frakh \in \mathcal{C}(\frakg_\K)$ such that $\theta(\frakh) = \frakh$.
		\item [(4)] Let $\fraka$ and $\frakh$ be subalgebras of $\frakg_\K$. Then 
		\begin{align*}
			\fraka \in V(\frakg_\K) & \iff \fraka_\R \in V(\frakg_\R) \iff \fraka_\F \in V(\frakg_\F)
	\\
	\frakh \in \mathcal{C}(\frakg_\K) &\iff \frakh_\R \in \mathcal{C}(\frakg_\R) \iff \frakh_\F \in \mathcal{C}(\frakg_\F).	
	\end{align*}
	\end{itemize}
\end{lemma}

\begin{proof}
	Our first goal is to see $V(\frakg_\K)$ as a semi-algebraic set. Let $k$ be any real closed field containing $\K$. Let $\eL$ be a semisimple $k$-algebra, such that $\eL$ is also a semi-algebraic set defined over $\K$. Let $r \leq n := \operatorname{dim}_k(\eL)$. The set $\operatorname{Grass}_r(\eL)$ can be identified with an algebraic subset of $k^{n\times n}$, namely
	$$
	\varphi \colon \operatorname{Grass}_r(\eL) \to \left\{ A \in k^{n\times n} \colon A\tran = A,\  A^2=A,\ \operatorname{Tr}(A) = r \right\}
	$$
	sends the $k$-dimensional subspace $V$ to the orthogonal projection $\eL \to V$, see \cite[Theorem 3.4.4]{BCR}. Moreover, for any $A \in \varphi(\operatorname{Grass}_r(\eL))$ we have the description
	$$
	\varphi^{-1}(A) = \{ v \in \eL \colon Av = v \}.
	$$
	An abelian subalgebra $\fraka \subseteq \eL$ is $k$-split exactly when $\operatorname{ad}(x)$ is diagonalizable over $k$ for all $x \in \fraka$. It is enough to require that for a basis $\{v_1, \ldots , v_r\}$ of $\fraka$, the maps $\operatorname{ad}(v_i) \colon \eL \to \eL$ are diagonalizable over $k$. A linear map with matrix $M=(M_{ij})$ is diagonalizable over $k$ if and only if its characteristic polynomial decomposes into linear factors, which can be written as a first-order formula:
	$$
	f(M) = \exists x_1, \ldots , x_n \colon \forall X \colon  \det(M_{ij} - X\delta_{ij} ) = \prod_{i=1}^n (X-x_i).
	$$
	We can write the statement ``$\fraka = \varphi^{-1}(A)$ is a $k$-split abelian subalgebra'' as a first-order formula
	\begin{align*}
		& \exists v_1, \ldots , v_n \in \eL \colon \left( \forall a_1 , \ldots ,a_n \colon \sum_{i=1}^n a_i v_i = 0 \to a_1 = 0 \wedge \ldots \wedge a_n = 0\right) \wedge\\
		&\bigwedge_{i=1}^r Av_i = v_i \ \wedge  \bigwedge_{i,j=1}^r	[v_i,v_j] = 0 \ \wedge \\
		&
		\bigwedge_{\ell = 1}^r \left( \exists M_{11}, M_{12}, \ldots , M_{nn} \colon \bigwedge_{i=1}^n [v_\ell , v_i] = \sum_{j=1}^n M_{ij} e_j \, \rightarrow f(M) \right), 
	\end{align*}
	in words: there exists a basis of $\eL$ whose first $r$ vectors form a basis of $\fraka = \varphi^{-1}(A)$, such that $\fraka$ is abelian and for every $\ell \in \{1,\ldots, r\}$ $\operatorname{ad}(v_\ell)$ with matrix $M$ in this basis is diagonalizable over $k$.
	
	By quantifier elimination we get an equivalent first-order statement $g(A)$ without quantifiers and we can write 
	$$
	\varphi(V(\eL)_r) = \{A \in \varphi(\operatorname{Grass}_r(\eL)) \colon g(A) \}
	$$
	as a semi-algebraic set defined by polynomials with coefficients in $k$, where $V(\eL)_r$ denotes the set of all $k$-split abelian subalgebras of dimension $r$.

	Now for $k = \R$ we know that $V(\frakg_\R) = V(\frakg_\R)_{r_\R(\frakg_\R)}$ is non-empty, but $V(\frakg_\R)_{r}$ is empty for any $r > r_\R(\frakg_\R)$. We consider the first-order formula
	$$
	\exists A \in  \varphi( V(\frakg_k)_r ),
	$$
	which is defined over $\K$ since $\frakg_\R$ and hence the brackets are defined over $\K$. Since the formula is satisfied for $k=\R$ and $r=r_\R(\frakg_\R)$, we conclude by the transfer principle that it also holds for $k=\K$ and $k=\F$, i.e. $V(\frakg_\K) = V(\frakg)_{r_\R(\frakg_\R)} \neq \emptyset$ and $V(\frakg_\F) = V(\frakg_\F)_{r_\R(\frakg_\R)} \neq \emptyset$. For any $r > r_\R(\frakg_\R)$, we know that the formula is not satisfied for $k=\R$, therefore by the transfer principle it is also not satisfied for $k=\K$ and $k=\F$, i.e. $V(\frakg_\K)_{r} = \emptyset$ and $V(\frakg_\F)_{r} = \emptyset$ for any  $r > r_\R(\frakg_\R)$. It follows that $r_\F(\frakg_\F) = r_\R(\frakg_\R) = r_\K (\frakg_\K)$, which is statement (1) of the Lemma we are proving.
	
	We describe 
	\begin{align*}
		\mathcal{C}(\frakg_k) = \left\{ \frakh \in \operatorname{Grass}_{r_{k}(\frakg_k)}(\frakg_k) \colon 
		\begin{matrix}
			[\frakh , \frakh] = 0 , \ \operatorname{Nor}_{\frakg_k}(\frakh) = \frakh, \\
			\exists \fraka \in V(\frakg_k) \colon \fraka \subseteq \frakh 
		\end{matrix}
		\right\}
	\end{align*}
	similarly. Let $v_1, \ldots , v_n$ again be a basis of a semisimple $k$-algebra $\eL$ such that $v_1 , \ldots , v_r$ is a basis of a subalgebra $\fraka \subseteq \eL$. We have $\operatorname{Nor}_{\eL}(\fraka) = \fraka$ whenever $[v_i,v_j] \in \fraka \ \forall i \in \{1, \ldots , n\}, j \in \{1, \ldots , r\}$, i.e. given $\fraka = \varphi^{-1}(A)$ we have
	\begin{align*}
		h(A) \colon \quad &
		\bigwedge_{i=1}^n \bigwedge_{j=1}^r A[v_i , v_j] = [v_i,v_j].
	\end{align*}
	We have that $\varphi^{-1}(A) \in \mathcal{C}(\eL)$ if and only if the following first-order formula holds
	\begin{align*}
		& \exists v_1, \ldots , v_n \in \eL \colon \left( \forall a_1 , \ldots ,a_n \colon \sum_{i=1}^n a_i v_i = 0 \to \bigwedge_{i=1}^n  a_i = 0 \right) \wedge\\
		& \bigwedge_{i=1}^r Av_i = v_i \ \ \wedge\ \  
		\bigwedge_{i,j=1}^r [v_i,v_j] = 0   \ \ \wedge \ \ h(A) \ \ \wedge \\
		&  \exists  B \in V(\eL) \colon \forall v \in \eL \colon Bv=v \to Av = v. 
	\end{align*}
	Again we use quantifier elimination to to get an equivalent statement $s(A)$ without quantifiers. Thus, $\mathcal{C}(\eL)$ can be identified with the semialgebraic set
	$$
	\varphi(\mathcal{C}(\eL)) = \{ A  \in \varphi(\operatorname{Grass}_{r_k(\eL)} (\eL)) \colon s(A)\}.
	$$
	From the theory of real Lie groups we know by Lemma \ref{lem:basic_facts}(3) that $\mathcal{C}(\frakg_{\R})$ is non-empty and since $\exists A \in \varphi(\mathcal{C}(\frakg_k))$ is a first-order statement, defined over $\K$, we can use the transfer principle to conclude that $\mathcal{C}(\frakg_\K)$ and $\mathcal{C}(\frakg_{\F})$ are also non-empty. Statement (4) follows from the semi-algebraic description of the sets $V(\frakg_\K)$ and $\mathcal{C}(\frakg_\K)$ and the transfer principle.
	
	The group $\bG(k)$ acts by conjugation on $\operatorname{Grass}_r(\frakg_k)$. The corresponding action on $\varphi(\operatorname{Grass}_r(\frakg_k))$ is given by $g.A = A \circ \operatorname{Ad}(g^{-1}) $, where $g\in \bG(k)$, ${A \in \varphi(\operatorname{Grass}_r(\frakg_k))}$ and $\operatorname{Ad}(g) \colon \frakg_k \to \frakg_k, X \mapsto gXg^{-1}$.
	
	We know that this action is transitive on $V(\frakg_\R)$ and $\mathcal{C}(\frakg_\R)$ by Lemma \ref{lem:basic_facts}(2). As all involved sets are semi-algebraic, we can formulate transitivity as a first-order formula,
	\begin{align*}
		\forall A,B \in \varphi(V(\frakg_k)) \ \exists g \in \bG \colon \forall v \in \frakg_k \colon 
		A(g^{-1}vg) = Bv. 
	\end{align*}
	and conclude that $\bG(\K)$ acts transitively on $V(\frakg_\K)$ and $\mathcal{C}(\frakg_\K)$, concluding the proof of (2).
	
	Finally, for $\theta \colon \frakg_k \to \frakg_k, X \mapsto -X\tran$ and $A \in \varphi(\mathcal{C}(\frakg_k))$, we note that $v \in \theta \left(\varphi^{-1}( A )\right)$ if and only if $A(\theta(v)) = \theta (v)$, that is $\theta A \theta v = v$. The condition $\theta(\fraka) = \fraka$ therefore corresponds to $A =\theta A\theta$. We know by Lemma \ref{lem:basic_facts}(3) that the first-order formula
	$$
	\exists A \in \varphi(\mathcal{C}(\frakg_k) ) \colon \forall v \in \frakg_k \colon Av = \theta A \theta v 
	$$
	is true for $k=\R$ and conclude that it therefore is also true for $k = \K$, proving (3). 
\end{proof}

\subsection{Split tori of algebraic groups over real closed fields}\label{sec:split_tori}

Let $\K$ and $\F$ be real closed fields with $\K \subseteq \R \cap \F$.
Let $\bG$ be a semi-simple self-adjoint ($g \in \bG$ implies $g\tran \in \bG$) algebraic $\K$-group. Let $\frakg_\K \subseteq \K^{n \times n}$ be the $\K$-points of the Lie algebra. Let $\frakg_\R$ and $\frakg_\F$ be the semialgebraic extensions, then $\frakg_\R$ is also the Lie algebra of $\bG(\R)$.

All subfields $\K$ of $\R$ are dense in $\R$, in the sense that $\overline{\K}=\R$ where $\overline{\K}$ is the closure of $\K$ in the usual topology of $\R$. The following generalization of this fact will be used in the proof of Theorem \ref{thm:split_tori}, the main result of this section.

\begin{lemma}\label{lem:closure_alg}
	Let $A\subseteq \R^n$ be an algebraic set defined over $\K$. Then we have $\overline{A_\K} = A_\R$ in the usual $\R^n$ topology.
\end{lemma}
\begin{proof}
	We first note that the algebraic set $A_\R$ is Zariski-closed and therefore also closed in the usual topology of $\R^n$, $\overline{A_\R} = A_\R$. We know that $A_\K \subseteq A_\R$ and therefore $\overline{A_\K} \subseteq \overline{A_\R} = A_\R$.
	
	On the other hand, let $x \in A_\R$. Since $\overline{\K} = \R$, there are $y_k \in \K^n$ and $\varepsilon_k >0$ with $|y_k - x|< \varepsilon_k$ and $\varepsilon_k \to 0$ as $k \to \infty$. Now we have the following first-order formula
	$$
	\varphi(y, \varepsilon ) \colon \quad \exists z\in A \colon \lVert z-y \rVert < \varepsilon.
	$$
	For every $k\in \N$, the formula $\varphi(y_k,\varepsilon_k)$ is true over $\R$, we can just take $z=x$ for every $k$. By the transfer principle it is also true over $\K$, which means that we have
	$
	z_k \in A_\K \colon \lVert z_k -y_k \rVert < \varepsilon_k.
	$
	We conclude
	$$
	\lVert z_k - x \rVert \leq \lVert z_k - y_k\rVert + \lVert y_k - x\rVert < 2 \varepsilon_k \to 0
	$$
	as $k \to \infty$, i.e. $x \in \overline{A_\K}$.
\end{proof}

By Subsection \ref{sec:ex_split}, a torus may be split with respect to some field, while not being split in a subfield. We now prove that this does not happen as long as all involved fields are real closed.

\begin{theorem}\label{thm:split_tori}
	Any maximal $\K$-split torus $\bS < \bG$ is maximal $\F$-split. Moreover, there is a maximal $\F$-split torus $\bS$ so that $g\tran = g$ for all $g \in \bS$.
\end{theorem}

\begin{proof}  We first find a maximal $\K$-split torus $\bT^s$ with $g\tran=g$ for all $g \in \bT^s$.
	
    Lemma \ref{lem:3}(3) states that there is a maximally $\K$-split Cartan subalgebra $\frakh \in \mathcal{C}(\frakg_\K)$ such that $\theta(\frakh) = \frakh$. We consider the complexification $\frakh_{\C} = \C \otimes_{\K} \frakh = \frakh_\R \oplus i \frakh_\R  \subseteq \frakg_{\C} = \frakg_\R \oplus i \frakg_\R$, which by Lemma \ref{lem:defs_Cartan_subalgebra}(ii) is also a Cartan subalgebra, in the sense that $\frakh_\C$ is abelian and $\operatorname{Nor}_{\frakg_\C}(\frakh_\C) = \frakh_\C$.
    Then
    $$
    \bT(\C) := \operatorname{Nor}_{\bG(\C)}(\frakh_\C)^{\circ} = \{ g \in \bG(\C) \colon \operatorname{Ad}(g)(\frakh_\C) = \frakh_\C \}^{\circ}.
    $$
    forms a Zariski-connected closed subgroup of $\bG(\C)$ with Lie algebra $\operatorname{Nor}_{\frakg_\C}(\frakh_\C) = \frakh_\C$. By \cite[Lemma 18.5]{Bor}, $\bT(\C)$ is defined over $\K$. Since $\frakh_\C$ is finite dimensional, $\bT(\C)$ can be written as the zero-set of finitely many polynomials and is therefore the $\C$-points of an algebraic group $\bT$.

	Since $\frakg_\C$ is semisimple and $\frakh_\C$ is a Cartan subalgebra, $\operatorname{ad}(H)$ is simultaneously diagonalizable for all $H \in \frakh_\C$ Lemma \ref{lem:defs_Cartan_subalgebra}(v). For $H \in \frakh_\C$, $\exp(H) \in T_\C$ and $\operatorname{Ad}(\exp(H)) = \exp(\operatorname{ad}(H))$ is diagonalizable. Possibly not all elements in $\bT(\C)$ are of the form $\exp(H)$ with $H \in \frakh_\C$, but $\exp(\frakh_\C)$ is an open neighborhood of the identity and thus generates $\bT(\C)$. For any $g\in \bT(\C)$ there are $H_1, \ldots , H_n \in \frakh_\C$ such that
	\begin{align*}
		\operatorname{Ad}(g) &= \operatorname{Ad}\left(\prod_{i=1}^n \exp(H_i) \right)= \operatorname{Ad}\left(\exp \left(\sum_{i=1}^n H_i\right)\right) \\
		&= \exp\left( \operatorname{ad}\left( \sum_{i=1}^n H_i  \right) \right)
		=\exp\left( \sum_{i=1}^n \operatorname{ad}\left(  H_i  \right) \right),
	\end{align*}
	where we used that $\frakh_\C$ is abelian. Since $\operatorname{ad}(\frakh_\C)$ is simultaneously diagonalizable, $\operatorname{Ad}(g)$ is diagonalizable. Similarly we have $\operatorname{Ad}(g)\operatorname{Ad}(h) = \operatorname{Ad}(h) \operatorname{Ad}(g)$ for all $g , h\in \bT(\C)$. Therefore $\operatorname{Ad}(\bT(\C))$ is simultaneously diagonalizable, meaning $\bT(\C)$ is diagonalizable by \cite[Proposition 8.4]{Bor}. Since $\bT(\C)$ is connected, $\bT$ is a torus by \cite[Proposition 8.5]{Bor}.
	
	If we restrict to symmetric matrices
	$$
	\bT^s = \{ g \in \bT \colon \sigma (g) = g^{-1} \}
	$$
	we see that $\operatorname{Lie}(\bT^s(\R)) = \operatorname{Lie}(\bT(\R)) \cap \frakp = \frakh_\R \cap \frakp$, where $\frakg_\R = \frakp \oplus \frakk$ is the Cartan decomposition corresponding to the standard Cartan involution $\theta = \operatorname{d}_e\sigma \colon X \mapsto -X\tran$. 
	
	Let $\fraka \in V(\frakg_\K)$ be a maximal abelian $\K$-split subalgebra of $\frakg_\K$ with $\fraka \subseteq \frakh$. By Lemma \ref{lem:3}(4), we have $\frakh_\R \in \mathcal{C}(\frakg_\R)$ and $\fraka_\R \in V(\fraka_\R)$ with $\fraka_\R \subseteq \frakh_\R$. By Proposition \ref{prop:defs_max_noncompact}, $\frakh_\R$ is a maximal $\R$-split Cartan subalgebra and $\fraka_\R = \frakh_\R \cap \frakp$.

	Thus $\operatorname{Lie}(\bT^s(\R)) = \frakh_\R\cap \frakp$ is maximal $\R$-split (or maximally noncompact in the terminology of Section \ref{sec:real_cartan_split}). We conclude that $\bT^s(\R)$ is a maximal $\R$-split torus. 	
	Note that $\bT^s$ is defined over $\K$.

	We now want to show that $\bT^s$ is in fact $\K$-split. 
	We consider the relative root system $\RPhi := \Phi(\bG,\bT^s) $. 
	For any root $\alpha \colon \bT^s \to \bG_m $, $\alpha \in ~\!\! \RPhi$, we know that $\alpha_\R(\bT^s(\K)) \subseteq \K^\times$, since $\alpha$ is defined by the property that $\operatorname{Ad}(g)X =\alpha_\R(g)X$ for $g \in \bT^s(\K), X \in \frakg_\K$. 
	The graph of $\alpha$ is an algebraic set
	$$
	\operatorname{graph}(\alpha) = \{ (x,z) \in  \bT^s \times \bG_m \colon z = \alpha(x) \}  
	$$
	defined over $\R$, our goal is to show that it is actually defined over $\K$. 
	
	Now $(x,z) \in \operatorname{graph}(\alpha)_\R$ if and only if $x \in \bT^s(\R)$ and $z = \alpha_\R(x)$. In view Lemma \ref{lem:closure_alg}, $\bT^s(\R) = \overline{\bT^s(\K)}$, so $x \in \bT^s(\R)$ is equivalent to saying that there is a sequence of $x_n \in \bT^s(\K)$ such that $x_n \to x$ as $n \to \infty$. Since $\alpha_\R$ is continuous,
	we have $\alpha_\R(x_n) \to z$ as $n \to \infty$. We conclude that $(x,z) \in \operatorname{graph}(\alpha)_\R$ if and only if there is a sequence $(x_n, \alpha_\R(x_n)) \in \operatorname{graph}(\alpha)_\K$ with $(x_n, \alpha_\R(x_n)) \to (x,z)$ as $n \to \infty$, i.e. 
	$$
	\overline{\operatorname{graph}(\alpha)_\K} = \operatorname{graph}(\alpha)_\R,
	$$ 
	which means that $\operatorname{graph}(\alpha)_\K$ is dense in $\operatorname{graph}(\alpha)_\R$, in particular Zariski-dense (Zariski-open sets are also open in the usual topology and a set is dense if every open set intersects it). Viewing $\operatorname{graph}(\alpha)$ as an algebraic group, we can use \cite[Proposition 3.1.8]{Zim} to conclude that $\operatorname{graph}(\alpha)$ is defined over $\K$. Hence every $\alpha \in \RPhi$ is defined over $\K$, indeed every multiplicative character is defined over $\K$. By \cite[Corollary 8.2]{Bor}, this means that $\bT^s$ is $\K$-split. We have thus found a maximal $\K$-split torus which satisfies $g\tran = g $ for all $g \in \bT^s$.
	
	Next we will prove that $\operatorname{rank}_\F (\bG) = \operatorname{rank}_\K (\bG)$ as follows.
	\begin{align*}
		\operatorname{rank}_\F (\bG) &\leq r_\F (\frakg_\F) 
		= r_\R (\frakg_\R) = \operatorname{rank}_\K (\bG) \leq \operatorname{rank}_\F (\bG)
	\end{align*}
	We recall that $\operatorname{rank}_\F(\bG)$ is the dimension of a maximal $\F$-split torus and $r_\F(\frakg_\F)$ is the dimension of an element in $V(\frakg_\F)$. Indeed, the Lie algebra of any maximal $\F$-split torus is abelian and $\F$-split, i.e. contained in an element of $V(\frakg_\F)$. The first equality is Lemma \ref{lem:3}(1). The second equality is what we did in this proof: we found the maximal $\K$-split torus $\bT^s$ with $\operatorname{Lie}(\bT^s(\R)) = \frakh_\R\cap\frakp \in V(\frakg_\R)$, with dimension $r_\R(\frakg_\R)$. 
	The last inequality holds because every $\K$-split torus is also $\F$-split.
	
	Let $\bS$ be a maximal $\K$-split torus. Then $\bS$ is also $\F$-split (but possibly not maximal). Let $\bT$ be a maximal $\F$-split torus with $\bS \subseteq \bT$. Then $\dim(\bS) = \operatorname{rank}_\K(\bG) = \operatorname{rank}_\F(\bG) = \dim (\bT)$ and since tori are connected $\bS=\bT$. Thus every maximal $\K$-split torus $\bS$ is also maximal $\F$-split.
\end{proof}

 \begin{corollary}\label{cor:split_tori}
	Let $\bS < \bG$ be a maximal $\K$-split torus. Then the set $\KPhi$ of $\K$-roots of $\bS$ in $\bG$ coincides with $\FPhi$ and hence the set of standard parabolic $\K$-subgroups coincides with the set of standard parabolic $\F$-subgroups. In particular any $\F$-parabolic subgroup is $\bG(\F)$-conjugate to a parabolic $\K$-subgroup.
\end{corollary}

\begin{proof}
	The set of $\K$-roots is defined as $\KPhi := \Phi(\bS,\bG)$, where $\bS$ is a maximal $\K$-split torus of $\bG$ \cite[21.1]{Bor}. Since $\bS$ is also $\F$-split by Theorem \ref{thm:split_tori}, we have $\KPhi := \Phi(\bS,\bG) = \FPhi $. 
	
	Following \cite[21.11]{Bor}, we choose an ordering on $\KPhi$ and fix the minimal parabolic $\K$-subgroup $\bP$ associated to the positive roots in $\KPhi$. Any parabolic $\K$-subgroup containing $\bP$ is called \emph{standard parabolic}. The standard parabolic $\K$-subgroups are in one-to-one correspondence with the subsets $I \subseteq \KDelta$ of the simple roots $\KDelta$ of $\KPhi$ \cite[Proposition 21.12]{Bor}. Since $\FDelta = \KDelta$, the standard parabolic $\K$-groups coincide with the standard parabolic $\F$-groups. By \cite[Proposition 21.12]{Bor}, any parabolic $\F$-group is conjugate to one and only one standard parabolic, by an element in $\bG(\F)$. In particular every $\F$-parabolic subgroup is $\bG(\F)$-conjugate to a $\K$-parabolic subgroup.
\end{proof}

\begin{remark}\label{rem:self-adjoint}
	In Theorem \ref{thm:split_tori} we rely on the fact that $\bG$ is self-adjoint, when we use the standard Cartan-involution $\theta \colon X \mapsto -X\tran$. Showing that a semialgebraic Cartan-involution exists also when $\bG$ is not self-adjoint would help in resolving the following question. 
\end{remark}

\begin{question}
	If $\K$ and $\F$ are real closed fields with $\K \subseteq \R \cap \F$ and $\bG$ is a semisimple algebraic $\K$-group (not necessarily self-adjoint), are all maximal $\K$-split tori maximal $\F$-split? Is there a maximal $\F$-split torus that is invariant under some Cartan-involution $\theta$? 
\end{question}

\section{Semialgebraic groups}\label{sec:decompositions}

If $\K$ is a real closed field, a subgroup of $\operatorname{GL}_n(\K)$ that is at the same time a semialgebraic set (with parameters in $\K$) is called a \emph{linear semialgebraic group defined over $\K$}, or short a \emph{semialgebraic $\K$-group}.
Let $\K$ and $\F$ be real closed fields such that $\K \subseteq \R \cap \F$. Let $\bG$ be a semisimple self-adjoint ($g\in \bG$ implies $g\tran \in \bG$) algebraic $\K$-group. 
The $\K$-points $\bG(\K)$ form a semialgebraic $\K$-group. Moreover, the semialgebraic extension of $\bG(\K)$ to $\F$ coincides with the $\F$-points of $\bG$, $\bG(\F) = \bG(\K)_\F$ and for every semialgebraic $\K$-group $G$, we have $G = G_\K$.

From now on, let $G$ be a group that satisfies 
$$
\bG(\K)^\circ < G < \bG(\K),
$$
where $\bG(\K)^\circ$ is the semialgebraic connected component of the identity of $\bG(\K)$. Since $G$ is a finite union of semialgebraic cosets, $G$ is a semialgebraic $\K$-group. The semialgebraic extensions to $\R$ then satisfy $(\bG(\K)^\circ)_\R = \bG(\R)^\circ < G_\R < \bG(\R) = \bG(\K)_\R$ and $G_\R$ is a Lie group. The algebraic group $\bK:= \bG \cap \operatorname{SO}_n$ similarly defines a semialgebraic subgroup $K:= \bK(\K) \cap G$ and the semialgebraic extension satisfies $K_\R = \bK(\R) \cap G_\R$ and is a Lie group.
\begin{lemma}
	The semialgebraic group $G$ is invariant under trasposition. 
\end{lemma}
\begin{proof}
	Since $\bG$ is invariant under transposition, so is the Lie group $\bG(\R)$ and hence $\bK(\R)$ is maximal compact. In particular, $\bK(\R)$ intersects every connected component (in the Euclidean topology and hence also in the semialgebraic topology), whence $\bG(\R) = \bK(\R) \cdot \bG(\R)^\circ$. Thus
	$$
	G_\R = \bG(\R) \cap G_\R = (\bK(\R) \cap G_\R) \cdot (\bG(\R)^\circ \cap G_\R) = K_\R \cdot \bG(\R)^\circ.
	$$
	By \cite[Theorem 2.4.5]{BCR}, $\bG(\R)^\circ$ is connected (and hence pathconnected) also in the Euclidean topology, hence invariant under transposition. Since also $K_\R$ is invariant under trasposition, $G_\R$ is invariant under transposition. This is a semialgebraic property which is therefore also shared with $G$.
\end{proof}

 We consider the $\K$-points $\frakg_\K \subseteq \K^{n\times n}$ of the Lie algebra $T_eG$.
 Let $\frakg_{\mathbb{R}}$ be the semialgebraic extension of $\frakg_\K$. Then $\frakg_{\R}$ is the Lie algebra of the real Lie group $G_{\mathbb{R}}$. Since $G_{\R}$ is self-adjoint, so is $\frakg_{\R}$. The differential of the Lie group homomorphism $G_{\R} \to G_{\R}, g \mapsto (g^{-1})\tran$ is the standard Cartan-involution $\theta \colon \frakg_{\R} \to \frakg_{\R}, X \mapsto -X^{T}$ and leads to the Cartan decomposition into two eigenspaces $\frakg_{\R} = \frakk \oplus \frakp$ as in Section \ref{sec:killing_involutions_decompositions}. The Lie algebra of $K_{\mathbb{R}}$ is 
$$
\operatorname{Lie}(K_{\R}) = \operatorname{Lie}(G_{\R} \cap \operatorname{SO}_n(\R)) = \operatorname{Lie}(G_{\R}) \cap \operatorname{Lie}(\operatorname{SO}_n(\R)) = \frakg_{\R} \cap \mathfrak{so}_n
 = \frakk.
 $$
Let $\bS$ a maximal $\K$-split torus, which we may assume to be self-adjoint by Theorem \ref{thm:split_tori}. 
Let $A = (S_{\K})^\circ$ be the semialgebraic connected component of $S_{\K}$ containing the identity. Let $A_{\R}$ be the semialgebraic extension of $A$. We denote the Lie algebra of $A_{\R}$ by $\fraka$. 
We note that $\fraka \subseteq \frakp$. By Theorem \ref{thm:split_tori}, $\bS$ is maximal $\K$-split as well as maximal $\R$-split and hence $\fraka$ is maximal abelian in $\frakp$. As in Section \ref{sec:killing_involutions_decompositions}, we can now define a root space decomposition 
$$
\frakg_{\R} = \frakg_0 \oplus \bigoplus_{\alpha \in \Sigma} \frakg_{\alpha}
$$
where $\Sigma = \{\alpha \in \fraka^{\star}\colon \alpha \neq 0, \frakg_{\alpha} \neq 0\}$ is a root system. 
We note that
$$
N_{\R} := \operatorname{Nor}_{K_{\R}}(\fraka) = \left\{ k \in K_{\R} \colon kXk^{-1} \in \fraka \text{ for all } X \in \fraka \right\} 
$$
and
$$
M_{\R} := \operatorname{Cen}_{K_{\R}}(\fraka) = \left\{ k \in K_{\R} \colon kXk^{-1} = X \text{ for all }  X \in \fraka  \right\},
$$
are semialgebraic groups, since it suffices to verify the conditions for $X$ on a basis of $\fraka$.
A choice of ordered basis $\Delta \subseteq \Sigma$ gives a total order on $\Sigma$. We consider the Lie algebra
$$
\frakn = \bigoplus_{\alpha >0} \frakg_{\alpha}
$$
which is nilpotent by Lemma \ref{lem:kan_form}. Thus the exponential map and the logarithm are polynomials and we can define the Lie group
$$
U_{\R} = \{ g \in G_{\R} \colon \log(g) \in \frakn \}
$$
which is the group of $\R$-points of an algebraic $\K$-group $\bU$ defined the same way\footnote{In the theory of Lie groups, this group is often denoted by $N$, while $U$ is more common in the setting of algebraic groups.}. 

\subsection{Examples}

For the algebraic group $\bG = \operatorname{SL}_n$ with maximal $\K$-split torus
$$
\bS = \left\{ \begin{pmatrix}
	\star & & \\ & \ddots & \\ && \star
\end{pmatrix} \in \operatorname{SL}_n \right\}.
$$
we have
\begin{align*}
	K_\F &= \operatorname{SO}_n(\F) \\
	A_{\F} &= \left\{ a = (a_{ij}) \in S_{\F} \colon a_{ii}>0   \right\} \\
	U_{\F} &= \left\{g=(g_{ij}) \in \operatorname{SL}_{n}(\F) \colon  g_{ii}= 1, \, g_{ij} = 0 \text{ for } i>j  \right\} \\
	N_{\F} &= \left\{ \text{ permutation matrices with entries in }\pm 1 \right\} \\
	M_{\F} & = \left\{ a=(a_{ij}) \in S_{\F} \colon a_{ii} \in \{ \pm 1\}    \right\} \\
	B_{\F} &= \left\{ g = (g_{ij}) \in \operatorname{SL}_n(\F) \colon g_{ij} = 0 \text{ for } i > j \right\}.
\end{align*}

\subsection{Compatibility of the root systems}\label{sec:compatibility}

In Chapter \ref{sec:alg_root_system} about algebraic groups, we defined the root system $\KPhi$ relative to a $\K$-split torus $\bS$ consisting of those $\alpha \in \hat{\bS}\setminus \{1\}$ for which 
$$
\frakg_{\alpha}^{(\bS)} = \{X \in \frakg \colon \operatorname{Ad}(s)X = \alpha(s)\cdot X \text{ for all } s \in \bS \}
$$
is non-zero. By Theorem \ref{thm:split_tori}, we have $\KPhi = \RPhi = \FPhi$. In this chapter we  defined a root system $\Sigma$ consisting of those $\alpha \in \fraka^{\star}\setminus \{0\}$ for which
$$
\frakg_\alpha = \{ X \in \frakg_{\R} \colon \operatorname{ad}(H)X = \alpha(H) \cdot X \text{ for all } H \in \fraka \}
$$
is non-empty. In this section we show that all these notions of root systems and the various notions of spherical Weyl groups coincide. Let us first see how to construct an algebraic character in $\KPhi$ from a root in $\Sigma$.

\begin{lemma}\label{lem:char}
	Let $\alpha \in \Sigma$. Then the homomorphism
	\begin{align*}
		\chi_\alpha \colon A_\R & \to \R_{>0} \\
		\exp(H) &\mapsto e^{\alpha(H)}
	\end{align*}
	is the restriction to $A_\R$ of an algebraic character in $\hat{\bS}$. 
\end{lemma}
\begin{proof}
	We choose a basis of $\frakg$ consistent with the root decomposition
	$$
	\frakg = \frakg_0 \oplus \bigoplus_{\alpha \in \Sigma} \frakg_\alpha.
	$$
	In this basis, the matrices in $\operatorname{Ad}(S_\R)$ and $\operatorname{ad}(\fraka)$ are diagonal. In fact any diagonal entry\footnote{Note that $\frakg_\alpha$ may not be one-dimensional, but then all the diagonal entries of $\operatorname{ad}(H)$ corresponding to basis vectors in $\frakg_\alpha$ are equal, so it makes sense to talk about $(\operatorname{ad}(H))_{\alpha \alpha}$.} corresponding to the root $\alpha \in \Sigma$ satisfies
	\begin{align*}
		(	\operatorname{ad}(H))_{\alpha \alpha} & = \alpha(H)
	\end{align*}
	for all $H \in \fraka$. We define a character
	\begin{align*}
		\chi (a) &:= (\operatorname{Ad}(a) )_{\alpha \alpha} \quad \quad \text{for } a \in \bS
	\end{align*}
	which is algebraic by its definition. Let $\chi_\R \colon S_\R \to \R$ be the $\R$-points of $\chi$ and $\chi_\R|_{A_\R}$ its restriction to $A_\R$. We now claim that $\chi_\alpha = \chi_\R|_{A_{\R}}$: when $a = \exp(H) \in A_{\R}$ for $H \in \fraka$, we have 
	\begin{align*}
		\chi_\R (a) &= (\operatorname{Ad}(a))_{\alpha \alpha} 
		= (\operatorname{Ad}(\exp(H)))_{\alpha \alpha} \nonumber \\
		&= (\exp (\operatorname{ad}(H)))_{\alpha \alpha}
		= e^{\operatorname{ad}(H)_{\alpha \alpha}} = e^{\alpha(H)} = \chi_\alpha(a).
	\end{align*} 
\end{proof}
\begin{lemma}\label{lem:oneparam_compatibility}
	Let $\alpha \in \Sigma$. Then there is an $x_\alpha \in \fraka$ such that the homomorphism
	\begin{align*}
		t_\alpha^\R \colon \R_{>0} &\to A_\R \\
		e^s &\mapsto \exp(sx_\alpha)
	\end{align*}
	is the restriction of an algebraic one-parameter group $t_\alpha$ in $X_\star(\bS)$, defined over $\K$. The non-degenerate bilinear form in Proposition \ref{prop:char_cochar} is then given by
	$$
	b( \chi_\alpha, t_\beta ) = \frac{2\langle \alpha, \beta \rangle}{\langle \beta,\beta\rangle} \in \Z.
	$$
\end{lemma}
\begin{proof}
	Recall that the Killing form gives rise to a scalar product $B_\theta$ on $\fraka = \operatorname{Lie}(A_\R)$, inducing an isomorphism $\fraka^\star \cong \fraka, \gamma \mapsto H_\gamma$ with the defining property $B_\theta(H_\gamma,H) = \gamma(H)$ for all $H \in \fraka$. For $\beta \in \Sigma$, the coroot
	$$
	\beta^\vee = \frac{2}{B_{\theta}(H_\beta,H_\beta)}\beta
	$$
	can be used to define
	$$
	x_\beta := H_{\beta^\vee} = \frac{2}{B_{\theta}(H_\beta,H_\beta)}H_\beta \in \fraka.
	$$
	We now define $t_\beta^\R(e^s) = \exp(sx_\beta)$ for $s \in \R$ and note that for every $\alpha \in \Sigma$
	$$
	\chi_\alpha^\R(t_\beta^\R(e^s)) =  e^{s\alpha(x_\beta)} = (e^s)^{2\frac{\langle \alpha,\beta \rangle}{\langle \beta,\beta \rangle }}
	$$
	and we note that this characterization uniquely determines $t_\beta^\R$.
	
	On the algebraic side, the non-degenerate bilinear pairing $b\colon \hat{\bS} \times X_{\star}(\bS) \to \Z $ from Proposition \ref{prop:char_cochar} induces an isomorphism
	\begin{align*}
			X_{\star}(\bS) &\stackrel{\sim}{\to} \operatorname{Hom}_{\mathbb{Z}}(\hat{\bS}, \mathbb{Z}) \\
		t &\mapsto (\chi \mapsto b(\chi,b)) .	
	\end{align*}
    which we can use to construct an algebraic one-parameter group $t_\beta$ associated to $\beta\in \Sigma$ as follows: choosing a basis $\Delta \subseteq \Sigma$, we obtain a $\mathbb{Z}$-basis $\{ \chi_\alpha \colon \alpha \in \Delta\}$ of the lattice $\hat{\bS}$. We specify an element of $\operatorname{Hom}_{\mathbb{Z}}(\hat{\bS}, \mathbb{Z})$ by requiring that $\chi_\delta$ is sent to $2 \langle \delta,\beta \rangle / \langle \beta, \beta\rangle$. Thus we have an algebraic one-parameter group $t_\beta$ such that for all $\chi_\alpha$ for $\alpha \in \Sigma$ we have
    $$
    \chi_\alpha\left(t_\beta\left( x \right)\right) = x^{2\frac{\langle \alpha,\beta \rangle}{\langle \beta,\beta \rangle }}
    $$ 
    for all $x \in \bG_m$. The restriction $t_\beta|_{\R _{>0}}$ then coincides with $t_\beta^{\R}$. Since $\bS$ is $\K$-split, $t_\beta$ is defined over $\K$. 
\end{proof}

\begin{proposition}\label{prop:rootsystem_compatible}
	The root systems $\KPhi$ and $\Sigma$ are isomorphic. The spherical Weyl groups $\KW$ and $N_\R / M_\R$ are isomorphic.
\end{proposition}
\begin{proof}
   We first find a compatible map $\KPhi \to \Sigma$. Let $\chi \in \KPhi$. Since $\bS$ is $\K$-split, $\chi$ is defined over $\K$ \cite[Corollary 8.2]{Bor}. Hence we can consider the Lie group homomorphism $\chi_{\R} \colon S_{\R} \to \R^{\times}$ and its derivative $\operatorname{d}_e\!\chi_{\R} \colon \fraka \to  \R$, which we claim to be an element in $\Sigma$. We first claim that $\operatorname{d}_e\!\chi_{\R}$ is nonzero, since otherwise $\chi_{\R}$ would be locally constant, hence only take finitely many values. But since $\chi_{R}(S_{\R})$ is Zariski dense in $\chi(\bS) = \bG_m$ \cite[Corollary 18.3]{Bor}, this cannot be the case. Next, we claim that 
   $$
   \frakg_{\chi} := \left(\frakg_{\chi}^{(\bS)}\right)_{\R} \subseteq \frakg_{\operatorname{d}_e\!\chi_{\R}},
   $$
   which shows that $\KPhi \to \Sigma$ is well defined. Let $H \in \fraka$, such that $a=\exp(H) \in A_{\R} \subseteq S_{\R}$. Since $\bS$ is $\K$-split, we have 
   $$
   \operatorname{Ad}(\exp(H)) = \operatorname{Id} \cdot \alpha_{\R}(\exp(H))
   $$ 
   as maps $\frakg_{\chi} \to \frakg_{\chi}$. In fact, $\operatorname{Ad}\circ \exp = \operatorname{Id} \cdot (\chi_{\R}\circ\exp)$ as maps $\fraka \to \operatorname{GL}(\frakg_{\chi})$. Taking the derivative at $0 \in \fraka$, we get
   \begin{align*}
\operatorname{ad} &= \operatorname{ad} \circ \operatorname{Id} =\operatorname{d}_e\!\operatorname{Ad} \circ \operatorname{d}_0 \exp =  \operatorname{d}_0 ( \operatorname{Ad} \circ \exp ) \\
&= \operatorname{d}_0 (\operatorname{Id} \cdot (\chi_{\R}\circ\exp)) = \operatorname{Id}\cdot \operatorname{d}_e\!\chi_{\R} \circ \operatorname{d}_0\exp = \operatorname{Id}\cdot  \operatorname{d}_e\!\chi_{\R}
   \end{align*}
   as maps $\fraka \to \mathfrak{gl}(\frakg_{\chi})$. This means that for all $X \in \frakg_{\chi}$ and all $H \in \fraka$, we have $\operatorname{ad}(H)X = \operatorname{d}_e\! \chi_{\R}(H)\cdot X$ and hence $X \in \frakg_{\operatorname{d}_e\!\chi_{\R}}$.
   
   Now, we will show that the map $\alpha \mapsto \chi_\alpha$ in Lemma \ref{lem:char}, $\chi_\alpha$ now viewed as an algebraic character in $\hat{\bS}$, actually sends $\Sigma \to \KPhi$. Since
   $$
   (\chi_\alpha)_\R(\exp(H)) = e^{\alpha(H)},
   $$
   we can easily see that $\chi_\alpha \neq 1$, when $\alpha \neq 0$. Next we will show that
   $$
   \frakg_\alpha \subseteq \left(\frakg_{\chi_\alpha}^{(\bS)} \right)_{\R}.
   $$ 
   For $a=\exp(H) \in A_{\R}$, we see from the above formula that
   \begin{align*}
   	\operatorname{Ad}(a) & = \exp(\operatorname{ad}(H)) = \exp( \operatorname{Id} \cdot \alpha(H) ) 
   	= \operatorname{Id} \cdot e^{\alpha(H)} = \operatorname{Id} \cdot \chi_\alpha(a) 
   \end{align*}
   as functions on $\frakg_\alpha$. Since $A_{\R}$ is Zariski dense in $\bS$, $\operatorname{Ad}(s) X = \chi_\alpha (a) \cdot X$ for all $s \in \bS$ and $X \in \frakg_{\alpha}$.
   
   We note that the two maps between $\KPhi$ and $\Sigma$ are inverses of each other, which follows from the fact that
   $$
   \left(\frakg_{\alpha}^{(\bS)}\right)_{\R} \subseteq \frakg_{\operatorname{d}_e\!\alpha_{\R}} \subseteq 
   \left(\frakg_{\chi_{\operatorname{d}_e\!\alpha_\R}}^{(\bS)}\right)_{\R} \quad \text{ for } \alpha \in \KPhi
   $$
   and 
   $$
   \frakg_\alpha \subseteq \left(\frakg_{\chi_\alpha}^{(\bS)}\right)_{\R} \subseteq \frakg_{\operatorname{d}_e\!(\chi_{\alpha})_{\R}} \quad \text{ for } \alpha \in \Sigma. 
   $$
   We note that since $[\frakg_{\alpha},\frakg_\beta] \subseteq \frakg_{\alpha + \beta}$ the map $\KPhi \to \Sigma$ extends to an isomorphism on their vector spaces. For both $\KPhi$ and $\Sigma$, the Weyl groups $\KW$ and $W_s=N_{\R}/M_{\R}$ are generated by reflections along hyperplanes perpendicular to the roots \cite[Theorem 21.2]{Bor} and \cite[Proposition 7.32]{Kna}. This implies that the scalar products of two roots are preserved under the map $\KPhi \to \Sigma$ and hence the two root systems are isomorphic.
\end{proof}

We can also conclude that the various Weyl groups that can be defined coincide.

\begin{proposition} \label{prop:weylgroups}
	The following definitions of Weyl groups are isomorphic.
	\begin{itemize}
		\item [(i)] The Weyl group $W$ generated by reflections in roots of the root system $\KPhi$.
		\item [(ii)] The Weyl group $\RW = \operatorname{Nor}_{\bG(\F)}(\bS(\F))/\operatorname{Cen}_{\bG(\F)}(\bS(\F))$ from the theory of algebraic groups.
		\item [(iii)] The Weyl group $W_s$ generated by reflections in roots of the root system $\Sigma$.
		\item [(iv)] The Weyl group in the Lie groups setting $W(G_{\R},A_{\R}) = N_{\R}/M_{\R} $.
		\item [(v)] The Weyl group $N_{\F}/M_{\F}$ of the semialgebraic extensions.
	\end{itemize}
\end{proposition}
\begin{proof}
	By Proposition \ref{prop:rootsystem_compatible}, $\KPhi \cong \Sigma$, so (i) and (iii) coincide. 
	By \cite[Theorem 21.2]{Bor}, $W=\RW $, so the notions (i) and (ii) agree. By \cite[Proposition 7.32]{Kna}, $W_s = W(G_{\R},A_{\R})$, so the notions (iii) and (iv) agree. Let $W_s = \{w_1, \ldots , w_{|W_s|}\}$ be a list of the finitely many elements in $W_s$ considered as an abstract group. The first-order formula
	\begin{align*}
		\varphi \colon \quad & \exists n_1 , \ldots , n_{|W_s|} \in N \colon 
		\left( \bigwedge_{i,j=0 \atop i\neq j}^{|W_s|} \neg \ n_j^{-1}n_i \in  M \right)  \wedge \\
		&\left(\forall n \in N  \colon \bigvee_{i=1}^{|W_s|}  n^{-1}n_i \in M  \right) \wedge  
		\left( \bigwedge_{i,j,\ell=0 \atop w_iw_j=w_{\ell}}^{|W_s|} n_\ell^{-1}n_in_j \in M \right)
	\end{align*}
	states that there are $|W_s|$ many elements in $N$ with distinct representatives in $N/M$ and there are only $|W_s|$ many and they satisfy the same group multiplication table as $W_s$. In short it says that $N/M$ is isomorphic to $W_s$. Now, by (iv) the formula $\varphi$ holds over $\R$ and we can apply the transfer principle to get the statement over $\F$, showing that the notion in (v) gives the same Weyl group.
\end{proof}

\subsection{Iwasawa decomposition $G=KAU$}\label{sec:KAU}

We recall the classical Iwasawa-decomposition, which applies to $G_{\R}$.
\begin{theorem}\label{thm:KAU_R}(\cite[Theorem 6.46]{Kna})
	Let $G_\R$ be a semisimple Lie group. Let $\frakg_\R = \frakk \oplus \fraka \oplus \frakn$ be an Iwasawa decomposition of its Lie algebra. Let $A_\R$ and $U_\R$ be the analytic subgroups with Lie algebras $\fraka$ and $\frakn$. Then the multiplication map $K_\R \times A_\R \times U_\R \to G_\R$ is a diffeomorphism. This decomposition is unique.
\end{theorem}

We note that all the groups $G_\R, K_\R, A_\R$ and $U_\R$ are semialgebraic defined over $\K$, hence we can consider $G_\F, K_\F, A_\F$ and $U_\F \subseteq \F^{n\times n}$ for any real closed field $\F$ containing $\K$. We use the transfer principle to deduce the following semialgebraic version of the Iwasawa decomposition. We remark that by taking inverses, we also have a decomposition $G=UAK$ in addition to $G=KAU$, both of which will be called the \emph{Iwasawa decomposition}.
\begin{theorem}[$G=KAU$]\label{thm:KAU}
	For every $g \in G_{\F}$, there are $k\in K_{\F}, a\in A_{\F}, u \in U_{\F}$ such that $g=kau$. This decomposition is unique. 
\end{theorem} 
\begin{proof}
	We write the decomposition as a first-order formula
	\begin{align*}
	\varphi \colon \quad & \left(\forall g \in G \ \exists k \in K , a \in A, u\in U \colon g= kau \right)\ \wedge \\
	&\left(\forall k,k' \in K, a,a'\in A, u,u'\in U \colon kau = k'a'u' \to \left(k=k' \wedge a=a' \wedge u = u'\right) \right)	
	\end{align*}
	which holds over $\R$ by the classical Iwasawa decomposition of Theorem \ref{thm:KAU_R}.
	By the transfer principle, Theorem \ref{thm:logic}, $\varphi$ holds over all real closed fields $\F$. Note that to apply the classical Iwasawa-theorem we use that $K_\R$ is maximal compact, $\operatorname{Lie}(A_\R) = \fraka$ and $\operatorname{Lie}(U_\R)= \frakn$.
\end{proof}

\subsection{Cartan decomposition $G=KAK$}\label{sec:KAK}

We use the Cartan decomposition for real Lie groups to find an analogue statement over $\F$.

\begin{theorem}(\cite[Theorem 7.39]{Kna})\label{thm:KAK_R}
	Every element $g\in G_{\R}$ has a decomposition $g = k_1 a k_2$ with $k_1,k_2 \in K_{\R}$ and $a \in A_{\R}$. In this decomposition, $a$ is uniquely determined up to a conjugation by a member of $W(G_{\R},A_{\R})$.
\end{theorem}
\begin{proof}
	By Proposition \ref{prop:weylgroups},
	$\RW = \operatorname{Nor}_G(S)/\operatorname{Cen}_G(S)$ is isomorphic to $W(G_{\R},A_{\R}) = \operatorname{Nor}_{K_{\R}}(\fraka)/\operatorname{Cen}_{K_{\R}}(\fraka)$. Then this is the statement of \cite[Theorem 7.39]{Kna}. 
\end{proof}

\begin{theorem}[$G=KAK$]\label{thm:KAK}
	Every element $g\in G_\F$ has a decomposition $g = k_1 a k_2$ with $k_1,k_2 \in K_{\F}$ and $a \in A_{\F}$. In this decomposition, $a$ is uniquely determined up to a conjugation by a member of $N_{\F}/M_{\F}$.
\end{theorem}
\begin{proof}
	The existence part of the statement follows from the first-order formula
	$$
	\varphi \colon \quad \forall g \in G\  \exists k_1 ,k_2 \in K\ \exists a \in A \colon g =k_1ak_2,
	$$
	which holds over $\R$ by Theorem \ref{thm:KAK_R} and hence over $\F$ by the transfer principle. For uniqueness of $a$, we consider the first-order logic formula
	\begin{align*}
		\psi \colon \quad & \forall a,a' \in A, k_1,k_1',k_2,k_2' \in K \colon k_1ak_2 = k_1'a'k_2' \  \to \\
		&\left( \exists n \in N \colon a = na'n^{-1} \ \wedge \ \left(\forall n' \in N \colon a=n'a'(n')^{-1} \to n^{-1}n' \in M \right) \right)
	\end{align*}
	which states that $a$ is determined up to a conjugation by a member of $N$ and that this member is unique up to an element of $M$. Over $\R$, $\psi$ holds by Theorem \ref{thm:KAK_R}, hence $\psi$ also holds over $\F$ by the transfer principle, concluding the proof.
\end{proof}

\subsection{Bruhat decomposition $G=BWB$}\label{sec:BWB}

By \cite[page 398]{Kna}, $B_{\R} := M_{\R}A_{\R}U_{\R}$ is a closed subgroup of $G_{\R}$ and we have the following Bruhat decomposition.

\begin{theorem}(\cite[Theorem 7.40]{Kna})\label{thm:BWB_R} Every element $g \in G_{\R}$ can be written as $g=b_1nb_2$ with $b_1,b_2 \in B_{\R}$ and $n \in N_{\R}$. In this decomposition, $n$ is unique up to multiplying by an element in $M_{\R}$. Since the spherical Weyl group is $W_s=N_{\R}/M_{\R}$, we have a disjoint union of double cosets
	$$
	G_{\R} = \bigsqcup_{[n] \in W_s} B_{\R}nB_{\R}.
	$$ 
\end{theorem}

The group $B_{\R}$ is semialgebraic and the Bruhat decomposition can be extended to $G_{\F}$.

\begin{theorem}[$G=BWB$]\label{thm:BWB}
	 Every element $g \in G_{\F}$ can be written as $g=b_1nb_2$ with $b_1,b_2 \in B_{\F}$ and $n \in N_{\F}$. In this decomposition $n$ is unique up to multiplying by an element in $M_{\F}$. We have a disjoint union of double cosets
	$$
	G_{\F} = \bigsqcup_{[n] \in W_s} B_{\F}nB_{\F}.
	$$ 
\end{theorem}
\begin{proof}
	The existence of the decomposition follows directly from the transfer principle and Theorem \ref{thm:BWB_R}. For uniqueness we utilize the first-order formula
	\begin{align*}
		\varphi \colon \quad &  \forall b_1,b_2,b_1',b_2' \in B, n,n' \in N \\
		& b_1nb_2 = b_1'n'b_2' \to n^{-1}n' \in M  ,
	\end{align*}
	which holds over $\R$ by Theorem \ref{thm:BWB_R}. Since $W_s= N_\F/M_\F =N_\R/M_\R$ by Proposition \ref{prop:weylgroups}, $G_\F$ is a disjoint union as described. 
\end{proof}

Let 
$$
\ApF = \left\{ a \in A_\F \colon \chi_\alpha(a) \geq 1 \text{ for all } \alpha \in \Delta \right\},
$$
where $\chi_\alpha$ is the algebraic character associated to $\alpha \in \Delta$. Then we may choose $a \in \ApF$ in the Cartan decomposition $G_\F = K_\F \ApF K_\F$ which is then unique, since the Weyl group acts simply transitively on the set of Weyl chambers.

\subsection{Baker-Campbell-Hausdorff formula}\label{sec:BCH}

Given $X,Y \in \frakg_\R$, the Baker-Campbell-Hausdorff formula gives a formal power series description of $Z = \log(\exp(X)\exp(Y))$ in terms of $X,Y$ and iterated commutators of $X$ and $Y$, see for instance \cite[Proposition V.1]{Jac79} or \cite{Mag54}. The formal power series converges in a neighborhood of the identity \cite[Theorem 3.1 in X.3]{Hoch65}, 
but may not converge everywhere in general. If $X,Y \in \frakn_\R$, then the power series is given by a polynomial and thus converges everywhere \cite{Wei63}. Since only finitely many terms are involved, one can see directly or invoke the transfer principle to obtain the Baker-Campbell-Hausdorff formula for elements in $\frakn_\F := \bigoplus_{\alpha >0} (\frakg_\alpha)_\F$.
\begin{proposition}\label{prop:BCH}
	For every $X,Y \in \frakn_\F$, there is $Z \in \frakn_\F$ such that $\exp(X)\exp(Y) = \exp(Z)$. The element $Z$ is given by a finite sum of iterated commutators, the first terms of which are given by
	$$
	Z = X + Y + \frac{1}{2}[X,Y] + \frac{1}{12}\left(\left[X,\left[X,Y\right]\right]- \left[Y, \left[Y,X\right]\right]\right) - \frac{1}{24} \left[Y,\left[X,\left[X,Y\right]\right]\right] + \ldots
	$$ 
\end{proposition}
There are various variations in the literature. We will make use of the Zassenhaus formula
$$
\exp(X+Y) = \exp(X) \exp(Y) \exp\left(-\frac{1}{2}\left[X,Y\right]\right)\exp \left( \frac{1}{3} [Y,[X,Y]] + \frac{1}{6}[X,[X,Y]] \right) \cdots 
$$
for $X,Y \in \frakn_\F$, which can be obtained from the above by calculating $\exp(-X)\exp(X+Y) = \exp(Z) $ iteratively \cite{Mag54}.

\subsection{The unipotent group $U$}\label{sec:U}

The root spaces $\frakg_\alpha$ in the root decomposition
$$
\frakg_\R = \frakg_0 \oplus \bigoplus_{\alpha \in \Sigma} \frakg_\alpha.
$$
corresponding to positive roots, consist of simultaneously nilpotent elements by Lemma \ref{lem:kan_form}. Therefore the group 
$$
U_\R = \exp(\frakn) = \exp\left( \bigoplus_{\alpha \in \Sigma_{>0}} \frakg_\alpha \right)
$$
is unipotent. In fact $U_\R$ is the $\R$-points of an algebraic group $\bU$ and the semialgebraic extension of the semialgebraic $\K$-group $U = U_\K = \bU(\K)$.

Let $\alpha \in \Sigma$. When $2\alpha \notin \Sigma$, the root space $\frakg_\alpha$ is an ideal. In any case, $\frakg_\alpha \oplus \frakg_{2\alpha}$ is an ideal, where $\frakg_{2\alpha} = 0$ if $2\alpha \notin \Sigma$. Thus, for every $\alpha \in \Sigma$ there is a unipotent subgroup $U_\alpha = \exp(\frakg_\alpha \oplus \frakg_{2\alpha}) < U$, called the \emph{root group} which is also a semialgebraic $\K$-group, since the exponential function is a polynomial on nilpotent elements. Note that $U_{2\alpha} < U_\alpha$ if both exist. 

The following Lemma shows that $A_\F$ normalizes the root groups $(U_\alpha)_\F$.

\begin{lemma} \label{lem:aexpXa}
	Let $\alpha \in \Sigma, X\in \mathfrak{g}_\alpha, X' \in \mathfrak{g}_{2\alpha}$ and $a \in A_\F $. Then 
	$$a\exp\left(X + X'\right)a^{-1} = \exp\left(\chi_\alpha (a) X + \chi_\alpha(a)^2 X'\right),$$
	where $\chi_\alpha \colon A_\F \to \mathbb{F}_{>0}$ is the algebraic character from Lemma \ref{lem:char}.
\end{lemma}
\begin{proof}
	We know	that $\operatorname{Ad}_a(X) = \chi_\alpha(a) X$ and $\operatorname{Ad}_a(X') = \chi_{2\alpha}(a) X' = \chi_\alpha(a)^2 X'$.
	We then use distributivity of matrix multiplication to obtain
	\begin{align*}
		a\exp(X+X')a^{-1} &= a\sum_{n=0}^\infty \frac{(X+X')^n}{n!}a^{-1} 
		=\sum_{n=0}^\infty \frac{1}{n!} a(X+X')^n a^{-1} \\
		& =\sum_{n=0}^\infty \frac{1}{n!} \left(a(X+X')a^{-1}\right)^n 
		= \exp \left(aXa^{-1} + aX'a^{-1}\right) \\
		&= \exp \left(\operatorname{Ad}_a(X) + \operatorname{Ad}_a(X')\right)
		= \exp \left(\chi_\alpha(a)X + \chi_\alpha(a)^2 X'\right)
	\end{align*}
\end{proof}
We consider $\Theta \subseteq \Sigma_{>0}$ \emph{closed under addition}, meaning that for any $\alpha, \beta \in \Theta$, if the sum $\alpha + \beta \in \Sigma$, then $\alpha + \beta \in \Theta$. For any $\Theta \subseteq \Sigma_{>0}$ closed under addition,
$$
\frakg_\Theta := \bigoplus_{\alpha \in \Theta} \frakg_\alpha 
$$
is an ideal, since $\left[ \frakg_\alpha , \frakg_\beta \right] \subseteq \frakg_{\alpha + \beta}$ for any $\alpha, \beta \in \Sigma$, see Proposition \ref{prop:root_decomp}. Hence 
$$
(U_\Theta)_\R :=\exp(\frakg_\Theta)
$$
is the $\R$-extension of a semialgebraic group $U_{\Theta}$, in fact for $\Theta = \Sigma_{>0}$ we recover $U = U_{\Sigma_{>0}}$. As a consequence of the Baker-Campbell-Hausdorff-formula we obtain the following description of $(U_\Theta)_\F$.

\begin{lemma}\label{lem:BCH_consequence}
	Let $\Theta = \left\{ \alpha_1, \ldots , \alpha_k \right\} \subseteq \Sigma_{>0}$ be a subset closed under addition with $\alpha_1 > \ldots > \alpha_k$. Then
	$$	
	(U_{\Theta})_\F := \exp\left( \bigoplus_{\alpha \in \Theta} (\frakg_\alpha)_\F\right)
	= \prod_{i = 1}^k \exp\left(\left(\frakg_{\alpha_i}\right)_\F\right) = \langle u \in U_{\alpha} \colon \alpha \in \Theta \rangle.
	$$
\end{lemma}
\begin{proof}
	We start by proving
	$$
	\exp\left( \bigoplus_{\alpha \in \Theta} (\frakg_\alpha)_\F\right) \subseteq  \prod_{i = 1}^k \exp\left(\left(\frakg_{\alpha_i}\right)_\F\right) 
	$$
    using induction over $|\Theta|$. If $|\Theta| =1$, then the statement is immediate. Now assume that $\Theta = \{\alpha_1\} \cup \Theta'$ with $\alpha_1 >\beta$ for all $\beta \in \Theta'$ and such that
	$$	
	\exp\left( \bigoplus_{\alpha \in \Theta'} (\frakg_\alpha)_\F\right)
	\subseteq \prod_{i = 2}^k \exp\left(\left(\frakg_{\alpha_i}\right)_\F\right) .
	$$
	Let $X_i \in (\frakg_{\alpha_i})_\F$. 
Recall that $[\frakg_\alpha, \frakg_{\beta}] \subseteq \frakg_{\alpha + \beta}$ for all $\alpha, \beta \in \Sigma$, see Proposition \ref{prop:root_decomp}. Thus we can use the Zassenhaus-formula in Section \ref{sec:BCH} about the Baker-Campbell-Hausdorff formula to obtain the finite product
	\begin{align*}
		\exp\left( \sum_{i=1}^k X_i \right) &= \exp(X_1)\exp\left(\sum_{i=2}^k X_i\right) \exp\left( -\frac{1}{2}\left[ X_1, \sum_{i=2}^k X_i  \right] \right)  \ldots 
	\end{align*}
	and then repeatedly apply the original version in Proposition \ref{prop:BCH} to simplify the expression back to
	\begin{align*}
		\exp(X_1)\exp\left(\sum_{i=2}^{k} \tilde{X}_i\right) 
	\end{align*}
	for some new $\tilde{X}_i \in (\frakg_{\alpha_i})_\F$.  Applying the induction hypothesis, we conclude
	\begin{align*}
		\exp\left( \sum_{i=1}^k X_i \right) &\in  \exp(X_1) \prod_{i=2}^k \exp\left(\left(\frakg_{\alpha_i}\right)_\F\right) \subseteq \prod_{i = 1}^k \exp\left(\left(\frakg_{\alpha_i}\right)_\F\right).
	\end{align*}
	Next, we notice that $\prod_{i=1}^k \exp((\frakg_{\alpha_i})_\F) \subseteq \langle u \in (U_{\alpha})_\F \colon \alpha \in \Theta\rangle$. Finally, we prove the inclusion
	$$
 \langle u \in U_{\alpha} \colon \alpha \in \Theta \rangle \subseteq \exp\left( \bigoplus_{\alpha \in \Theta} (\frakg_\alpha)_\F\right)
	$$
	using induction over the word length of an element in $\langle u \in U_{\alpha} \colon \alpha \in \Theta \rangle $. If the word length is $1$, then the statement holds. Now assume that 
	$$
	v =\exp\left( \sum_{i=1}^k X_i\right)
	$$
	for some $X_i \in (\frakg_{\alpha_i})_\F$ and consider $u=\exp(X)$ for some $X\in (\frakg_{\alpha})_\F$ and some $\alpha \in \Theta$. We apply Proposition \ref{prop:BCH} to obtain
	\begin{align*}
		uv &= \exp(X) \exp\left(\sum_{i=1}^kX_i\right) \\
		&= \exp\left( X + \sum_{i=1}^k X_i + \frac{1}{2}\left[X, \sum_{i=1}^k X_i\right] + \ldots \right) \in \exp\left( \bigoplus_{\alpha \in \Theta} (\frakg_\alpha)_\F\right)
	\end{align*}
	concluding the proof.
\end{proof}
We point out that the order of the product expression in Lemma \ref{lem:BCH_consequence} starts with $\exp ((\frakg_{\alpha})_\F)$ corresponding to the largest $\alpha = \alpha_1$ followed in decreasing order. Writing elements of $U_{\Theta}$ as the inverses of elements in $U_\Theta$, also gives an expression starting with the smallest root, followed by an increasing order. The following technical Lemma is useful in applications.

\begin{lemma}\label{lem:BCH_normalizer}
	Let $\Theta \subseteq \Sigma_{>0}$ be a subset closed under addition and $\alpha \in \Sigma_{>0}$ such that $\alpha > \beta$ for all $\beta \in \Theta$. Then $uU_{\Theta}u^{-1}= U_{\Theta}$ for all $u \in \exp((\frakg_{\alpha})_\F)$.
	
	For every subset $\Psi \subseteq \Sigma_{>0}$, elements $u \in U_{\Theta}$ can be expressed as $u = u' u''$ with
	$$
	u' \in \prod_{\alpha \in \Theta \cap \Psi} \exp((\frakg_{\alpha})_\F) \quad \text{ and } \quad u'' \in \prod_{\alpha \in \Theta \setminus \Psi} \exp((\frakg_{\alpha})_\F) .
	$$ 
\end{lemma}
\begin{proof}
	Let $X \in (\frakg_{\alpha})_\F$ and $\exp(Y) \in U_{\Theta}$. Then we can apply the Baker-Campbell-Hausdorff formula, Proposition \ref{prop:BCH}, twice to obtain
	\begin{align*}
		\exp(X) \exp(Y) \exp(X)^{-1} &= \exp\left(X + Y + \frac{1}{2}[X,Y] + \ldots\right) \exp(-X)  \\
		&= \exp\left( Y + \frac{1}{2} [X,Y] + \ldots\right) \in U_\Theta.
	\end{align*}
	We show the second statement using induction over the size of $\Theta$. If $|\Theta| = 1$, the statement is clear. Now consider the subset closed under addition $\Theta' := \Theta \setminus \{\alpha\} $ where $\alpha$ is the largest element of $\Theta$. For $u \in U_\Theta$, we use Lemma \ref{lem:BCH_consequence} to obtain $u = u_\alpha \bar{u}$ with $\bar{u} \in U_{\Theta'}$. If $\alpha \notin \Psi$, then the first part of this Lemma can be used to instead write $u = \bar{u} u_\alpha$. Either way, the induction assumption gives 
	$$
		\bar{u}' \in \prod_{\beta \in \Theta' \cap \Psi} \exp((\frakg_{\beta})_\F) \quad \text{ and } \quad \bar{u}'' \in \prod_{\beta \in \Theta' \setminus \Psi} \exp((\frakg_{\beta})_\F)
	$$
	such that $\bar{u} = \bar{u}' \bar{u}''$. Then $u = u_{\alpha} \bar{u}' \bar{u}''$ or $u =  \bar{u}' \bar{u}'' u_\alpha $ as required.
\end{proof}

Lemmas \ref{lem:aexpXa} and \ref{lem:BCH_consequence} can be used to prove that $A_\F$ normalizes all of $U_\F$. 

\begin{proposition} \label{prop:anainN}
	Let $\Theta \subseteq \Sigma_{>0}$ closed under addition. Then $A_\F$ normalizes $(U_\Theta)_\F$: for all $a \in A_\F, u \in (U_\Theta)_\F \colon aua^{-1} \in (U_\Theta)_\F$.
	In particular $aU_\F a^{-1} = U_\F$ for all $a \in A_\F$.
\end{proposition}
\begin{proof}
	Let $a \in A_\F$. By Lemma \ref{lem:BCH_consequence}, any $u\in (U_\Theta)_\F$ can be written as $u = u_{\alpha_1} \cdot \ldots \cdot u_{\alpha_k}$ with $u_{\alpha_i} \in (U_{\alpha_i})_\F$ where $\Theta = \{\alpha_1, \ldots , \alpha_k\}$. By Lemma \ref{lem:aexpXa}, $au_{\alpha_i}a^{-1} \in (U_{\alpha_i})_\F$, so
	$$
	aua^{-1} = au_{\alpha_1}a^{1} \cdot \ldots \cdot au_{\alpha_k}a^{-1} \in \prod_{i=1}^k (U_{\alpha_i})_\F = (U_\Theta)_\F,
	$$
	where we used Lemma \ref{lem:BCH_consequence} again.
\end{proof}

\subsection{Jacobson-Morozov Lemma for algebraic groups}\label{sec:Jacobson_Morozov}

On the level of algebraic groups, the Jacobson-Morozov Lemma seems to be folklore. Since we could not find a detailed treatment in the literature we will give its statement and proof here. The classical Jacobson-Morozov Lemma applies to semisimple Lie algebras over fields of characteristic $0$. The following formulation is more general than what we proved in Lemma \ref{lem:JM_basic}.
\begin{proposition}\label{prop:Jacobson_Morozov_algebra}
	\cite[VIII \S 11.2 Prop. 2]{Bou08_789} Let $\frakg$ be a semisimple Lie algebra and $x\in \frakg$ a non-zero nilpotent element. Then there exist $h,y \in \frakg$ such that $(x,y,h)$ is an $\mathfrak{sl}_2$-triplet, meaning that 
	$$
	[h,x] = 2x, \quad [h,y] = -2y, \quad [x,y] = h
	$$
	and hence the Lie algebra generated by $x,h,y$ is isomorphic to $\mathfrak{sl}_2$.  
\end{proposition} 
 
We will use \cite[Chapter 7]{Bor} and \cite{PR93} to show the following version of the Jacobson-Morozov Lemma for algebraic groups over an algebraically closed field $\D$ of characteristic $0$. Note that any semisimple linear algebraic group with Lie algebra $\mathfrak{sl}_2$ is isomorphic to either $\operatorname{SL}_2$ or $\operatorname{PGL}_2$ \cite[Corollary 32.2]{Hum1}. 

\begin{proposition}\label{prop:Jacobson_Morozov_group}
	Let $g\in \bG$ be any unipotent element in a semisimple linear algebraic group $\bG$ over an algebraically closed field $\D$ of characteristic $0$. Then there is an algebraic subgroup $\operatorname{SL}_g < \bG$ with Lie algebra $\operatorname{Lie}(\operatorname{SL}_g) \cong \mathfrak{sl}_2$ and $g \in \operatorname{SL}_g$. The element $\log(g) \in \operatorname{Lie}(\bG)$ corresponds to 
	$$
	\begin{pmatrix}
		0 & 1 \\ 0 & 0
	\end{pmatrix}\in \mathfrak{sl}_2.
	$$
	Moreover, if $g\in \bG(\F)$ for a field $\F \subseteq \D$, then $\operatorname{SL}_g$ can be assumed to be defined over $\F$. 
\end{proposition} 

\begin{proof}
	Since $g \in \bG$ is unipotent, the nilpotent element $x=\log(g)\in \frakg $ exists. 
	By the Jacobson-Morozov Lemma, Proposition \ref{prop:Jacobson_Morozov_algebra}, there are $y,h \in \frakg$ such that $(x,y,h)$ is an $\mathfrak{sl}_2$-triplet. Let $\mathfrak{sl}_g$ denote the subalgebra of $\frakg$ generated by $x,y$ and $h$. Thus, $x \in \operatorname{sl}_g$ corresponds to
	$$
	\begin{pmatrix}
		0 & 1 \\ 0 & 0
	\end{pmatrix}\in \mathfrak{sl}_2
	$$
	under the isomorphism $\mathfrak{sl}_g \cong \mathfrak{sl}_2$.
	We follow \cite[7.1]{Bor} and define
	$$
	\mathcal{A}(\mathfrak{sl}_g) = \bigcap \left\{ \bH \colon \bH \text{ algebraic subgroup of $\bG$ with } \mathfrak{sl}_g \subseteq \operatorname{Lie}(H) \right\},
	$$
	which is a 
	closed connected algebraic subgroup. Let 
	$$
	\operatorname{SL}_g = [	\mathcal{A}(\mathfrak{sl}_g) , 	\mathcal{A}(\mathfrak{sl}_g) ]
	$$ 
	be its commutator group, which is an algebraic group by \cite[2.3]{Bor} since $\mathcal{A}(\mathfrak{sl}_g)$ is connected. We then use \cite{Bor} to obtain
	\begin{align*}
		\operatorname{Lie}(\operatorname{SL}_g ) &\stackrel{7.8}{=} [ \operatorname{Lie} (\mathcal{A}(\mathfrak{sl}_g)) ,\operatorname{Lie}( \mathcal{A}(\mathfrak{sl}_g) )] \stackrel{7.9}{=} [\mathfrak{sl}_g , \mathfrak{sl}_g] = \mathfrak{sl}_g
	\end{align*}
	where the last equality follows from $\mathfrak{sl}_g \cong \mathfrak{sl}_2$. The map
	\begin{align*}
		\alpha \colon \bG_a & \to \bG \\
		t & \mapsto \exp(tx) = \sum_{n= 0}^\infty \frac{(tx)^n}{n!}
	\end{align*}
	is a polynomial and hence a morphism of algebraic groups. We have $\operatorname{Lie}(\alpha(\bG_a)) = \langle x \rangle \subseteq \mathfrak{sl}_g$ and hence 
	$
	g \in \alpha(\bG_a) \subseteq \operatorname{SL}_g
	$ 
	by \cite[7.1(2)]{Bor}.
	
	If $g \in \bG(\F)$, $\log(g) \in \frakg_\F$, and we may assume $y,h \in \frakg_\F$ as well. Thus $\mathfrak{sl}_g \subseteq \frakg_\F$. Now $\mathcal{A}(\mathfrak{sl}_g)$ is defined over $\F$ as in \cite[2.1(b)]{Bor}. 
	Then by \cite[2.3]{Bor}, $\operatorname{SL}_g$ is defined over $\F$. 
\end{proof}

Let us return to the case of real closed fields $\K \subseteq \F \cap \R$. Recall from Section \ref{sec:U}, that $U_\F$ has subgroups $(U_\alpha)_\F$ consisting of unipotent elements for $\alpha \in \Sigma_{>0}$ defined as the semialgebraic extensions of $U_\alpha = \exp((\frakg_\alpha)_\K \oplus (\frakg_{2\alpha })_\K)$. The following is another variation of the Jacobson-Morozov Lemma that only works when $\Sigma$ is reduced. 
\begin{proposition}\label{prop:Jacobson_Morozov_real_closed}
	Let $\alpha \in \Sigma$ and assume $\frakg_{2\alpha} = 0$. Let $u \in (U_{\alpha})_\F$. Then there are $ X \in (\frakg_\alpha)_\F$ and $t \in \F$ such that $u=\exp(tX)$. Let $(X,Y,H)$ be the $\mathfrak{sl}_2$-triplet of Lemma \ref{lem:JM_basic}. Then there is a group homomorphism $\varphi_\F \colon \operatorname{SL}(2,\F) \to G_\F$ that is a restriction of a morphism of algebraic groups
	$
	\varphi \colon \operatorname{SL}(2,\D)  \to \bG 
	$
	defined over $\F$ such that $\varphi$ has finite kernel and
	\begin{align*}
		\varphi_\F \begin{pmatrix}
			1 & t \\ 0& 1
		\end{pmatrix} =  u = \exp(tX)  \quad \text{and} \quad  \varphi_\F \begin{pmatrix}
		1 & 0 \\ t & 1
		\end{pmatrix} = \exp(tY).
	\end{align*}
	If $\varphi$ is not injective, then $\ker (\varphi) \cong \Z/2\Z$ and $\varphi$ factors through the isomorphism
	$$
	\operatorname{PGL}(2,\D) := \operatorname{SL(2,\D)}/\ker(\varphi)  \xrightarrow{\sim} \varphi(\operatorname{SL}(2,\D))
	$$
	which is also defined over $\F$. Moreover $\varphi(g\tran) = \varphi(g)\tran$, for any $g \in \operatorname{SL}(2,\F)$.
\end{proposition}
\begin{proof}
	Let $u,X,t$ as in the statement. We apply Proposition \ref{prop:Jacobson_Morozov_group} to $\exp(X) \in (U_\alpha)_\F$ to obtain an algebraic group $\operatorname{SL}_{\exp(X)}(\D) < \bG$ defined over $\F$ with Lie algebra $\operatorname{sl}_2(\D)$. By \cite[Corollary 32.3]{Hum1}, the only algebraic groups with Lie algebra $\mathfrak{sl}_2(\D)$ are $\operatorname{SL}(2,\D)$ and $\operatorname{PGL}(2,\D) := \operatorname{SL}(2,\D)/Z(\operatorname{SL(2,\D)})$. In both cases we obtain an algebraic homomorphism $\varphi \colon \operatorname{SL}_2(\D) \to \operatorname{SL}_{\exp(X)}$ with finite kernel.
	
	We note that since $X\in (\frakg_\alpha)_\F$, the $\mathfrak{sl}_2$-triplet is described by $(X,Y,H)$ in Lemma \ref{lem:JM_basic} and we have a Lie algebra isomorphism $\mathfrak{sl}_2(\D) \cong \operatorname{Lie}(\operatorname{SL}_{\exp(X)}(\D))$. We note that the two Lie algebra isomorphisms
	 \begin{align*}
	 \varphi \colon \operatorname{SL}_2(\mathbb{D}) 
	 	\xrightarrow{\sim} & \  \operatorname{SL}_{\exp(X)}(\mathbb{D}) \\
	 	\mathfrak{sl}_2(\mathbb{D}) \cong & \  \mathfrak{sl}_{\exp(X)}(\mathbb{D}) &&\hspace{-14ex} \cong \mathfrak{sl}_2(\mathbb{D}) \\
	 	& \ X &&\hspace{-14ex} \mapsfrom
	 	\begin{pmatrix}
	 		0 & 1 \\ 0 & 0
	 	\end{pmatrix}         
	 \end{align*}
	 may not coincide. Since $\operatorname{sl}_2(\D)$ does not have any outer automorphisms (\cite[Theorem 14.1]{Hum1} and \cite[Proposition D.40]{FuHa}), the two isomorphisms only differ by $\operatorname{Ad}(g)$ for some $g\in \operatorname{SL}(2,\D)$. Up to conjugation we may therefore assume that $\operatorname{d}\!\varphi \colon \mathfrak{sl}_2(\D) \to \operatorname{Lie}(\operatorname{SL}_{\exp(X)}(\D))$ maps 
	 \begin{align*}
	                  \begin{pmatrix}
	                  	0 & 1 \\ 0 & 0
	                  \end{pmatrix}        \mapsto X ,            \quad
	                  	 	\begin{pmatrix}
	                   		1 & 0 \\ 0 & -1
	                  	 	\end{pmatrix}   \mapsto   H   \quad \text{and} \quad
	                  	 	\begin{pmatrix}
	                  	 		0 & 0 \\ 1 & 0
	                  	 	\end{pmatrix}      \mapsto  Y    .
	 \end{align*}
	 Even though the conjugation may be by an element in $\operatorname{SL}_2(\D)$, the explicit description of $\operatorname{d}\!\varphi$ shows that $\varphi$ is still defined over $\F$. The description of $\exp(tX)$ and $\exp(tY)$ in the statement of the Proposition also follows. Finally, note that transposition is $\operatorname{d}\!\varphi$-equivariant and hence also $\varphi$-equivariant.
 \end{proof} 

Since in the $\mathfrak{sl}_2$-triplet, $H \in \fraka_\F$ we obtain multiplicative one-parameter groups
\begin{align*}
	\F_{>0} & \to A_\F \\
\lambda &\mapsto \varphi_\F\begin{pmatrix}
	\lambda & 0 \\ 0 & \lambda^{-1}
\end{pmatrix} 	
\end{align*}
which satisfy the following property that will be useful later.
\begin{lemma}\label{lem:Jacobson_Morozov_oneparam}
	Let $\alpha \in \Sigma$ such that $(\frakg_{2\alpha})_\F = 0$. Let $u\in (U_\alpha)_\F$ and $\varphi \colon \operatorname{SL}(2,\D) \to \bG$ as in Proposition \ref{prop:Jacobson_Morozov_real_closed}. For every $\lambda \in \F_{>0}$
	$$
	\chi_\alpha \left(\varphi_\F \begin{pmatrix}
		\lambda & 0 \\ 0 & \lambda^{-1}
	\end{pmatrix} \right) = \lambda^2.
	$$
\end{lemma}
\begin{proof}
	We note that for all $g \in \operatorname{SL}(2,\F)$, the diagrams 
	$$
	\begin{tikzcd}
		{\mathfrak{sl}(2,\mathbb{F})} \arrow[r, "\operatorname{Ad}(g)"] \arrow[d, "\operatorname{d}\!\varphi_\F"] & {\mathfrak{sl}(2,\mathbb{F})} \arrow[d, "\operatorname{d}\!\varphi_\F"] &   {\operatorname{SL}(2,\mathbb{F})} \arrow[d, "\varphi_\F"] \arrow[r, "c_g"] & {\operatorname{SL}(2,\mathbb{F})} \arrow[d, "\varphi_\F"] \\
		\mathfrak{g}_{\mathbb{F}} \arrow[r, "\operatorname{Ad}(\varphi_\F(g))"]                                   & \mathfrak{g}_{\mathbb{F}}                                              & G_{\mathbb{F}} \arrow[r, "c_{\varphi_\F(g)}"]                              & G_{\mathbb{F}}                                        
	\end{tikzcd}
	$$
	commute. For
	$$
	a = \varphi_\F \begin{pmatrix}
		\lambda & 0 \\ 0 & \lambda^{-1}
	\end{pmatrix} \quad \text{ and } \quad X = \operatorname{d}\! \varphi_\F \begin{pmatrix}
	0 & 1 \\ 0 &0 
	\end{pmatrix} \in (\frakg_{\alpha})_\F
	$$
	we then have
	\begin{align*}
		\operatorname{Ad}\left(a \right) X &= \operatorname{d}\!\varphi_\F \left(\operatorname{Ad}\begin{pmatrix}
		\lambda & 0 \\ 0 & \lambda^{-1}
		\end{pmatrix} \begin{pmatrix}
		0 & 1 \\ 0 & 0
		\end{pmatrix}
		\right) = \operatorname{d}\!\varphi_\F \begin{pmatrix}
			0 & \lambda^{2} \\ 0 & 0
		\end{pmatrix} = \lambda^2 X
	\end{align*}
	and since $\chi_\alpha(a)$ is defined by $\operatorname{Ad}(a)X = \chi_{\alpha}(a)X$ we have $\chi_\alpha(a) = \lambda^{2}$.
\end{proof}

\begin{lemma}\label{lem:Jacobson_Morozov_m}
	Let $\alpha \in \Sigma$ such that $(\mathfrak{g}_{2\alpha})_\F = 0$. Let $u \in (U_\alpha)_\F $ and $X\in (\frakg_\alpha)_\F$, $t\in \F$ and $\varphi \colon \operatorname{SL}(2,\D) \to \bG$ as in Proposition \ref{prop:Jacobson_Morozov_real_closed}, so that $u = \exp(tX)$. The element
	$$
	m(u) := \varphi_\F \begin{pmatrix}
		0 & t \\ -1/t & 0
	\end{pmatrix} \in G_\F
	$$ 
	is contained in $\operatorname{Nor}_{G_\F}(A_\F)$ and 
	$$
	\chi_\alpha(m(u)\cdot a\cdot m(u)^{-1}) = \chi_\alpha(a)^{-1}
	$$ 
	for any $a \in A_\F$.
	
	When $t=1$, even $m(u) \in N_\F = \operatorname{Nor}_{K_\F}(A_\F)$ and 
	$$
	m(u) \cdot (U_\alpha)_\F \cdot  m(u)^{-1} = (U_{-\alpha})_\F.
	$$
\end{lemma}
\begin{proof}
	Let $a\in A_\F$ and
	$$
	a_\perp := \varphi_\F \begin{pmatrix}
		\sqrt{\chi_\alpha(a)} & 0 \\ 0 & \sqrt{\chi_\alpha(a)}^{-1}
	\end{pmatrix},
	$$ 
	then $a_0:= aa_\perp^{-1}$ satisfies
	$$
	\chi_\alpha (a_0) = \chi_\alpha(a) \chi_\alpha(a_\perp)^{-1} = \chi_\alpha(a) \frac{1}{\chi_\alpha(a)} = 1 
	$$
	by Lemma \ref{lem:Jacobson_Morozov_oneparam}. For any $u' \in \exp((\frakg_{\alpha})_\F)$ or $u'\in \exp((\frakg_{-\alpha})_\F)$, $a_0u'=u'a_0$ by Lemma \ref{lem:aexpXa}. We note that
	\begin{align*}
			m(u) &= \varphi_\F \begin{pmatrix}
			0 & t \\ -1/t & 0
		\end{pmatrix} =  \varphi_\F \left( \begin{pmatrix}
			1 & 0 \\ -1/t & 1
		\end{pmatrix}  \begin{pmatrix}
			1 & t \\ 0 & 1
		\end{pmatrix} \begin{pmatrix}
			1 & 0 \\ -1/t & 1
		\end{pmatrix}\right) \\
		& \in \exp((\frakg_{-\alpha})_\F) \cdot u \cdot \exp((\frakg_{-\alpha})_\F)
	\end{align*}
	which implies $a_0m(u) = m(u)a_0$. Now
	\begin{align*}
		m(u) \cdot a \cdot m(u)^{-1} &= m(u) \cdot a_0 \cdot a_\perp \cdot m(u)^{-1} = a_0 \cdot m(u) \cdot  a_\perp \cdot m(u)^{-1} \\
		&= a_0 \cdot \varphi_\F \left(\begin{pmatrix}
			0 & t \\ -1/t & 0
		\end{pmatrix}\begin{pmatrix}
		\sqrt{\chi_\alpha(a)} & 0 \\ 0 & \sqrt{\chi_\alpha(a)}^{-1}
		\end{pmatrix}\begin{pmatrix}
		0 & -t \\ 1/t & 0
		\end{pmatrix}\right) \\
		&= a_0\cdot \varphi_\F \begin{pmatrix}
		\sqrt{\chi_\alpha(a)}^{-1} & 0 \\ 0 & \sqrt{\chi_\alpha(a)}
		\end{pmatrix} = a_0 a_\perp^{-1} \in A_\F
	\end{align*}
	and thus $m(u) = \operatorname{Nor}_{G_\F}(A_\F)$. We see directly that 
	$$
	\chi_\alpha (m(u) \cdot a\cdot m(u)^{-1}) = \chi_\alpha(a_0 \cdot a_\perp^{-1}) = \chi_\alpha(a_0) \cdot \chi_\alpha(a_\perp)^{-1} = \chi_\alpha(a)^{-1}.
	$$
	Now if $t=1$ we can use that $\varphi$ preserves transposition by Proposition \ref{prop:Jacobson_Morozov_real_closed}, to show
	\begin{align*}
		m(u) \cdot m(u)\tran &= \varphi_\F \begin{pmatrix}
			0 & 1 \\ -1 & 0
		\end{pmatrix} \left( \varphi_\F \begin{pmatrix}
			0 & 1 \\ -1 & 0
		\end{pmatrix}\right)\tran \\
		& = \varphi_\F \left( \begin{pmatrix}
			0 & 1 \\ -1 & 0
		\end{pmatrix}   \begin{pmatrix}
			0 & -1 \\ 1 & 0
		\end{pmatrix} \right)= \varphi_\F\begin{pmatrix}
			1 & 0 \\ 0 & 1
		\end{pmatrix} = \operatorname{Id}
	\end{align*}
	and hence $m(u)\in K_\F$, so $m(u) \in N_{\F} = \operatorname{Nor}_{K_\F}(A_\F)$. Thus $m(u)$ is a representative of an element $w=[m(u)]$ of the spherical Weyl group $W_s = N_\F/M_\F$. While $u\in (U_{\alpha})_\F$,
	\begin{align*}
		m(u)\cdot u\cdot m(u)^{-1} &= \varphi_\F \left( \begin{pmatrix}
			0 & 1 \\ -1 & 0
		\end{pmatrix}\begin{pmatrix}
			1 & t \\ 0 & 1
		\end{pmatrix} \begin{pmatrix}
			0 & -1 \\ 1 & 0
		\end{pmatrix} \right) \\
		&= \varphi_\F \begin{pmatrix}
			1 & 0 \\ -t & 1
		\end{pmatrix} \in (U_{-\alpha})_\F,
	\end{align*}
	and thus we have $w(\alpha) = -\alpha$. Then 
$$
m(u) \cdot (U_\alpha)_\F \cdot  m(u)^{-1} = (U_{w(\alpha)})_\F =  (U_{-\alpha})_\F.
$$
\end{proof}

\subsection{Rank 1 subgroups}\label{sec:rank1}
Let $\alpha \in \Sigma$. In this section we investigate the group generated by the algebraic groups $\bU_\alpha$ and $\bU_{-\alpha}$. When $\dim(\bU_\alpha) = 1$, then this group is given by the image of the Jacobson-Morozov-morphism $\varphi \colon \operatorname{SL}(2,\D) \to \bG$ from Proposition \ref{prop:Jacobson_Morozov_real_closed}. In general, when $\dim(\bU_\alpha)$ is not $1$, the group generated is larger than the image of $\varphi$, but is still rank $1$. 

\begin{theorem}\label{thm:levi_group}
	Let $\alpha \in \Sigma$. Then there is a semisimple self-adjoint linear algebraic group $\bL_{\pm \alpha}$ defined over $\K$ such that
	\begin{enumerate}
		\item [(i)] $\operatorname{Lie}(\bL_{\pm \alpha}) = ( \frakg_\alpha \oplus \frakg_{2\alpha} ) \oplus (\frakg_{-\alpha} \oplus \frakg_{-2\alpha} ) \oplus ( [\frakg_\alpha , \frakg_{-\alpha}] + [\frakg_{2\alpha}, \frakg_{-2\alpha}] )$.
		\item [(ii)] $\operatorname{Rank}_\R(\bL_{\pm \alpha}) = \operatorname{Rank}_\F(\bL_{\pm \alpha}) = 1$.
	\end{enumerate}
\end{theorem}
\begin{proof}
    We consider the semisimple $\D$-Lie algebra 
    $$
    \frakl :=  ( \frakg_\alpha \oplus \frakg_{2\alpha} ) \oplus (\frakg_{-\alpha} \oplus \frakg_{-2\alpha} ) \oplus ( [\frakg_\alpha , \frakg_{-\alpha}] + [\frakg_{2\alpha}, \frakg_{-2\alpha}] ).
    $$
    Similar to the proof of Proposition \ref{prop:Jacobson_Morozov_group} we follow Borel \cite[7.1]{Bor} by defining the connected normal algebraic subgroup
    $$
    \mathcal{A}(\frakl) = \bigcap \left\{ \bH \colon \bH<\bG \text{ is an algebraic subgroup } \frakl \subseteq \operatorname{Lie}(\bH) \right\}
    $$
    of $\bG$. We set $\bL_{\pm \alpha} := [\mathcal{A}(\frakl), \mathcal{A}(\frakl)]$. Then (i) follows by
    $$
    \operatorname{Lie}(\bL_{\pm \alpha}) = [\mathcal{A}(\frakl) , \mathcal{A}(\frakl)] = [\frakl , \frakl] = \frakl ,
    $$
    where we used that $\frakl$ is semisimple in the step $[\frakl, \frakl] = \frakl$. The algebraic group $\mathcal{A}(\frakl)$ is defined over $\K$ by \cite[2.1(b)]{Bor}. Thus $\bL_{\pm \alpha}$ is defined over $\K$ and connected by \cite[2.3]{Bor}. Since $\frakl$ is semisimple, so is $\bL_{\pm \alpha}$. Since $\theta(X) = -X\tran$ and $\theta(\frakg_\alpha) = \frakg_{-\alpha}$ and $\theta(\frakg_{2\alpha}) = \frakg_{-2\alpha}$ by Proposition \ref{prop:root_decomp}, $\theta([\frakg_\alpha, \frakg_{-\alpha}]) = [\frakg_\alpha, \frakg_{-\alpha}]$ and $\theta([\frakg_{2\alpha}, \frakg_{-2\alpha}]) = [\frakg_{2\alpha}, \frakg_{-2\alpha}]$, and hence $\frakl$ and $\bL_{\pm \alpha}$ are self-adjoint. 
    
    For (ii), Lemma \ref{lem:levi_algebra} tells us that the real rank of $\frakl_\R$ is $1$ and hence $\operatorname{Rank}_\R(\bL_{\pm \alpha}) = 1$. By the theorem on split tori, Theorem \ref{thm:split_tori}, we then have $\operatorname{Rank}_\R(\bL_{\pm\alpha}) = \operatorname{Rank}_\F(\bL_{\pm \alpha}) = 1$. 
\end{proof}
We now consider the semialgebraic subgroup $L_{\pm \alpha} := \bL_{\pm \alpha} \cap G$ of $G$.
\begin{lemma}\label{lem:levi_fixes_A}
  $(L_{\pm \alpha})_\F \subseteq \operatorname{Cen}_{G_\F}(\left\{ a \in A_\F \colon \chi_\alpha (a) = 1 \right\})$.
\end{lemma}
\begin{proof}
This is a semialgebraic statement, hence it suffices to prove it over the real numbers. Let $a \in A_\R$ with $\chi_\alpha(a) = 1$. We have 
$$
\operatorname{Lie}(L_{\pm \alpha}) = ( \frakg_\alpha \oplus \frakg_{2\alpha} ) \oplus (\frakg_{-\alpha} \oplus \frakg_{-2\alpha} ) \oplus ( [\frakg_\alpha , \frakg_{-\alpha}] \oplus [\frakg_{2\alpha}, \frakg_{-2\alpha}] ).
$$ 
For $X \in (\frakg_\alpha)_\R$ and $X' \in (\frakg_{2\alpha})_\R$, we have by Lemma \ref{lem:aexpXa}
\begin{align*}
	a \exp(X) a^{-1} & = \exp(\chi_\alpha(a)X ) = \exp(X) \\
	a \exp(X') a^{-1}  & = \exp(\chi_\alpha(a)^2 X') = \exp(X')
\end{align*}
and the same argument shows $a \exp(Y) a^{-1} = \exp(Y) $ for $Y \in (\frakg_{-\alpha})_\R \oplus (\frakg_{-2\alpha})_\R$. Since $\frakg_0 \cap \frakl = \mathfrak{z}_{\frakk\cap \frakl} (\fraka \cap \frakl) \oplus (\fraka \cap \frakl)$, also $a \exp(H)a^{-1} =\exp(H)$ for $H \in \frakg_0 \cap \frakl$ is clear. We have shown that 
$$
\operatorname{Ad}(a) \colon \operatorname{Lie}((L_{\pm \alpha})_\R) \to  \operatorname{Lie}((L_{\pm \alpha})_\R) 
$$
is the identity, and conjugation by $a$ is constant on connected components of $(L_{\pm})_\R$. Since $L_{\pm \alpha}$ is connected as an algebraic group and $c_a^{-1}(\operatorname{Id})$ is a closed algebraic set, 
conjugation by $a$ is the identity.
\end{proof}

We can treat $L_{\pm \alpha}$ as an example of the theory we have developed so far, in particular we can apply group decompositions as outlined in the beginning of Section \ref{sec:decompositions}. Since $\bS <\bG$ is a self-adjoint maximal split torus, whose Lie algebra contains $[\frakg_\alpha, \frakg_{-\alpha}] \oplus [\frakg_{2\alpha}, \frakg_{-2\alpha}]$, $\bS_{\pm \alpha}:= \bS \cap \bL_{\pm \alpha}$ is a self-adjoint maximal split torus of $\bL_{\pm \alpha}$. Then $K_{\pm \alpha} := L_{\pm \alpha} \cap \operatorname{SO}_n = L_{\pm \alpha} \cap K$. Considering the $\F$-points, the semialgebraic connected component $(A_{\pm \alpha})_\K$ of $(S_{\pm \alpha})_\K$ containing the identity can be semialgebraically extended to $(A_{\pm \alpha})_\F$. The following is a version on Lemma \ref{lem:Jacobson_Morozov_oneparam}, but now for the rank 1 group $(L_{\pm \alpha})_\F$.
\begin{lemma}
	For every $t \in \F_{>0}$, there is an $a \in (A_{\pm \alpha})_\F$ such that $\chi_\alpha(a) = t$.
\end{lemma}
\begin{proof}
	This is clearly a first-order statement and it therefore suffices to show it for $\R = \F$. Let $X \in (\frakg_\alpha)_\R \setminus 0$. From Lemma \ref{lem:levi_algebra}, we know that $\exp([X, \theta(X)]) \in (A_{\pm \alpha})_\R$. Given $t \in \R_{>0}$, let
	$$
	a := \exp \left( \frac{ \log(t) }{\alpha([X,\theta(X)])}  \cdot [X,\theta(X)] \right).
	$$
	Then 
	$
	\chi_\alpha(a) = e^{\alpha(\log(a))} = e^{\log(t)} = t
	$
	by Lemma \ref{lem:char}.
\end{proof}
The root space decomposition then gives $U_{\pm \alpha}=\exp(\frakg_\alpha \oplus \frakg_{2\alpha}) = U_\alpha$, and the spherical Weyl group $W_{\pm \alpha} = (N_{\pm \alpha})_\F / (M_{\pm \alpha})_\F$ can be defined from \begin{align*}
	(N_{\pm\alpha})_\F &:= \operatorname{Nor}_{(K_{\pm \alpha})_\F}((A_{\pm \alpha})_\F)\\
	(M_{\pm\alpha})_\F &:= \operatorname{Cen}_{(K_{\pm \alpha})_\F}((A_{\pm \alpha})_\F).
\end{align*}
We note that as a consequence of Lemma \ref{lem:levi_fixes_A}, $N_{\pm \alpha} \subseteq N_\F$ and also $\operatorname{Nor}_{L_{\pm \alpha}}(A_{\pm \alpha}) \subseteq \operatorname{Nor}_{K_\F}(A_\F)$.

In the case that $\Sigma$ is reduced, there is an interpretation in terms of the Jacobson-Morozov Lemma. Given $u \in (U_{\pm \alpha})_\F$, we note that the morphism $\varphi_\F \colon \operatorname{SL}(2,\F) \to G_\F$ from Proposition \ref{prop:Jacobson_Morozov_real_closed} takes values in $(L_{\pm \alpha})_\F$. In fact, 
$$
(A_{\pm \alpha})_\F = \varphi_\F\left( \left\{ \begin{pmatrix}
	\lambda & 0 \\ 0 & \lambda^{-1}
\end{pmatrix}   \colon \lambda>0 \right\} \right)
$$
since over $\R$, both of these groups are connected and one-dimensional.
By Lemma \ref{lem:Jacobson_Morozov_m}, the element
$$
m = \varphi_\F\begin{pmatrix}
	0 & 1 \\ -1 & 0
\end{pmatrix} \in (N_{\pm \alpha})_\F
$$
is a representative of the only non-trivial element in the spherical Weyl group $W_{\pm \alpha} = \{ [\operatorname{Id}] , [m] \}$. We can now apply the Bruhat decomposition Theorem \ref{thm:BWB} to $(L_{\pm \alpha})_\F$.

\begin{corollary}\label{cor:levi_Bruhat}
	Let $(B_{\alpha})_\F := (M_{\pm \alpha})_\F (A_{\pm \alpha})_\F (U_\alpha)_\F$. Then 
	$$
	(L_{\pm \alpha})_\F = (B_{\alpha})_\F (N_{\pm \alpha})_\F (B_{\alpha})_\F ,
	$$
	and the element $m \in (N_{\pm \alpha})_\F$ is a representative of a unique element in $W_{\pm \alpha}$, so
	$$
	(L_{\pm \alpha})_\F = (B_{\alpha})_\F \ \sqcup \   (B_{\alpha})_\F \cdot m \cdot (B_{ \alpha})_\F.
	$$
\end{corollary}
The following variation is useful in applications.
\begin{corollary}\label{cor:levi_Bruhat_alternative}
	For
	\begin{align*}
		(B_{\alpha})_\F &:= (M_{\pm \alpha})_\F (A_{\pm \alpha})_\F (U_\alpha)_\F, \\
		(B_{-\alpha})_\F &:= (M_{\pm \alpha})_\F (A_{\pm \alpha})_\F (U_{-\alpha})_\F,
	\end{align*} 
	we obtain the decompositions
	\begin{align*}
		(L_{\pm \alpha})_\F &=(B_{\alpha})_\F \cdot (B_{-\alpha})_\F  \ \sqcup \   m \cdot (B_{ -\alpha})_\F,  \\
		(L_{\pm \alpha})_\F &=(B_{\alpha})_\F \cdot (B_{-\alpha})_\F  \ \sqcup \   (B_{\alpha})_\F \cdot m \cdot (B_{ -\alpha})_\F,  \\
		(L_{\pm \alpha})_\F &= (B_{ \alpha})_\F (N_{\pm \alpha})_\F (B_{- \alpha})_\F 
	\end{align*}
where in the last one, the element in $(N_{\pm \alpha})_\F$ is a representative of a unique element in $W_{\pm \alpha} = \left\{ \operatorname{Id}, [m] \right\}$.
\end{corollary}
\begin{proof}
	We use $m^{-1} \cdot (U_\alpha)_\F\cdot  m = (U_{-\alpha})_\F$ from Lemma \ref{lem:Jacobson_Morozov_m}. By Corollary \ref{cor:levi_Bruhat}, we choose $m^{-1}$ as the representative of $[m^{-1}]=[m] \in W_{\pm \alpha} $ and obtain
	$$
	(L_{\pm \alpha})_\F = (B_{\alpha})_\F \ \sqcup \   (B_{\alpha})_\F \cdot m^{-1} \cdot (B_{ \alpha})_\F,
	$$
	so 
	$$
	(L_{\pm \alpha})_\F = m \cdot (B_{-\alpha})_\F \cdot m^{-1} \ \sqcup \   (B_{\alpha})_\F  \cdot (B_{ -\alpha})_\F \cdot m^{-1}.
	$$
	Multiplying this expression on the right by $m\in (L_{\pm \alpha})_\F$ results in
	$$
	(L_{\pm \alpha})_\F =(B_{\alpha})_\F \cdot (B_{-\alpha})_\F  \ \sqcup \   m \cdot (B_{ -\alpha})_\F, 
	$$ 
	and further multiplying by $(B_\alpha)_\F$ on the left results in the remaining two decompositions.
\end{proof}

\subsection{Kostant convexity}\label{sec:kostant}
Using the Iwasawa decomposition $G_\F=U_\F A_\F K_\F$, Theorem \ref{thm:KAU_R}, we associate to every $g=uak\in G_\R$ its $A$-component $a_\R(g) = a \in A_\R$. The following is Kostant's convexity theorem, which we will generalize to $G_\F$ in this chapter.
\begin{theorem}\cite[Theorem 4.1]{Kos}\label{thm:kostant_R}
	For every $b \in A_\R$, 
	$$
	\left\{ a_\R(kb) \in A_\R \colon k \in K_\R \right\} = \exp \left( \operatorname{conv}(W_s \log(b)) \right),
	$$
	where $\log \colon A_\R \to \fraka$ is the inverse of $\exp$, $W_s$ is the spherical Weyl group acting on $\fraka$ and $\operatorname{conv}(W_s \log(b))$ is the convex hull of the Weyl group orbit of $\log(b)$.
\end{theorem}
The left hand side of the equation in Theorem \ref{thm:kostant_R} is already a semialgebraic set and we will reformulate the right hand side as a semialgebraic set as well. For this, we first analyze the root system $\Sigma \subseteq \fraka^\star$.

Recall from Section \ref{sec:killing_involutions_decompositions}, that we have a scalar product $B_\theta$ on $\fraka$, which can be used to set up an isomorphism $\fraka^\star \to \fraka, \alpha \mapsto H_\alpha$ that satisfies the defining property $\lambda(H) = B_\theta(H_\lambda, H)$ for all $H \in \fraka$. In this section we will denote $B_\theta$ as well as the corresponding scalar product on $\fraka^\star$ with brackets $\langle \cdot , \cdot \rangle$. For $\lambda \in \fraka^\star$ define 
$$
x_\lambda = \frac{2}{\langle H_\lambda, H_\lambda \rangle } H_\lambda  =  \frac{2}{\langle \lambda, \lambda \rangle } H_\lambda
$$
which satisfies the property that for all $\alpha, \beta \in \Sigma$,
$$
\langle H_\alpha, x_\beta \rangle = \frac{2 \langle H_\alpha, H_\beta \rangle }{\langle H_\beta, H_\beta\rangle} \in \Z
$$
since $\Sigma$ is a crystallographic root system, Theorem \ref{thm:sigma_root}. Let $\Delta = \left\{ \delta_1, \ldots , \delta_r \right\}$ be a basis of $\Sigma$ and abbreviate $x_i := x_{\delta_i}$ and $H_i := H_{\delta_i}$. Let 
$$
\frakap = \left\{ H \in \fraka \colon \delta(H) \geq 0 \text{ for all } \delta \in \Delta \right\}. 
$$
On the way to prove Theorem \ref{thm:kostant_R}, Kostant describes $\operatorname{conv}(W_s x)$ using the closed convex cone
$$
\fraka_p := \left\{ x \in \fraka \colon x = \sum_{i=1}^r \R_{\geq 0}x_i \right\}
$$
illustrated in Figure \ref{fig:kostant}.

\begin{figure}[h]
	\centering
	\includegraphics[width=0.6\linewidth]{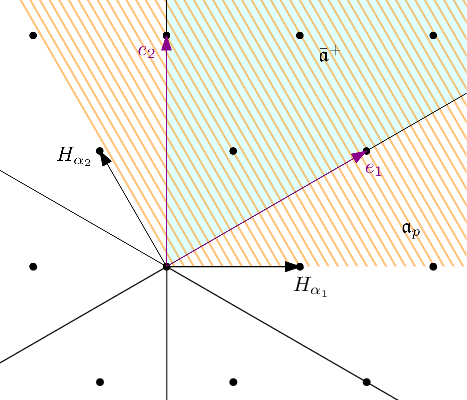}
	\caption{Root system of type $A_2$ associated to $\operatorname{SL}_3$. The convex cone $\fraka_p$ (orange stripes) can be viewed as spanned by the $H_{\alpha_i}$ or as the intersection of the half-spaces defined by the primitive vectors $e_i$.  }
	\label{fig:kostant}
\end{figure}

\begin{lemma}\cite[Lemma 3.3.(2)]{Kos}\label{lem:kostant}
	Let $x,y\in \frakap$. Then 
	\begin{align*}
		y \in  \operatorname{conv}(W_s x) \quad \iff \quad  x-y \in \fraka_p.
	\end{align*}
\end{lemma}
Now we want to describe $\fraka_p$ by inequalities. The cone $\fraka_p$ is an intersection of open half-spaces defined by the half-planes
$$
E_j = \sum_{k \neq j} \R x_k .
$$  
Any vector $x\in \fraka$ orthogonal to $E_j$ satisfies $\langle x,x_i\rangle = 0$ for all $i\neq j$. Writing $x = \sum_k \lambda_{jk}H_k$, we have for all $i \neq j$
$$
\langle x, x_i \rangle = \sum_k \lambda_{jk} \langle H_k, x_i \rangle = 0.
$$
For every $i\neq j$, this is a homogeneous linear equation with variables $\lambda_{j1}, \ldots ,\lambda_{jr}$ and coefficients $\langle H_k,x_i \rangle \in \Z$. Therefore there is a rational (and hence integer) solution for the $\lambda_{jk}$.
This shows that $x\in E_j^\perp \setminus \{0\}$ may be chosen to lie in the lattice $\Gamma = \sum_{l = 1}^r \Z H_k$. There are two primitive vectors in $\Gamma \cap E_j^\perp$. Let $e_j$ be the unique one that is on the same side of $\partial E_j$ as $\fraka_p$. Thus
$$
\fraka_p = \left\{ x \in \fraka \colon \langle e_j, x \rangle \geq 0 \text{ for all } j\right\}.
$$

Under the isomorphism $\fraka^\star \cong \fraka$, the lattice $\Gamma$ corresponds to $L = \sum_{\alpha\in \Delta} \Z \alpha$ and we define $\gamma_j \in L$ to be the element corresponding to $e_j \in \Gamma$. By Lemma \ref{lem:char}, $\gamma_j$ defines an algebraic character $\chi_j  \colon A_\R \to \R $ satisfying
$$
\chi_j (\exp(H)) = e^{\gamma_j(H)}
$$
for all $H \in \fraka$. The multiplicative closed Weyl chamber is
$$
\Ap_\R : = \left\{ a \in A_{\R} \colon \chi_\delta(a) \geq 1 \text{ for all } \delta \in \Delta \right\} = \exp \frakap
$$
which is a semialgebraic set and thus has an $\F$-extension $\Ap_\F$. Using the the Iwasawa decomposition $G_\F=U_\F A_\F K_\F$, Theorem \ref{thm:KAU}, we associate to every $g=uak\in G_\F$ its $A$-component $a_\F(g) = a \in A_\F$ as before. We can now conclude the following version of Kostant's convexity theorem for $G_\F$, illustrated in Figure \ref{fig:kostant7}.

\begin{figure}[h]
	\centering
	\includegraphics[width=0.6\linewidth]{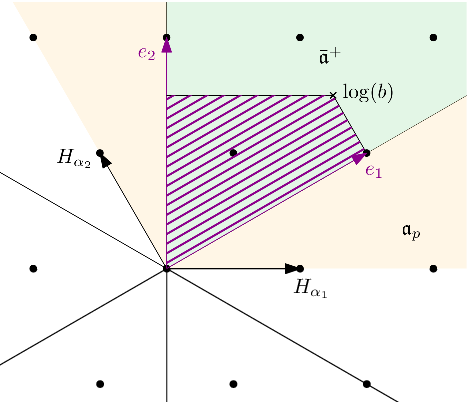}
	\caption{Root system of type $A_2$ associated to $\operatorname{SL}_3$. The convex set in Kostant's convexity Theorem \ref{thm:kostant_F} defined by inequalities is illustrated in purple. }
	\label{fig:kostant7}
\end{figure}

\begin{theorem}\label{thm:kostant_F}
	For all $b \in \ApF$, we have
	$$
	\left\{ a \in \ApF \colon \exists k \in K_\F , a_\F(kb) = a \right\}
	= \left\{ a\in \ApF \colon \chi_i (a) \leq \chi_i (b) \text{ for all } i \right\}.
	$$
\end{theorem}
\begin{proof}
	We first verify that the theorem holds for $\F=\R$. An element $a \in \ApR$ satisfies $a_\R(kb) = a$, for some $k \in K_\R$, if and only if $a \in \ApR \cap \exp(\operatorname{conv}(W_s \log(b)))$ by Theorem \ref{thm:kostant_R}. This means $a = \exp(H)$ for some $H \in \frakap$ and $H\in \operatorname{conv}(W_s\log(b))$. By Lemma \ref{lem:kostant}, this is equivalent to 
	$$
	\log(b) - H \in \fraka_p = \left\{ x \in \fraka \colon \langle e_j, x \rangle \geq 0 \text{ for all } j\right\},
	$$
	and taking exponents this is equivalent to
	$$
	\chi_i(ba^{-1}) \geq 1 
	$$
	for all $i$, so
	$$
	\chi_i(a) \leq \chi_i(b),
	$$
	as in the Theorem. This concludes the case $\F=\R$. Now we note that being part of either set in the statement can be formulated as a first-order formula. Since over $\R$ the two formulae imply each other, they also imply each other over $\F$ by the transfer principle. This concludes the proof. 
\end{proof}

The following Lemma is useful in applications.
\begin{lemma}\label{lem:kostant_gammaalpha}
	For all $\eta \in L = \sum_{\delta \in \Delta} \Z \delta $, we have
	$$
	\eta = \sum_{\ell=1}^r \frac{\langle \eta, \delta_\ell\rangle}{\langle \gamma_\ell,\delta_\ell\rangle} \gamma_\ell
	$$
	and for $\eta^+ := \sum_{\alpha \in \Sigma_{>0}} \alpha$, $\eta^+$ is a positive linear combination of the $\gamma_\ell$,
	$$
	\eta^+ = \sum_{\alpha \in \Sigma_{>0}} \alpha \in \sum_{\ell=1}^r  \Q_{>0} \gamma_\ell.
	$$
\end{lemma}
\begin{proof}
	By definition $\langle x_j, e_\ell \rangle = 0 $ for all $\ell \neq j$, but $\langle x_\ell, e_\ell \rangle \neq 0$, since otherwise $e_\ell = 0$. This implies that $\left\{e_1, \ldots, e_r\right\}$ and hence $\left\{ \gamma_1, \ldots, \gamma_r \right\}$ are linearly independent: for if $x=\sum_\ell \lambda_\ell e_\ell = 0$, then $\langle x_j, x\rangle = \lambda_j \langle x_j,e_j\rangle =0$, so $\lambda_j =0$ for all $j$.
	
	Therefore we can find $n_{i\ell} \in \Q$ such that
	$$
	\delta_i = \sum_{\ell=1 }^r n_{i\ell}\gamma_\ell.
	$$
	Since $\langle \gamma_\ell,\delta_k\rangle = 0 $ when $\ell \neq k$, we have $\langle \delta_i,\delta_k \rangle = n_{ik}\langle \gamma_k, \delta_k \rangle$. Since $x_\ell \in \fraka_p$, $\langle e_\ell,x_\ell\rangle \geq 0$ and since $e_j\neq 0$ and $\langle e_\ell, x_j \rangle = 0$ when $\ell\neq j$, actually $\langle e_\ell, x_\ell\rangle >0$. Thus also $\langle \gamma_\ell,\delta_\ell \rangle >0$, so we can divide $\langle \delta_i , \delta_\ell \rangle = n_{i\ell} \langle \gamma_\ell , \delta_\ell \rangle$ by $\langle \gamma_\ell , \delta_\ell \rangle $ to get
	$$
	n_{i\ell} = \frac{\langle \delta_i ,\delta_\ell \rangle}{\langle \gamma_\ell, \delta_\ell \rangle }.
	$$
	For $\eta \in L = \sum_{\delta \in \Delta} \delta$ we have
	\begin{align*}
		\eta &= \sum_{i=1}^r \lambda_i \delta_i =  \sum_{i=1}^r \sum_{\ell=1}^r \lambda_i \frac{\langle \delta_i ,\delta_\ell \rangle}{\langle \gamma_\ell, \delta_\ell \rangle }  \gamma_\ell \\
		&= \sum_{\ell=1}^r \frac{\langle \sum_{i=1}^r \lambda_i \delta_i, \delta_\ell\rangle}{\langle \gamma_\ell, \delta_\ell \rangle} \gamma_\ell = 
		\sum_{\ell=1}^r \frac{\langle \eta, \delta_\ell \rangle }{\langle \gamma_\ell, \delta_\ell \rangle} \gamma_\ell.
	\end{align*}
	
	For $\delta \in \Delta$, recall that the reflection $\sigma_\delta$ permutes the elements of $\Sigma_{>0} \setminus \{\delta, 2\delta\}$, \cite[VI.1.6 Cor. 1]{Bou08}. Let
	$$
	\eta^+ = \sum_{\alpha \in \Sigma^{+}}\alpha.
	$$
	If $\sigma_\delta (\eta - \delta) = \eta - \delta$, for instance when $\Sigma$ is reduced, then we can use that the reflection $\sigma_\delta$ preserves the scalar product to obtain
	\begin{align*}
		\langle \eta , \delta \rangle &= \langle \sigma_\delta(\eta) , -\delta \rangle = \langle \sigma_\delta\left( \eta - \delta \right) + \sigma_\delta(\delta), -\delta \rangle \\
		&=  \langle (\eta - \delta) - \delta , -\delta \rangle 
		= -\langle \eta, \delta \rangle + 2 \langle \delta , \delta\rangle. 
	\end{align*}
	so $\langle \eta, \delta \rangle = \langle \delta, \delta \rangle >0$. If $\Sigma$ is not reduced and $\delta,2\delta \in \Sigma$, then $\sigma (\eta - 3 \delta) = \eta - 3 \delta$ and 
	\begin{align*}
		\langle \eta , \delta \rangle &= \langle \sigma_\delta(\eta) , -\delta \rangle = \langle \sigma_\delta\left( \eta - 3\delta \right) + \sigma_\delta(3\delta), -\delta \rangle \\
		&=  \langle (\eta - 3\delta) - 3\delta , -\delta \rangle
		= - \langle \eta, \delta \rangle + 6 \langle \delta , \delta \rangle , 
	\end{align*}
	so $\langle \eta, \delta \rangle = 3\langle \delta ,\delta \rangle >0$. This shows that $\eta$ is a positive linear combination of elements $\gamma_\ell$.
\end{proof}

Note that for the slightly simpler $\eta = \sum_{\delta \in \Delta} \delta$, $\eta$ may not be a positive linear combination of elements in $\sum \Z \gamma_\ell$, see Figure \ref{fig:G2_sum} for an example of type $\operatorname{G}_2$.
\\

\begin{figure}[h]
	\centering
	\includegraphics[width=0.95\linewidth]{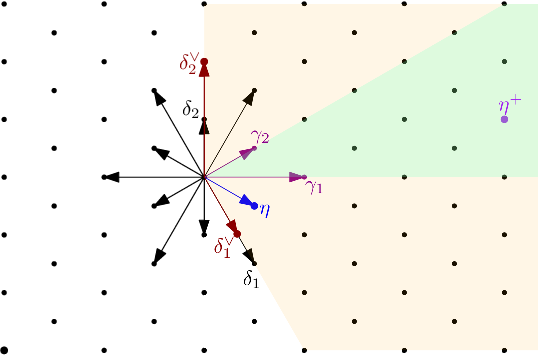}
	\caption{ The root system of type $\operatorname{G}_2$ with basis $\delta_1,\delta_2$, the coroots $\delta_1^\vee, \delta_2^\vee $ corresponding to the coroots $x_1,x_2$ and the elements $\gamma_1, \gamma_2$ spanning the Weyl chamber $\overline{(\mathfrak{a}^\star)}^+$. The element $\eta^+ := \sum_{\alpha >0} \alpha$ from Lemma \ref{lem:kostant_gammaalpha} lies in $\overline{(\mathfrak{a}^\star)}^+$, while the other candidate $\eta := \delta_1 + \delta_2$ does not.}
	\label{fig:G2_sum}
\end{figure}

    \bibliographystyle{alpha_noand} 
    \bibliography{refs} 
 
	
\end{document}